\documentclass{article}

\usepackage{microtype}
\usepackage{graphicx}
\usepackage{subfigure}
\usepackage{booktabs} 

\usepackage{amssymb}
\usepackage{color}

\usepackage{appendix}
\usepackage{amsthm}
\usepackage{amsmath}

\usepackage{amsfonts}
\usepackage{multirow}
\usepackage{color}

\usepackage{algorithm}
\usepackage{algorithmic}

\newtheorem{theorem}{Theorem}
\newtheorem{lemma}{Lemma}

\newtheorem{assumption}{Assumption}
\newtheorem{remark}{Remark}

\usepackage{hyperref}

\usepackage[accepted]{icml2020}

\icmltitlerunning{Accelerated Stochastic Gradient-free and Projection-free Methods}

\begin{document}

\twocolumn[
\icmltitle{Accelerated Stochastic Gradient-free and Projection-free Methods}




\begin{icmlauthorlist}
\icmlauthor{Feihu Huang }{1,2}
\icmlauthor{Lue Tao }{1,2}
\icmlauthor{Songcan Chen }{1,2}
\end{icmlauthorlist}

\icmlaffiliation{1}{College of Computer Science \&
Technology, Nanjing University of Aeronautics and Astronautics, Nanjing 211106, China}
\icmlaffiliation{2}{MIIT Key Laboratory of Pattern Analysis \& Machine Intelligence}

\icmlcorrespondingauthor{Feihu Huang}{huangfeihu@nuaa.edu.cn}

\icmlkeywords{Machine Learning, ICML}

\vskip 0.3in
]



\printAffiliationsAndNotice{}  

\begin{abstract}
In the paper, we propose a class of accelerated stochastic gradient-free and projection-free
(a.k.a., zeroth-order Frank-Wolfe) methods
to solve the constrained stochastic and finite-sum nonconvex optimization.
Specifically, we propose an accelerated stochastic zeroth-order Frank-Wolfe (Acc-SZOFW)
method based on the variance reduced technique of SPIDER/SpiderBoost and a novel momentum accelerated technique.
Moreover, under some mild conditions, we prove that the Acc-SZOFW has the function query complexity of $O(d\sqrt{n}\epsilon^{-2})$ for finding an $\epsilon$-stationary point in the finite-sum problem,
which improves the exiting best result by a factor of $O(\sqrt{n}\epsilon^{-2})$,
and has the function query complexity of $O(d\epsilon^{-3})$ in the stochastic problem, which improves the exiting best result by a factor of $O(\epsilon^{-1})$.
To relax the large batches
required in the Acc-SZOFW, we further propose a novel accelerated stochastic zeroth-order Frank-Wolfe (Acc-SZOFW*) based on
a new variance reduced technique of STORM, which still
reaches the function query complexity of $O(d\epsilon^{-3})$ in the stochastic problem without relying on any large batches.
In particular, we present an accelerated framework of the Frank-Wolfe methods based on the proposed momentum accelerated technique.
The extensive experimental results on black-box adversarial attack and robust black-box classification demonstrate the efficiency of our algorithms.
\end{abstract}
\section{Introduction}
 In the paper, we focus on solving the following constrained stochastic and finite-sum optimization problems
 \begin{align} \label{eq:1}
  \min_{x\in \mathcal{X}} \ f(x)=\left\{
\begin{aligned}
 \mathbb{E}_{\xi} [f(x;\xi)]  & \qquad \mbox{(stochastic)}\\
 \frac{1}{n}\sum_{i=1}^n f_i(x) & \qquad \mbox{(finite-sum)}
 \end{aligned} \right.
 \end{align}
where $f(x): \mathbb{R}^d \rightarrow \mathbb{R}$ is a nonconvex and smooth loss function, 
and the restricted domain $\mathcal{X} \subseteq \mathbb{R}^d$ is supposed to be convex and compact, 
and $\xi$ is a random variable that following an unknown distribution.
When $f(x)$ denotes the expected risk function, the problem \eqref{eq:1} will be seen as a stochastic problem. 
While $f(x)$ denotes the empirical risk function, it will be seen as a finite-sum problem.
In fact, the problem \eqref{eq:1} appears in many machine learning models such as multitask learning, recommendation systems and,
structured prediction \cite{jaggi2013revisiting,lacoste2013block,hazan2016variance}.
For solving the constrained problem \eqref{eq:1}, one common approach is the projected gradient method \cite{iusem2003convergence}
that alternates between optimizing in the unconstrained space and projecting onto the constrained set $\mathcal{X}$.
However, the projection is quite expensive to compute in many constrained sets such as the set of all bounded nuclear norm matrices.
The Frank-Wolfe algorithm (i.e., conditional gradient)\cite{frank1956algorithm,jaggi2013revisiting} is a good candidate for solving the problem \eqref{eq:1}, which only
needs to compute a linear operator instead of projection operator at each iteration.
Following \cite{jaggi2013revisiting}, the linear optimization on $\mathcal{X}$ is much faster than
the projection onto $\mathcal{X}$ in many problems such as the set of all bounded nuclear norm matrices.

\begin{table*}
  \centering
  \caption{ \emph{Function query complexity} comparison of the representative non-convex zeroth-order \emph{Frank-Wolfe} methods for finding
  an $\epsilon$-stationary point of the problem \eqref{eq:1}, \emph{i.e.,} $\mathbb{E}\|\nabla \mathcal {G}(x)\|\leq \epsilon$.
  $T$ denotes the total iterations.
  GauGE, UniGE and CooGE are abbreviations of Gaussian distribution, Uniform distribution and Coordinate-wise smoothing gradient estimators, respectively.
  Note that FW-Black and Acc-ZO-FW are deterministic algorithms, the other are stochastic algorithms.
  Here \textbf{query-size} denotes the \emph{function query size} required in estimating one zeroth-order gradient in these algorithms. Note that these query-sizes are only used in the theoretical analysis. }
  \label{tab:1}
  \resizebox{\textwidth}{!}{
\begin{tabular}{c|c|c|c|c|c}
  \hline
  \textbf{Problem} & \textbf{Algorithm} & \textbf{Reference}  & \textbf{Gradient Estimator} & \textbf{Query Complexity} & \textbf{Query-Size} \\ \hline
  \multirow{3}*{Finite-Sum} & FW-Black & \citet{chen2018frank}   & GauGE or UniGE &  $O(dn\epsilon^{-4})$ & $O(ndT)$ \\ \cline{2-6}
  &Acc-ZO-FW & Ours  & CooGE & {\color{blue}{ $O(dn\epsilon^{-2})$ }} & $O(nd)$ \\ \cline{2-6}
  &Acc-SZOFW & Ours  & CooGE & {\color{blue}{ $O(dn^{\frac{1}{2}}\epsilon^{-2})$ }} & $O(n^{\frac{1}{2}}d)$ \\ \hline
  \multirow{4}*{Stochastic} & ZO-SFW  &\citet{sahu2019towards} & GauGE & $O(d^{\frac{4}{3}}\epsilon^{-4})$ & $O(1)$  \\  \cline{2-6}
  &ZSCG  & \citet{balasubramanian2018zeroth} & GauGE & $O(d\epsilon^{-4})$ & $O(dT)$ \\  \cline{2-6}
  &Acc-SZOFW  & Ours & CooGE & {\color{blue}{ $O(d\epsilon^{-3})$ }} &$O(dT^{1/2})$ \\  \cline{2-6}
  &Acc-SZOFW  & Ours & UniGE & {\color{blue}{ $O(d\epsilon^{-3})$ }} &$O(d^{-1/2}T^{1/2})$ \\  \cline{2-6}
  &Acc-SZOFW*  & Ours & CooGE & {\color{blue}{ $O(d\epsilon^{-3})$ }} &$O(d)$ \\  \cline{2-6}
  &Acc-SZOFW*  & Ours & UniGE & {\color{blue}{ $O(d^{\frac{3}{2}}\epsilon^{-3})$ }} & $O(1)$ \\
  \hline
\end{tabular}
}
\end{table*}

Due to its projection-free property and ability to handle structured constraints,
the Frank-Wolfe algorithm has recently regained popularity in many machine learning applications, and its variants have been widely studied.
For example, several convex variants of Frank-Wolfe algorithm \cite{jaggi2013revisiting,lacoste2015global,lan2016conditional,xu2018frank} have been studied.
In the big data setting, the corresponding online and stochastic Frank-Wolfe algorithms \cite{hazan2012projection,hazan2016variance,hassani2019stochastic,xie2019stochastic} have been developed,
and their convergence rates were studied. The above Frank-Wolfe algorithms were mainly studied in the convex setting.
In fact, the Frank-Wolfe algorithm and its variants are also successful in solving nonconvex problems such as adversarial attacks \cite{chen2018frank}.
Recently, some nonconvex variants of Frank-Wolfe algorithm \cite{lacoste2016convergence,reddi2016stochastic,qu2018non,shen2019complexities,
yurtsever2019conditional,hassani2019stochastic,zhang2019one}
have been developed.

Until now, the above Frank-Wolfe algorithm and its variants need to compute the gradients of objective functions at each iteration.
However, in many complex machine learning problems, the explicit gradients of the objective functions are difficult or infeasible to obtain.
For example, in the reinforcement learning \citep{malik2019derivative,huang2020accelerated}, some complex graphical model inference \citep{wainwright2008graphical} and metric learning \cite{chen2019curvilinear} problems,
it is difficult to compute the explicit gradients of objective functions.
Even worse, in the black-box adversarial attack problems \cite{liu2018zeroth,chen2018frank},
only function values (\emph{i.e.}, prediction labels) are accessible.
Clearly, the above Frank-Wolfe methods will fail in dealing with these problems.
Since it only uses the function values in optimization, the gradient-free (zeroth-order) optimization
method \citep{duchi2015optimal,Nesterov2017RandomGM} is a promising choice to address these problems.
More recently, some zeroth-order Frank-Wolfe methods \cite{balasubramanian2018zeroth,chen2018frank,sahu2019towards} have been proposed and studied.
However, these zeroth-order Frank-Wolfe methods suffer from high function query complexity in solving the problem \eqref{eq:1} (please see Table \ref{tab:1}).

In the paper, thus, we propose a class of accelerated zeroth-order Frank-Wolfe methods to solve the problem \eqref{eq:1}, where
$f(x)$ is possibly black-box. Specifically, we propose an accelerated stochastic zeroth-order Frank-Wolfe (Acc-SZOFW)
method based on the variance reduced technique of SPIDER/SpiderBoost \cite{fang2018spider,wang2018spiderboost} and a novel momentum accelerated technique.
Further, we propose a novel accelerated stochastic zeroth-order Frank-Wolfe (Acc-SZOFW*) to relax the large mini-batch size required in the Acc-SZOFW.
\subsection*{Contributions}
In summary, our main contributions are given as follows:
\vspace*{-6pt}
\begin{itemize}
\setlength\itemsep{0em}
\item[1)] We propose an accelerated stochastic zeroth-order Frank-Wolfe (Acc-SZOFW) method based on the variance reduced technique
          of SPIDER/SpiderBoost and a novel momentum accelerated technique.
\item[2)] Moreover, under some mild conditions, we prove that the Acc-SZOFW has the function query complexity of $O(d\sqrt{n}\epsilon^{-2})$ for finding an $\epsilon$-stationary point in
          the finite-sum problem \eqref{eq:1}, which improves the exiting best result by a factor of $O(\sqrt{n}\epsilon^{-2})$,
          and has the function query complexity of $O(d\epsilon^{-3})$ in the stochastic problem \eqref{eq:1}, which improves the exiting best result by a factor of $O(\epsilon^{-1})$.
\item[3)] We further propose a novel accelerated stochastic zeroth-order Frank-Wolfe (Acc-SZOFW*) to relax the large mini-batch size
          required in the Acc-SZOFW. We prove that the Acc-SZOFW* still
          has the function query complexity of $O(d\epsilon^{-3})$ without relying on the large batches.
\item[4)] In particular, we propose an accelerated framework of the Frank-Wolfe methods based on the proposed momentum accelerated technique.
\end{itemize}
\vspace*{-6pt}
\section{Related Works}
\subsection{Zeroth-Order Methods}
\vspace*{-4pt}
Zeroth-order (gradient-free) methods can be effectively used to solve many machine learning problems,
where the explicit gradient is difficult or infeasible to obtain.
Recently, the zeroth-order methods have been widely studied in machine learning community.
For example, several zeroth-order methods \cite{ghadimi2013stochastic,duchi2015optimal,Nesterov2017RandomGM}
have been proposed by using the Gaussian smoothing technique.
Subsequently, \cite{liu2018zeroth,ji2019improved} recently proposed
accelerated zeroth-order stochastic gradient methods based on the variance reduced techniques.
To deal with nonsmooth optimization, some zeroth-order proximal gradient methods \cite{ghadimi2016mini,Huang2019faster,ji2019improved}
and zeroth-order ADMM-based methods \cite{gao2018information,liu2018admm,huang2019zeroth,huang2019nonconvex} have been proposed.
In addition, more recently, \cite{chen2019zo} has proposed a zeroth-order adaptive momentum method.
To solve the constrained optimization,
the zeroth-order Frank-Wolfe methods \cite{balasubramanian2018zeroth,chen2018frank,sahu2019towards} and the zeroth-order projected gradient methods
\cite{liu2018projected} have been recently proposed
and studied.
\subsection{ Variance-Reduced and Momentum Methods}
To accelerate stochastic gradient descent (SGD) algorithm, various variance-reduced algorithms such as SAG \cite{roux2012stochastic},
SAGA \cite{defazio2014saga}, SVRG \cite{johnson2013accelerating} and SARAH \cite{nguyen2017sarah} have been presented and studied.
Due to the popularity of deep learning, recently the large-scale nonconvex learning problems received wide interest in machine learning community.
Thus, recently many corresponding variance-reduced algorithms to nonconvex SGD have also been proposed and studied, e.g., SVRG \cite{allen2016variance,reddi2016stochastic},
SCSG \cite{lei2017non}, SARAH \cite{nguyen2017stochastic}, SPIDER \cite{fang2018spider}, SpiderBoost \cite{wang2018spiderboost,wang2019spiderboost},
SNVRG \cite{zhou2018stochastic}.

Another effective alternative is to use momentum-based method to accelerate SGD.
Recently, various momentum-based stochastic algorithms for the convex optimization
have been proposed and studied, e.g.,  APCG \cite{lin2014accelerated}, AccProxSVRG \cite{nitanda2014stochastic} and Katyusha \cite{allen2017katyusha}.
At the same time, for the nonconvex optimization,  some momentum-based stochastic algorithms
have been also studied, e.g., RSAG \cite{ghadimi2016accelerated}, Prox-SpiderBoost-M \cite{wang2019spiderboost}, STORM \cite{cutkosky2019momentum}
and Hybrid-SGD \cite{tran2019hybrid}.
\vspace*{-6pt}
\section{Preliminaries}
\subsection{ Zeroth-Order Gradient Estimators }
In this subsection, we introduce two useful zeroth-order gradient estimators, i.e., uniform smoothing gradient estimator (UniGE)
and coordinate smoothing gradient estimator (CooGE) \cite{liu2018zeroth,ji2019improved}.
Given any function $f_i(x): \mathbb{R}^d \rightarrow \mathbb{R}$, the UniGE can generate an approximated gradient as follows:
\begin{align}
 \hat{\nabla}_{uni} f_i(x) = \frac{ d(f_i(x + \beta u)-f_i(x))}{\beta}u,
\end{align}
where $u \in \mathbb{R}^d$ is a vector generated from the uniform distribution over the unit sphere,
and $\beta$ is a smoothing parameter.
While the CooGE can generate an approximated gradient:
\begin{align}
 \hat{\nabla}_{coo} f_i(x) = \sum_{j=1}^d \frac{ f_i(x+\mu_j e_j)-f_i(x-\mu_j e_j)}{2\mu_j}e_j,
\end{align}
where $\mu_j$ is a coordinate-wise smoothing parameter, and $e_j$ is
a basis vector with 1 at its $j$-th coordinate, and 0 otherwise.
Without loss of generality, let $\mu=\mu_1=\cdots=\mu_d$.
\subsection{ Standard Frank-Wolfe Algorithm and Assumptions }
The standard Frank-Wolfe (i.e., conditional gradient) algorithm solves the above problem \eqref{eq:1} by the following iteration:
at $t+1$-th iteration,
\begin{align}
\left\{
\begin{aligned}
& w_{t+1} = \arg\max_{w\in \mathcal{X}} \langle w,-\nabla f(x_{t})\rangle, \\
 &x_{t+1} = (1-\gamma_{t+1})x_{t} + \gamma_{t+1}w_{t+1},
 \end{aligned} \right.
\end{align}
where $\gamma_{t+1} \in (0,1)$ is a step size. For the nonconvex optimization, we apply the following duality gap (i.e., Frank-Wolfe gap \cite{jaggi2013revisiting})
\begin{align}
 \mathcal{G}(x) = \max_{w\in \mathcal{X}} \langle w-x,-\nabla f(x)\rangle,
\end{align}
to give the standard criteria of convergence $\|\mathcal{G}(x)\| \leq \epsilon$ (or $\mathbb{E}\|\mathcal{G}(x)\|\leq \epsilon$)
for finding an $\epsilon$-stationary point, as in \cite{reddi2016stochastic}.

Next, we give some standard assumptions regarding problem \eqref{eq:1} as follows:
\begin{assumption}
Let $f_i(x)=f(x;\xi_i) $, where $\xi_i$ samples from the distribution of random variable $\xi$.
Each loss function $f_i(x)$ is $L$-smooth such that
\vspace*{-4pt}
\begin{align}
& \|\nabla f_i(x)-\nabla f_i(y)\| \leq L \|x - y\|, \ \forall x,y \in \mathcal{X}, \nonumber \\
& f_i(x) \leq f_i(y) + \nabla f_i(y)^T(x-y) + \frac{L}{2}\|x-y\|^2.  \nonumber
\end{align}
\end{assumption}
\vspace*{-4pt}
Let $f_{\beta}(x)=\mathbb{E}_{u\sim U_B}[f(x+\beta u)]$ be a smooth approximation of $f(x)$,
where $U_B$ is the uniform distribution over the $d$-dimensional unit Euclidean ball $B$. Following \cite{ji2019improved},
we have $\mathbb{E}_{(u,\xi)}[\hat{\nabla}_{uni} f_{\xi}(x)]=\nabla f_{\beta}(x)$.
\begin{assumption}
The variance of stochastic (zeroth-order) gradient is bounded, \emph{i.e.,} there exists a constant $\sigma_1 >0$ such that for all $x$,
it follows $\mathbb{E}\|\nabla f_{\xi}(x) - \nabla f(x)\|^2 \leq \sigma_1^2$; There exists a constant $\sigma_2 >0$ such that for all $x$,
it follows $\mathbb{E}\|\hat{\nabla}_{uni} f_{\xi}(x) - \nabla f_{\beta}(x)\|^2 \leq \sigma_2^2$.
\end{assumption}
\vspace*{-4pt}
\begin{assumption}
 The constraint set $\mathcal{X}\subseteq \mathbb{R}^d$ is compact with the diameter: $\max_{x,y\in \mathcal{X}}\|x-y\|\leq D$.
\end{assumption}
\vspace*{-4pt}
\begin{assumption}
 The objective function $f(x)$ is bounded from below in $\mathcal{X}$, \emph{i.e.,}  there exists a non-negative constant $\triangle$, for all $x\in \mathcal{X}$
 such as $f(x) - \inf_{y\in \mathcal{X}} f(y) \leq \triangle$.
\end{assumption}
Assumption 1 imposes the smoothness on each loss function $f_i(x)$ or $f(x,\xi_i)$ , which is commonly used in the convergence analysis
of nonconvex algorithms \cite{ghadimi2016mini}.
Assumption 2 shows that the variance of stochastic or zeroth-order gradient is bounded in norm,
which have been commonly used in the convergence analysis of stochastic zeroth-order algorithms
\cite{gao2018information,ji2019improved}.
Assumptions 3 and 4 are standard for the convergence analysis of Frank-Wolfe algorithms \cite{jaggi2013revisiting,shen2019complexities,yurtsever2019conditional}.
\section{ Accelerated Stochastic Zeroth-Order Frank-Wolfe Algorithms }
In the section, we first propose an accelerated stochastic zeroth-order Frank-Wolfe (Acc-SZOFW) algorithm based on the variance reduced technique
of SPIDER/SpiderBoost and a novel momentum accelerated technique. We then further propose a novel accelerated stochastic zeroth-order Frank-Wolfe (Acc-SZOFW*) algorithm to relax the large mini-batch size required in the Acc-SZOFW.
\subsection{Acc-SZOFW Algorithm}
In the subsection, we propose an Acc-SZOFW algorithm to solve the problem \eqref{eq:1},
where the loss function is possibly black-box.
The Acc-SZOFW algorithm is given in Algorithm \ref{alg:1}.

We first propose an accelerated \textbf{deterministic} zeroth-order Frank-Wolfe (Acc-ZO-FW) algorithm to solve the finite-sum problem \eqref{eq:1}
as a baseline by using the zeroth-order gradient $v_t =\frac{1}{n}\sum_{i=1}^n \hat{\nabla}_{coo} f_i(z_t)$ in Algorithm \ref{alg:1}.
Although \citet{chen2018frank} has proposed a \textbf{deterministic} zeroth-order Frank-Wolfe (FW-Black) algorithm by using the momentum-based accelerated
zeroth-order gradients,
our Acc-ZO-FW algorithm still has lower query complexity than the FW-Black algorithm (see Table \ref{tab:1}).
\begin{algorithm}[tb]
\caption{Acc-SZOFW Algorithm}
\label{alg:1}
\begin{algorithmic}[1] 
\STATE {\bfseries Input:}  Total iteration $T$, step-sizes $\{\eta_t,\gamma_t \in (0,1) \}_{t=0}^{T-1}$,
weighted parameters $\{\alpha_t \in [0,1] \}_{t=0}^{T-1}$, epoch-size $q$, mini-batch size $b$ or $b_1, \ b_2$; \\
\STATE {\bfseries Initialize:} $x_0 = y_0 = z_0 \in \mathcal{X}$; \\
\FOR{$t = 0, 1, \ldots, T-1$}
\IF{$\mod(t,q) = 0$}
\STATE For the \textbf{finite-sum} setting, compute $v_t = \hat{\nabla}_{coo} f(z_t) = \frac{1}{n}\sum_{i=1}^n\hat{\nabla}_{coo} f_{i}(z_t)$;  \\
\STATE For the \textbf{stochastic} setting, randomly select $b_1$ samples $\mathcal{B}_1=\{\xi_1,\cdots,\xi_{b_1}\}$,
        and compute $v_t = \hat{\nabla}_{coo} f_{\mathcal{B}_1}(z_t)$, or
        draw i.i.d. $\{u_1,\cdots,u_{b_1}\}$ from uniform distribution over unit sphere, then compute
        $v_t = \hat{\nabla}_{uni} f_{\mathcal{B}_1}(z_t)$;\\
\ELSE
\STATE For the \textbf{finite-sum} setting, randomly select $b=|\mathcal{B}|$ samples $\mathcal{B} \subseteq \{1,\cdots,n\}$, and
       compute $v_t = \frac{1}{b} \sum_{j \in \mathcal{B}} [\hat{\nabla}_{coo} f_{j}(z_t) - \hat{\nabla}_{coo} f_{j}(z_{t-1})] + v_{t-1}$; \\
\STATE For the \textbf{stochastic} setting, randomly select $b_2$ samples $\mathcal{B}_2=\{\xi_1,\cdots,\xi_{b_2}\} $,
       and compute $v_t = \frac{1}{b_2} \sum_{j \in \mathcal{B}_2} [\hat{\nabla}_{coo} f_{j}(z_t) - \hat{\nabla}_{coo} f_{j}(z_{t-1})] + v_{t-1}$, or  \\
        draw i.i.d. $\{u_1,\cdots,u_{b_2}\}$ from uniform distribution over unit sphere, then compute $v_t = \frac{1}{b_2} \sum_{j \in \mathcal{B}_2} [\hat{\nabla}_{uni} f_{j}(z_t) - \hat{\nabla}_{uni} f_{j}(z_{t-1})] + v_{t-1}$; \\
\ENDIF
\STATE Optimize $w_t = \arg\max_{w \in \mathcal{X}} \langle w, -v_t \rangle$;
\STATE Update $x_{t+1} = x_t + \gamma_t (w_t - x_t)$;
\STATE Update $y_{t+1} = z_t + \eta_t (w_t - z_t)$;
\STATE Update $z_{t+1} = (1 - \alpha_{t+1}) y_{t+1} + \alpha_{t+1} x_{t+1}$;
\ENDFOR
\STATE {\bfseries Output:}  $z_{\zeta}$ chosen uniformly random from $\{z_t\}_{t=1}^{T}$.
\end{algorithmic}
\end{algorithm}

When the sample size $n$ is very large in the finite-sum optimization problem \eqref{eq:1}, we will need to waste lots of time to
obtain the estimated full gradient of $f(x)$, and in turn make the whole algorithm became very slow.
Even worse, for the stochastic optimization problem \eqref{eq:1}, we can never obtain the estimated full gradient of $f(x)$.
As a result, the stochastic optimization method is a good choice. Specifically, we can draw a mini-batch $\mathcal{B} \subseteq \{1,2,\cdots,n\} \ (b=|\mathcal{B}|)$
or $\mathcal{B}= \{\xi_1,\cdots,\xi_b\}$ from the distribution of random variable $\xi$, and can obtain the following stochastic
zeroth-order gradient:
\vspace*{-4pt}
\begin{align}
 \hat{\nabla} f_{\mathcal{B}}(x) = \frac{1}{b} \sum_{j\in \mathcal{B}}\hat{\nabla} f_{j}(x),\nonumber
\end{align}
where $\hat{\nabla} f_{j}(\cdot)$ includes $\hat{\nabla}_{coo} f_{j}(\cdot)$ and $\hat{\nabla}_{uni} f_{j}(\cdot)$.

However, this standard zeroth-order stochastic Frank-Wolfe algorithm suffers from large variance in the zeroth-order stochastic gradient.
Following \cite{balasubramanian2018zeroth,sahu2019towards}, this variance will result in high function query complexity.
Thus, we use the variance reduced technique of SPIDER/SpiderBoost as in \cite{ji2019improved} to reduce the variance in the stochastic gradients. Specifically, in Algorithm \ref{alg:1},  we use the following semi-stochastic gradient for solving the stochastic problem:
\vspace*{-4pt}
\begin{align}
 v_t =\left\{ \begin{aligned}
 &\frac{1}{b_1}  \sum_{i\in \mathcal{B}_1} \hat{\nabla} f_i(x_t), \quad \mbox{if} \mod(t,q)=0 \\
 & \frac{1}{b_2}  \sum_{i\in \mathcal{B}_2} \big( \hat{\nabla} f_i(x_t) - \hat{\nabla} f_i(x_{t-1}) \big)
  + v_{t-1}, \ \mbox{otherwise}
  \end{aligned} \right. \nonumber
\end{align}

Moreover, we propose a novel momentum accelerated framework for the Frank-Wolfe algorithm.
Specifically, we introduce two intermediate variables $x$ and $y$, as in \cite{wang2019spiderboost},
and our algorithm keeps all variables $\{x,y,z\}$ in the constraint set $\mathcal{X}$.
In Algorithm \ref{alg:1}, when set $\alpha_{t+1}=0$ or $\alpha_{t+1}=1$, our algorithm will reduce to the
zeroth-order Frank-Wolfe algorithm with the variance reduced technique of SPIDER/SpiderBoost.
When $\alpha_{t+1} \in (0,1)$, our algorithm will generate the following iterations:
\begin{align}
z_1 & = z_0 + ((1 - \alpha_1) \eta_0 + \alpha_1 \gamma_0) (w_0 - z_0), \nonumber\\
z_2 & = z_1 + ((1 - \alpha_2) \eta_1 + \alpha_2 \gamma_1) (w_1 - z_1) \nonumber \\
&  + \alpha_2 (1 - \gamma_1) (1 - \alpha_1) (\gamma_0 - \eta_0) (w_0 - z_0), \nonumber\\
z_3 & = z_2 + ((1 - \alpha_3) \eta_2 + \alpha_3 \gamma_2) (w_2 - z_2) \nonumber \\
&  + \alpha_3 (1 - \gamma_2) (1 - \alpha_2) (\gamma_1 - \eta_1) (w_1 - z_1) \nonumber \\
&  + \alpha_3 (1 \!-\! \gamma_2) (1 \!-\! \alpha_2) (1 \!-\! \gamma_1) (1 \!-\! \alpha_1) (\gamma_0 \!-\! \eta_0) (w_0 \!-\! z_0), \nonumber \\
& \cdots \nonumber
\end{align}
From the above iterations, the updating parameter $z_t$ is a linear combination of the previous terms $w_i-z_i \ (i \le t)$, which coincides the aim of momentum accelerated technique \cite{nesterov2004introductory,allen2017katyusha}. In fact, our momentum accelerated technique does not rely on the version of gradient $v_t$. In other words, our momentum accelerated technique
can be applied in the zeroth-order, first-order, determinate or
stochastic Frank-Wolfe algorithms.

\begin{algorithm}[tb]
\caption{Acc-SZOFW* Algorithm}
\label{alg:2}
\begin{algorithmic}[1] 
\STATE {\bfseries Input:}  Total iteration $T$, step-sizes $\{\eta_t,\gamma_t \in (0,1)\}_{t=0}^{T-1}$, weighted parameters $\{\alpha_t \in [0,1] \}_{t=1}^{T-1}$
 and the parameter $\{\rho_t\}_{t=1}^{T-1}$; \\
\STATE {\bfseries Initialize:} $x_0 = y_0 = z_0 \in \mathcal{X}$; \\
\FOR{$t = 0, 1, \ldots, T-1$}
\IF{$t = 0$}
\STATE Sample a point $\xi_0$, and compute $v_0 = \hat{\nabla}_{coo} f_{\xi_0}(z_0)$, or draw a vector $u\in \mathbb{R}^d$ from uniform distribution
            over unit sphere, then compute $v_0 = \hat{\nabla}_{uni} f_{\xi_0}(z_0)$;\\
\ELSE
\STATE Sample a point $\xi_t$, and compute $v_{t} = \hat{\nabla}_{coo} f_{\xi_t}(z_t) + (1-\rho_t)\big(v_{t-1} -\hat{\nabla}_{coo} f_{\xi_t}(z_{t-1})\big)$,
       or draw a vector $u\in \mathbb{R}^d$ from uniform distribution
            over unit sphere, then compute
$v_t = \hat{\nabla}_{uni} f_{\xi_t}(z_t) + (1-\rho_t)\big(v_{t-1} -\hat{\nabla}_{uni} f_{\xi_t}(z_{t-1})\big)$; \\
\ENDIF
\STATE Optimize $w_t = \arg\max_{w \in \mathcal{X}} \langle w, -v_t \rangle$;
\STATE Update $x_{t+1} = x_t + \gamma_t (w_t - x_t)$;
\STATE Update $y_{t+1} = z_t + \eta_t (w_t - z_t)$;
\STATE Update $z_{t+1} = (1 - \alpha_{t+1}) y_{t+1} + \alpha_{t+1} x_{t+1}$;
\ENDFOR
\STATE {\bfseries Output:}  $z_{\zeta}$ chosen uniformly random from $\{z_t\}_{t=1}^{T}$.
\end{algorithmic}
\end{algorithm}
\subsection{Acc-SZOFW* Algorithm}
In this subsection, we propose a novel Acc-SZOFW* algorithm
based on a new momentum-based variance reduced technique of STORM/Hybrid-SGD \cite{cutkosky2019momentum,tran2019hybrid}.
Although the above Acc-SZOFW algorithm reaches a lower function query complexity, it requires
large batches (Please see Table \ref{tab:1}).
Clearly, the Acc-SZOFW algorithm can not be well competent to the very large-scale problems and the
data flow problems. Thus, we further propose a novel Acc-SZOFW* algorithm
to relax the large batches required in the Acc-SZOFW.
Algorithm \ref{alg:2} details the Acc-SZOFW* algorithm.

In Algorithm \ref{alg:2}, we apply the variance-reduced technique
of STORM to estimate
the zeroth-order stochastic gradients, and update the parameters $\{x,y,z\}$ as in Algorithm \ref{alg:1}.
Specifically, we use the zeroth-order stochastic gradients as follows:
\vspace*{-4pt}
\begin{align}
  v_t &\!=\! \rho_t\underbrace{\hat{\nabla} f_{\xi_t}(z_t)}_{\mbox{SGD}} \!+\! (1-\rho_t)\big(\underbrace{\hat{\nabla} f_{\xi_t}(z_t) \!-\!\hat{\nabla} f_{\xi_t}(z_{t-1}) \!+\! v_{t-1}}_{\mbox{SPIDER}}\big) \nonumber \\
  & = \hat{\nabla} f_{\xi_t}(z_t) + (1-\rho_t)\big( v_{t-1} - \hat{\nabla} f_{\xi_t}(z_{t-1})\big),
\end{align}
where $\rho_t\in (0,1]$.
Recently, \cite{zhang2019one,xie2019stochastic} have been applied this variance-reduced technique of STORM
to the Frank-Wolfe algorithms.
However, these algorithms strictly rely on the unbiased stochastic gradient.
To the best of our knowledge, we are the first to apply the STORM to the zeroth-order algorithm,
which does not rely on the unbiased stochastic gradient.
\section{Convergence Analysis}
In the section, we study the convergence properties of
both the Acc-SZOFW and Acc-SZOFW* algorithms. \emph{All related proofs are provided in the supplementary document.}
Throughout the paper, $\|\cdot\|$ denotes the vector $\ell_2$ norm and the matrix spectral norm, respectively.
Without loss of generality, let $\alpha_t = \frac{1}{t+1}$,
$\gamma_t = (1+\theta_t)\eta_t$ with $\theta_t = \frac{1}{(t+1)(t+2)}$ in our algorithms.
\subsection{Convergence Properties of Acc-SZOFW Algorithm}
In this subsection, we study the convergence properties of the Acc-SZOFW Algorithm based on the CooGE and UniGE zeroth-order gradients,
respectively.
The detailed proofs are provided in the Appendix \ref{Appendix:A1}.

We first study the convergence properties of the deterministic \textbf{Acc-ZO-FW} algorithm as a baseline,
which is Algorithm \ref{alg:1} using the \textbf{deterministic} zeroth-order gradient $v_t=\frac{1}{n}\sum_{i=1}^n \hat{\nabla}_{coo} f_i(z_t)$ for solving the \textbf{finite-sum} problem \eqref{eq:1}.
\subsubsection{ Deterministic Acc-ZO-FW Algorithm }
\begin{theorem} \label{th:01}
Suppose $\{x_t,y_t,z_t\}^{T-1}_{t=0}$ be generated from Algorithm \ref{alg:1} by using the \textbf{deterministic} zeroth-order gradient $v_t=\frac{1}{n}\sum_{i=1}^n \hat{\nabla}_{coo} f_i(z_t)$,
and let $\alpha_t = \frac{1}{t+1}$, $\theta_t = \frac{1}{(t+1)(t+2)}$,
$\gamma_t = (1+\theta_t)\eta_t$, $\eta = \eta_t= T^{-\frac{1}{2}}$, $\mu=d^{-\frac{1}{2}}T^{-\frac{1}{2}}$,
then we have
\begin{align}
 \mathbb{E}[\mathcal{G}(z_\zeta)] = \frac{1}{T}\sum_{t=1}^{T-1}\mathcal{G}(z_{t}) \leq O(\frac{1}{T^{\frac{1}{2}}}) + O(\frac{\ln(T)}{T^{\frac{3}{2}}}),  \nonumber
\end{align}
where $z_{\zeta}$ is chosen uniformly randomly from $\{z_t\}_{t=0}^{T-1}$.
\end{theorem}
\begin{remark}
 Theorem \ref{th:01} shows that the deterministic Acc-ZO-FW algorithm under the \textbf{CooGE} has $O(T^{-\frac{1}{2}} )$ convergence rate. The Acc-ZO-FW algorithm needs $nd$ samples
 to estimate the zeroth-order gradient $v_t$ at each iteration. For finding an $\epsilon$-stationary point, i.e.,
 $\mathbb{E}[\mathcal{G}(z_\zeta)] \leq \epsilon$, by $T^{-\frac{1}{2}} \leq \epsilon$, we choose $T=\epsilon^{-2}$.
 Thus the \textbf{deterministic} Acc-ZO-FW has the function query complexity of $ndT=O(dn\epsilon^{-2})$.
 Comparing with the existing deterministic zeroth-order Frank-Wolfe algorithm, \emph{i.e.,} FW-Black \cite{chen2018frank},
 our Acc-ZO-FW algorithm has a lower query complexity of $ndT=O(dn\epsilon^{-2})$, which improves the existing result by a factor of $O(\epsilon^{-2})$ (please see Table \ref{tab:1}).
\end{remark}
\subsubsection{ Acc-SZOFW (CooGE) Algorithm }
\begin{lemma}
 Suppose the zeroth-order stochastic gradient $v_t $ be generated from Algorithm \ref{alg:1} by using the \textbf{CooGE} zeroth-order gradient estimator.
 Let $\alpha_t = \frac{1}{t+1}$, $\theta_t = \frac{1}{(t+1)(t+2)}$ and
 $\gamma_t = (1+\theta_t)\eta_t$ in Algorithm \ref{alg:1}.
 For the \textbf{finite-sum} setting, we have
 \begin{align}
  \mathbb{E} \|\nabla f(z_t)-v_t\|  \leq L\sqrt{d}\mu + \frac{L(\sqrt{6d}\mu + 2\sqrt{3D}\eta)}{\sqrt{b/q}}. \nonumber
 \end{align}
 For the \textbf{stochastic} setting, we have
 \begin{align}
  \mathbb{E} \|\nabla f(z_t)-v_t\|  & \leq L\sqrt{d}\mu + \frac{L(\sqrt{6d}\mu + 2\sqrt{3D}\eta)}{\sqrt{b_2/q}} \nonumber \\
   & \quad + \frac{\sqrt{3}\sigma_1}{\sqrt{b_1}} + \sqrt{6d}L\mu.  \nonumber
 \end{align}
\end{lemma}
\begin{theorem} \label{th:1}
Suppose $\{x_t,y_t,z_t\}^{T-1}_{t=0}$ be generated from Algorithm \ref{alg:1} by using the \textbf{CooGE} zeroth-order gradient estimator,
and let $\alpha_t = \frac{1}{t+1}$, $\theta_t = \frac{1}{(t+1)(t+2)}$,
$\gamma_t = (1+\theta_t)\eta_t$, $\eta = \eta_t= T^{-\frac{1}{2}}$, $\mu=d^{-\frac{1}{2}}T^{-\frac{1}{2}}$, $b=q$, or $b_2 = q$ and $b_1 = T$,
then we have
\begin{align}
\mathbb{E}[\mathcal{G}(z_\zeta)] = \frac{1}{T}\sum_{t=1}^{T-1}\mathbb{E} [\mathcal{G}(z_{t})] \leq O(\frac{1}{T^{\frac{1}{2}}}) + O(\frac{\ln(T)}{T^{\frac{3}{2}}}), \nonumber
\end{align}
where $z_{\zeta}$ is chosen uniformly randomly from $\{z_t\}_{t=0}^{T-1}$.
\end{theorem}
\begin{remark}
 Theorem \ref{th:1} shows that the Acc-SZOFW (CooGE) algorithm has convergence rate of $O(T^{-\frac{1}{2}} )$ . When $\mod(t,q)=0$, the Acc-SZOFW algorithm needs $nd$ or $b_1d$ samples
 to estimate the zeroth-order gradient $v_t$ at each iteration and needs $T/q$ iterations, otherwise it needs $2bd$ or $2b_2d$ samples to estimate $v_t$ at each iteration and
 needs $T$ iterations. In the \textbf{finite-sum }setting, by $T^{-\frac{1}{2}} \leq \epsilon$, we choose $T=\epsilon^{-2}$, and let $b=q=\sqrt{n}$, the Acc-SZOFW has the function query complexity of $dnT/q+2dbT=O(d\sqrt{n}\epsilon^{-2})$
 for finding an $\epsilon$-stationary point. In the \textbf{stochastic} setting, let $b_2=q=\epsilon^{-1}$ and $b_1=T=\epsilon^{-2}$,
 the Acc-SZOFW has the function query complexity of $db_1T/q+2db_2T=O(d\epsilon^{-3})$ for finding an $\epsilon$-stationary point.
\end{remark}
\subsubsection{ Acc-SZOFW (UniGE) Algorithm }
\begin{lemma}
 Suppose the zeroth-order stochastic gradient $v_t$ be generated from Algorithm \ref{alg:1} by using the \textbf{UniGE} zeroth-order gradient estimator.
 Let $\alpha_t = \frac{1}{t+1}$, $\theta_t = \frac{1}{(t+1)(t+2)}$ and
 $\gamma_t = (1+\theta_t)\eta_t$ in Algorithm \ref{alg:1}.
 For the \textbf{stochastic} setting, we have
 \begin{align}
  \mathbb{E} \|\nabla f(z_t)-v_t\|  \leq \frac{\beta L d}{2} + \frac{L(\sqrt{3}d\beta + 2\sqrt{6Dd}\eta)}{\sqrt{2b_2/q}} +  \frac{\sigma_2}{\sqrt{b_1}}. \nonumber
 \end{align}
\end{lemma}
\begin{theorem} \label{th:2}
Suppose $\{x_t,y_t,z_t\}^{T-1}_{t=0}$ be generated from Algorithm \ref{alg:1} by using the \textbf{UniGE} zeroth-order gradient estimator,
and let $\alpha_t = \frac{1}{t+1}$, $\theta_t = \frac{1}{(t+1)(t+2)}$,
$\gamma_t = (1+\theta_t)\eta_t$, $\eta = \eta_t= T^{-\frac{1}{2}}$, $\beta=d^{-1}T^{-\frac{1}{2}}$, $b_2 = q$, and $b_1 = T/d$,
then we have
\begin{align}
\mathbb{E}[\mathcal{G}(z_\zeta)] = \frac{1}{T}\sum_{t=1}^{T-1}\mathbb{E} [\mathcal{G}(z_{t})] \leq O(\frac{\sqrt{d}}{T^{\frac{1}{2}}}) + O(\frac{\sqrt{d}\ln(T)}{T^{\frac{3}{2}}}),  \nonumber
\end{align}
\end{theorem}
where $z_{\zeta}$ is chosen uniformly randomly from $\{z_t\}_{t=0}^{T-1}$.
\begin{remark}
 Theorem \ref{th:2} shows that the Acc-SZOFW (UniGE) algorithm has $O(\sqrt{d}T^{-\frac{1}{2}} )$ convergence rate.
 When $\mod(t,q)=0$, the Acc-SZOFW (UniGE) algorithm needs $b_1$ samples
 to estimate the zeroth-order gradient $v_t$ at each iteration and needs $T/q$ iterations, otherwise it needs $2b_2$ samples to estimate $v_t$ at each iteration and
 needs $T$ iterations. By $\sqrt{d}T^{-\frac{1}{2}} \leq \epsilon$, we choose $T=d\epsilon^{-2}$, and let $b_2=q=\epsilon^{-1}$ and $b_1=\epsilon^{-2}$,
 the Acc-SZOFW has the function query complexity of $b_1T/q+2b_2T=O(d\epsilon^{-3})$
 for finding an $\epsilon$-stationary point.
\end{remark}
\subsection{Convergence Properties of Acc-SZOFW* Algorithm}
In this subsection, we study the convergence properties of the Acc-SZOFW* Algorithm based on the CooGE and UniGE, respectively.
The detailed proofs are provided in the Appendix \ref{Appendix:A2}.

\subsubsection{ Acc-SZOFW* (CooGE) Algorithm }
\begin{lemma}
 Suppose the zeroth-order gradient $v_{t} = \hat{\nabla}_{coo} f_{\xi_t}(z_t) + (1-\rho_t)\big(v_{t-1} -\hat{\nabla}_{coo} f_{\xi_t}(z_{t-1})\big)$
 be generated from Algorithm \ref{alg:2}.
 Let $\alpha_t = \frac{1}{t+1}$, $\theta_t = \frac{1}{(t+1)(t+2)}$,
$\gamma_t = (1+\theta_t)\eta_t$, $\eta=\eta_t \leq (t+1)^{-a}$ and $\rho_t = t^{-a}$ for some $a\in (0,1]$ and the smoothing parameter $\mu = \mu_t \leq d^{-\frac{1}{2}}(t+1)^{-a}$,
 then we have
 \begin{align}
  \mathbb{E} \|v_t - \nabla f(z_t)\| \leq L\sqrt{d}\mu + \sqrt{C}(t+1)^{-\frac{a}{2}},
 \end{align}
 where $C=\frac{2(12L^2 D^2 + 12L^2 + 3\sigma^2_1)}{2-2^{-a}-a}$ for some $a\in (0,1]$.
\end{lemma}
\begin{theorem} \label{th:3}
Suppose $\{x_t,y_t,z_t\}^{T-1}_{t=0}$ be generated from Algorithm \ref{alg:2} by using the \textbf{CooGE} zeroth-order gradient estimator.
Let $\alpha_t = \frac{1}{t+1}$, $\theta_t = \frac{1}{(t+1)(t+2)}$,
$\eta=\eta_t = T^{-\frac{2}{3}}$, $\gamma_t = (1+\theta_t)\eta_t$, $\rho_t = t^{-\frac{2}{3}}$ for $t\geq 1$
and $\mu=d^{-\frac{1}{2}}T^{-\frac{2}{3}}$, then we have
\begin{align}
\mathbb{E} [\mathcal{G} (z_{\zeta})] = \frac{1}{T}\sum_{t=1}^{T-1}\mathbb{E} [\mathcal{G}(z_{t})] \leq O(\frac{1}{T^{\frac{1}{3}}}) + O(\frac{\ln(T)}{T^{\frac{4}{3}}}),  \nonumber
\end{align}
where $z_{\zeta}$ is chosen uniformly randomly from $\{z_t\}_{t=0}^{T-1}$.
\end{theorem}
\begin{remark}
 Theorem \ref{th:3} shows that the Acc-SZOFW*(CooGE) algorithm has $O(T^{-\frac{1}{3}} )$ convergence rate.
 It needs $2d$ samples to estimate the zeroth-order gradient $v_t$ at each iteration,
 and needs $T$ iterations. For finding an $\epsilon$-stationary point, \emph{i.e.,} ensuring
 $\mathbb{E}[\mathcal{G}(z_\zeta)] \leq \epsilon$, by $T^{-\frac{1}{3}} \leq \epsilon$,
 we choose $T=\epsilon^{-3}$. Thus the Acc-SZOFW* has the function query complexity of $2dT = O(d\epsilon^{-3})$.
 \textbf{Note that} the Acc-SZOFW* algorithm only requires a small mini-batch size such as $2$
 and reaches the same function query complexity as the Acc-SZOFW algorithm that requires large batch sizes $b_2 = \epsilon^{-1}$ and $b_1 = \epsilon^{-2}$.
 For clarity, we need to emphasize that the \textbf{mini-batch size} denotes the sample size required at each iteration, while the \textbf{query-size} (in Table \ref{tab:1})
 denotes the function query size required in estimating one zeroth-order gradient in these algorithms. In fact, there exists a positive correlation between them.
 For example, in the Acc-SZOFW* algorithm, the mini-batch size is 2, and the corresponding query-size is $2d$.
\end{remark}
\subsubsection{ Acc-SZOFW* (UniGE) Algorithm  }
\begin{lemma}
 Suppose the zeroth-order gradient $v_{t} = \hat{\nabla}_{uni} f_{\xi_t}(z_t) + (1-\rho_t)\big(v_{t-1} -\hat{\nabla}_{uni} f_{\xi_t}(z_{t-1})\big)$ be generated from Algorithm \ref{alg:2}.
 Let $\alpha_t = \frac{1}{t+1}$, $\theta_t = \frac{1}{(t+1)(t+2)}$,
 $\gamma_t = (1+\theta_t)\eta_t$, $\eta=\eta_t \leq (t+1)^{-a}$ and $\rho_t = t^{-a}$ for some $a\in (0,1]$ and the smoothing parameter $\beta=\beta_t \leq d^{-1}(t+1)^{-a}$, then
 we have
 \begin{align}
  \mathbb{E} \|v_t - \nabla f(z_t)\| \leq \frac{\beta L d}{2}+ \sqrt{C}(t+1)^{-\frac{a}{2}},
 \end{align}
 where $C=\frac{24dL^2D^2 + 3L^2+ 2\sigma^2_2}{2-2^{-a}-a}$ for some $a\in (0,1]$.
\end{lemma}
\begin{theorem} \label{th:4}
Suppose $\{x_t,y_t,z_t\}^{T-1}_{t=0}$ be generated from Algorithm \ref{alg:2} by using the \textbf{UniGE} zeroth-order gradient estimator.
Let $\alpha_t = \frac{1}{t+1}$, $\theta_t = \frac{1}{(t+1)(t+2)}$,
$\eta=\eta_t = T^{-\frac{2}{3}}$, $\gamma_t = (1+\theta_t)\eta_t$, $\rho_t = t^{-\frac{2}{3}}$ for $t\geq 1$ and $\beta=d^{-1}T^{-\frac{2}{3}}$,
then we have
 \begin{align}
 \mathbb{E} [\mathcal{G} (z_{\zeta})] = \frac{1}{T}\sum_{t=1}^{T-1}\mathbb{E} [\mathcal{G}(z_{t})] \leq O(\frac{\sqrt{d}}{T^{\frac{1}{3}}})+ O(\frac{\sqrt{d}\ln(T)}{T^{\frac{4}{3}}}), \nonumber
\end{align}
where $z_{\zeta}$ is chosen uniformly randomly from $\{z_t\}_{t=0}^{T-1}$.
\end{theorem}
\begin{remark}
 Theorem \ref{th:4} states that the Acc-SZOFW*(UniGE) algorithm has $O(\sqrt{d} T^{-\frac{1}{3}} )$ convergence rate.
 It needs $2$ samples to estimate the zeroth-order gradient $v_t$ at each iteration,
 and needs $T$ iterations. By $\sqrt{d}T^{-\frac{1}{3}} \leq \epsilon$,
 we choose $T=d^{\frac{3}{2}}\epsilon^{-3}$. Thus, the Acc-SZOFW* has the function query complexity of $2T = O(d^{\frac{3}{2}}\epsilon^{-3})$
 for finding an $\epsilon$-stationary point.
\end{remark}
\begin{figure}[htbp]
    \vskip 0.2in
    \begin{center}
	\subfigure[MNIST]{
		\begin{minipage}[b]{0.225\textwidth}
			\includegraphics[width=1\textwidth]{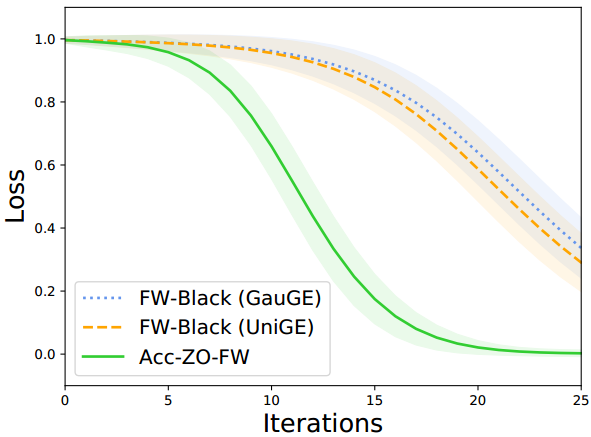}
		\end{minipage}
		\label{fig:sap_mnist}
	}
    \subfigure[CIFAR10]{
        \begin{minipage}[b]{0.225\textwidth}
        \includegraphics[width=1\textwidth]{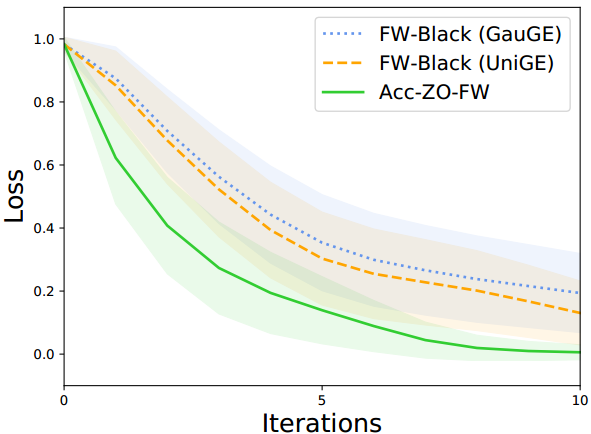}
        \end{minipage}
        \label{fig:sap_cifar10}
    }
    \vspace*{-4pt}
	\caption{The convergence of attack loss against iterations of three algorithms on the SAP problem. }
    \label{fig:sap}
    \end{center}
    \vskip -0.2in
\end{figure}
\begin{figure}[tb]
    \vskip 0.2in
    \begin{center}
	\subfigure[MNIST]{
		\begin{minipage}[b]{0.225\textwidth}
			\includegraphics[width=1\textwidth]{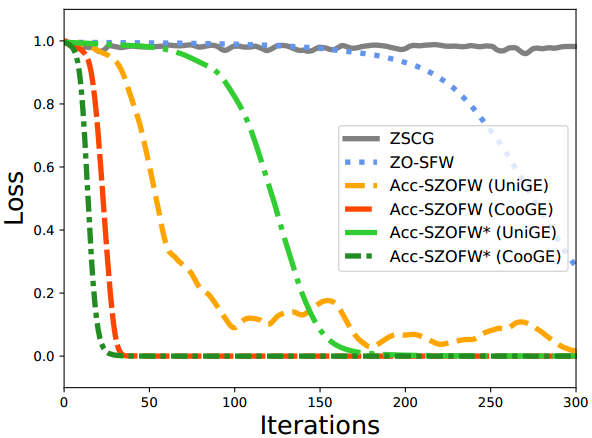}
		\end{minipage}
		\label{fig:uap_mnist_iteration}
	}
    \subfigure[CIFAR10]{
        \begin{minipage}[b]{0.225\textwidth}
            \includegraphics[width=1\textwidth]{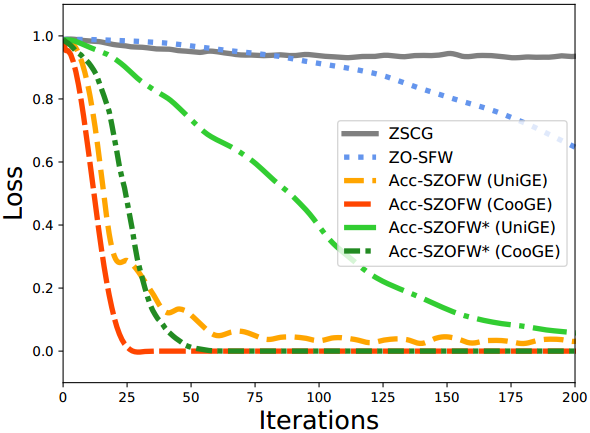}
        \end{minipage}
        \label{fig:uap_cifar10_iteration}
    }
	\subfigure[MNIST]{
		\begin{minipage}[b]{0.225\textwidth}
		    \includegraphics[width=1\textwidth]{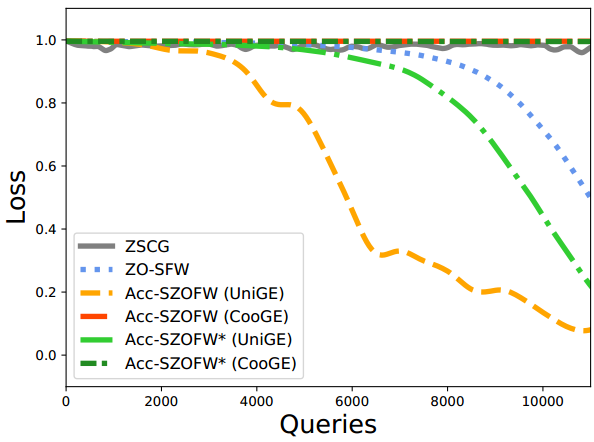}
		\end{minipage}
		\label{fig:uap_mnist_queries}
	}
    \subfigure[CIFAR10]{
        \begin{minipage}[b]{0.225\textwidth}
            \includegraphics[width=1\textwidth]{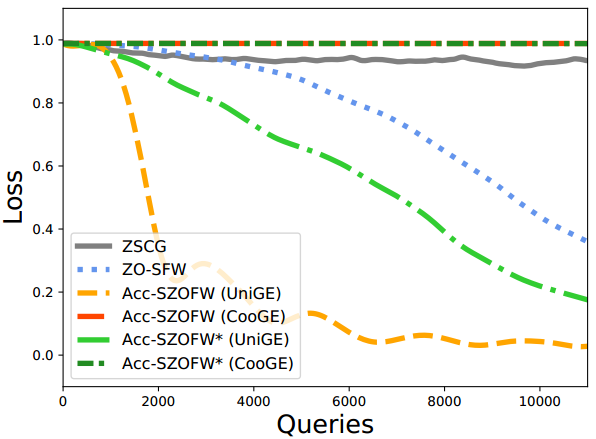}
        \end{minipage}
        \label{fig:uap_cifar10_queries}
    }
	\caption{Comparison of six algorithms for the UAP problem. Above: the convergence of attack loss against iterations. Below: the convergence of attack loss against queries.}
    \label{fig:uap}
    \end{center}
    \vskip -0.2in
\end{figure}
\section{Experiments}
In this section, we evaluate the performance of our proposed algorithms on two applications:
1) generating adversarial examples from black-box deep neural networks (DNNs)
and 2) robust black-box binary classification with $\ell_1$ norm bound constraint.
In the first application, we focus on two types of black-box adversarial attacks: \textit{single adversarial perturbation} (SAP) against an image
and \textit{universal adversarial perturbation} (UAP) against multiple images.
Specifically, we apply the SAP to demonstrate the efficiency of our deterministic Acc-ZO-FW algorithm
and compare with the FW-Black \cite{chen2018frank} algorithm.
While we apply the UAP and robust black-box binary classification to verify the efficiency of
our stochastic algorithms (i.e., Acc-SZOFW and Acc-SZOFW*)
and compare with the ZO-SFW \cite{sahu2019towards} algorithm and the ZSCG \cite{balasubramanian2018zeroth} algorithm.
All of our experiments are conducted on a server with an Intel Xeon 2.60GHz CPU
and an NVIDIA Titan Xp GPU.
Our implementation is based on PyTorch and the code to reproduce our results is publicly available at https://github.com/TLMichael/Acc-SZOFW.
\subsection{Black-box Adversarial Attack}
In this subsection, we apply the zeroth-order algorithms to generate adversarial perturbations to attack the pre-trained black-box DNNs,
whose parameters are hidden and only its outputs are accessible.
Let $(a, b)$ denote an image $a$ with its true label $b \in\{1,2, \cdots, K\}$,
where $K$ is the total number of image classes. For the SAP, we will design a perturbation $x$ for a single image $(a, b)$;
For the UAP, we will design a universal perturbation $x$ for multiple images $\{a_{i}, b_{i}\}_{i=1}^{n}$.
Following \cite{guo2019simple}, we solve the untargeted attack problem as follows:
\begin{equation}
    \min_{x \in \mathbb{R}^d} \frac{1}{n} \sum_{i=1}^{n} \ p (b_i \ |\  a_i+x), \quad \textrm{s.t.} \  \|x\|_\infty \leq \varepsilon \label{eq:ap}
\end{equation}
where $p(\cdot\ |\ a)$ represents probability associated with each class, that is, the final output after softmax of
neural network. In the problem \eqref{eq:ap}, we normalize the pixel values to $[0, 1]^d$.

In the experiment, we use the pre-trained DNN models on MNIST \cite{lecun2010mnist} and CIFAR10 \cite{krizhevsky2009learning} datasets
as the target black-box models,
which can attain 99.16\% and 93.07\% test accuracy, respectively.
In the SAP experiment, we choose $\varepsilon=0.3$ for MNIST and $\varepsilon=0.1$ for CIFAR10.
In the UAP experiment, we choose $\varepsilon=0.3$ for both MNIST dataset and CIFAR10 dataset.
For fair comparison, we choose the mini-batch size $b=20$ for all stochastic zeroth-order methods.
We refer readers to Appendix \ref{Appendix:A3}
for more details of the experimental setups
and the generated adversarial examples by our proposed algorithms.

Figure~\ref{fig:sap} shows that the convergence behaviors of three algorithms on SAP problem, where for each curve,
we generate 1000 adversarial perturbations on MNIST and 100 adversarial perturbations on CIFAR10,
the mean value of loss are plotted and the range of standard deviation is shown as a shadow overlay.
For both datasets, the results show that the attack loss values of our Acc-ZO-FW algorithm faster decrease than those of the FW-Black algorithms,
as the iteration increases, which demonstrates the superiority of our novel momentum technique and CooGE used in the Acc-ZO-FW algorithm.

Figure~\ref{fig:uap} shows that the convergence of six algorithms on UAP problem. For both datesets, the results show that all of our accelerated zeroth-order algorithms have faster convergence speeds (i.e. less iteration complexity) than the existing algorithms, while the Acc-SZOFW (UniGE) algorithm and the Acc-SZOFW* (UniGE) have faster convergence speeds (i.e. less function query complexity) than other algorithms (especially ZSCG and ZO-SFW), which verifies that the effectiveness of the variance reduced technique and the novel momentum technique in our accelerated algorithms. We notice that the periodic jitter of the curve of Acc-SZOFW (UniGE), which is due to the gradient variance reduction period of the variance reduced technique and the imprecise estimation of the uniform smoothing gradient estimator makes the jitter more significant. The jitter is less obvious in Acc-SZOFW (CooGE). Figure~\ref{fig:uap_mnist_queries} and Figure~\ref{fig:uap_cifar10_queries} represent the attack loss against the number of function queries. We observe that the performance of our CooGE-based algorithms degrade since the need of large number of queries to construct coordinate-wise gradient estimates. From these results, we also find that the CooGE-based methods can not be competent to high-dimensional datasets due to estimating each coordinate-wise gradient required at least $d$ queries. In addition, the performance of the Acc-SZOFW algorithms is better than the Acc-SZOFW* algorithms in most cases, which is due to the considerable mini-batch size used in the Acc-SZOFW algorithms.
\begin{figure*}[t]
    \vskip 0.2in
    \begin{center}
        \subfigure[phishing]{
            \begin{minipage}[b]{0.46\columnwidth}
                \includegraphics[width=1\textwidth]{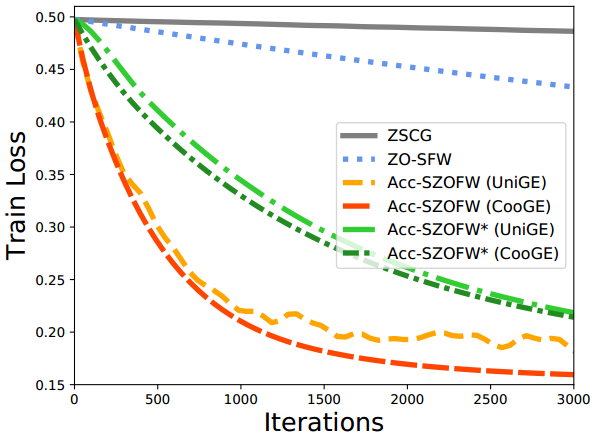}
            \end{minipage}
            \label{fig:rlc_phishing_iteration}
        }
        \subfigure[a9a]{
            \begin{minipage}[b]{0.46\columnwidth}
                \includegraphics[width=1\textwidth]{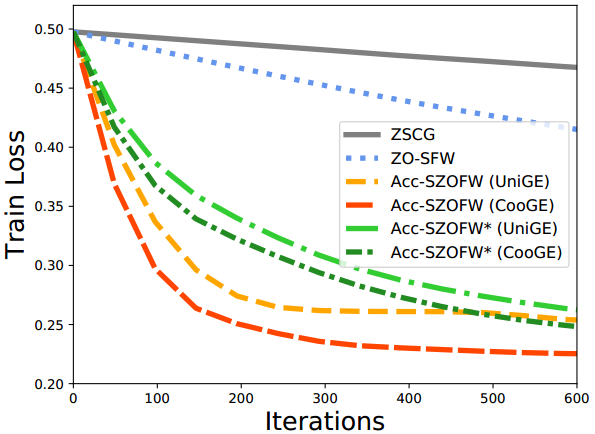}
            \end{minipage}
            \label{fig:rlc_a9a_iteration}
        }
        \subfigure[w8a]{
            \begin{minipage}[b]{0.46\columnwidth}
                \includegraphics[width=1\textwidth]{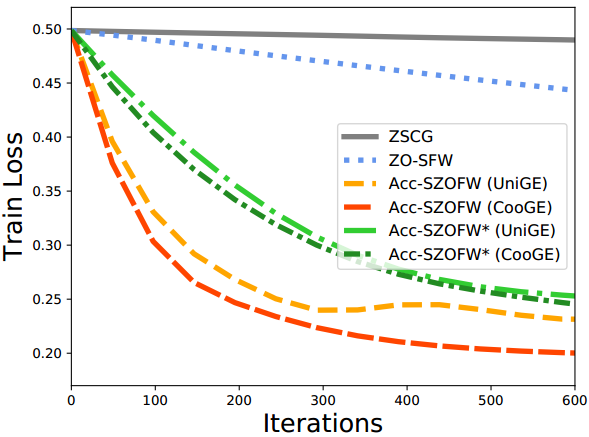}
            \end{minipage}
            \label{fig:rlc_w8a_iteration}
        }
        \subfigure[covtype.binary]{
            \begin{minipage}[b]{0.46\columnwidth}
                \includegraphics[width=1\textwidth]{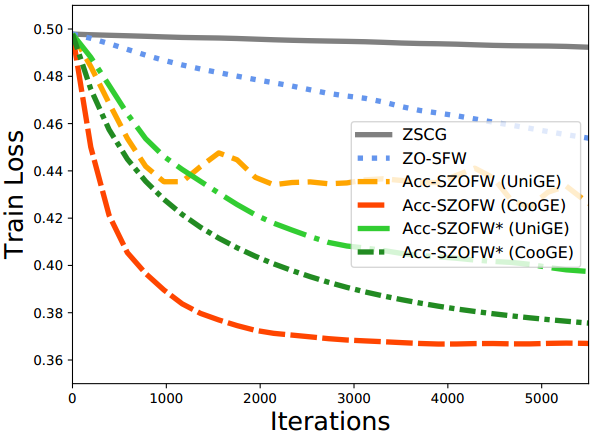}
            \end{minipage}
            \label{fig:rlc_covtype_iteration}
        }

        \subfigure[phishing]{
            \begin{minipage}[b]{0.46\columnwidth}
                \includegraphics[width=1\columnwidth]{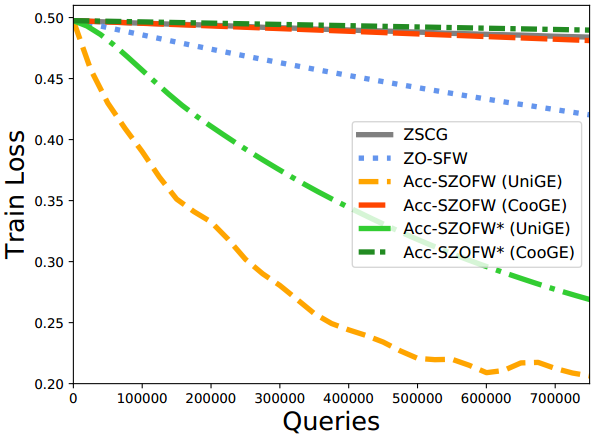}
            \end{minipage}
            \label{fig:rlc_phishing_queries}
        }
        \subfigure[a9a]{
            \begin{minipage}[b]{0.46\columnwidth}
                \includegraphics[width=1\columnwidth]{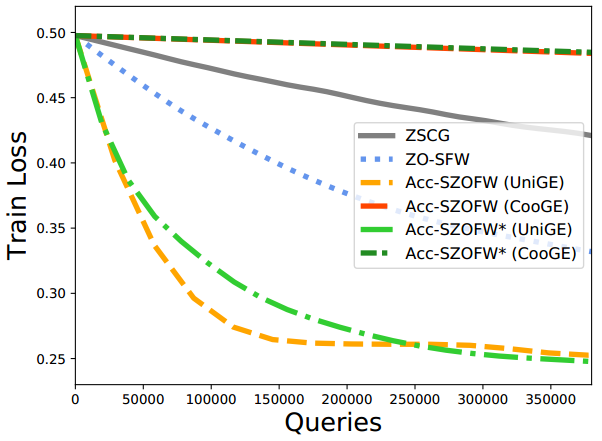}
            \end{minipage}
            \label{fig:rlc_a9a_queries}
        }
        \subfigure[w8a]{
            \begin{minipage}[b]{0.46\columnwidth}
                \includegraphics[width=1\columnwidth]{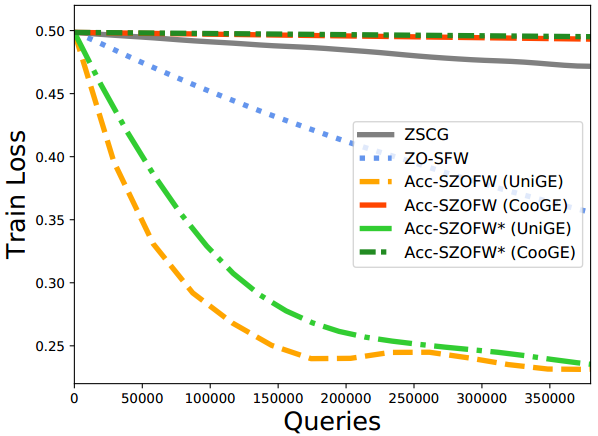}
            \end{minipage}
            \label{fig:rlc_w8a_queries}
        }
        \subfigure[covtype.binary]{
            \begin{minipage}[b]{0.46\columnwidth}
                \includegraphics[width=1\columnwidth]{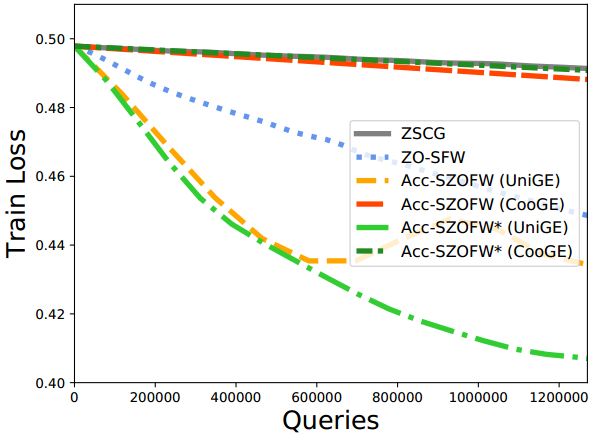}
            \end{minipage}
            \label{fig:rlc_covtype_queries}
        }
	\caption{Comparison of six algorithms for robust black-box binary classification. Above: the convergence of train loss against iterations. Below: the convergence of train loss against queries.}
    \label{fig:rlc-trainloss}
    \end{center}
    \vskip -0.2in
\end{figure*}
\subsection{Robust Black-box Binary Classification}
In this subsection, we apply the proposed algorithms to solve the robust black-box binary classification task.
Given a set of training samples $(a_{i}, l_{i})_{i=1}^{n}$, where $a_i \in \mathbb{R}^d$ and $l_i \in \{-1, +1\}$, we find optimal parameter $x \in \mathbb{R}^d$ by solving the problem:
\begin{equation}
    \min _{x \in \mathbb{R}^{d}} \frac{1}{n} \sum_{i=1}^{n} f_{i}(x), \quad \text { s.t. }\|x\|_{1} \leq \theta,
\end{equation}
where $f_{i}(x)$ is the black-box loss function, that only returns the function value given an input.
Here, we specify the loss function $f_{i}(x)=\frac{\sigma^{2}}{2}(1-\exp (-\frac{(l_{i}-a_{i}^{T} x)^{2}}{\sigma^{2}}))$,
which is the \textit{nonconvex} robust correntropy induced loss.
In the experiment, we use four public real datasets\footnote{These data are from the website \url{https://www.csie.ntu.edu.tw/~cjlin/libsvmtools/datasets/}}.
We set $\sigma=10$ and $\theta=10$.
For fair comparison, we choose the mini-batch size $b=100$ for all stochastic zeroth-order methods.
In the experiment, we use four public real datasets, which are summarized in Table \ref{tab:datasets-rlr}.
For each dataset, we use half of the samples as training data and the rest as testing data.
We elaborate the details of the parameter setting in Appendix \ref{Appendix:A4}.

\begin{table}[htb]
    \caption{Real datasets for black-box binary classification.}
    \label{tab:datasets-rlr}
    \vskip 0.15in
    \begin{center}
    \begin{small}
    \begin{sc}
    \begin{tabular}{lcccr}
    \toprule
    Data set & \#Samples & \#Features & \#Classes \\
    \midrule
    \textit{phishing}        & 11,055   & 68    & 2 \\
    \textit{a9a}             & 32,561   & 123   & 2 \\
    \textit{w8a}             & 49,749   & 300   & 2 \\
    \textit{covtype.binary}  & 581,012  & 54    & 2 \\
    \bottomrule
    \end{tabular}
    \end{sc}
    \end{small}
    \end{center}
    \vskip -0.1in
\end{table}

Figure~\ref{fig:rlc-trainloss} shows that the convergence of six algorithms on the black-box binary classification problem. We see that
the results are similar as in the case of the UAP problem. For all datasets, the results show that all of our accelerated algorithms have faster convergence speeds (i.e. less iteration complexity) than the existing algorithms, while the Acc-SZOFW (UniGE) algorithm and the Acc-SZOFW* (UniGE) have faster convergence speeds (i.e. less function query complexity) than other algorithms (especially ZSCG and ZO-SFW), which further demonstrates the efficiency of our accelerated algorithms. Similar to Figure~\ref{fig:uap}, the periodic jitter of the curve of Acc-SZOFW (UniGE) also appears and seems to be more intense in the covtype.binary dataset. We speculate that this is because the variance of the random gradient estimator is too high in this situation. We also provides the convergence of test loss in Appendix \ref{Appendix:A4},
which is analogous to those of train loss.
\vspace*{-6pt}
\section{Conclusions}
In the paper, we proposed a class of accelerated stochastic gradient-free and projection-free (zeroth-order Frank-Wolfe) methods.
In particular, we also proposed a momentum accelerated framework for the Frank-Wolfe methods.
Specifically, we presented an accelerated stochastic zeroth-order Frank-Wolfe (Acc-SZOFW) method
based on the variance reduced technique of SPIDER and the proposed momentum accelerated technique.
Further, we proposed a novel accelerated stochastic zeroth-order Frank-Wolfe (Acc-SZOFW*) to relax the large mini-batch size required in the Acc-SZOFW. Moreover, both the Acc-SZOFW and Acc-SZOFW* methods obtain a lower query complexity, which improves the state-of-the-art query complexity in both finite-sum and stochastic settings.

\section*{Acknowledgements}
We thank the anonymous reviewers for their valuable comments.
This paper was partially supported by the Natural Science Foundation of China (NSFC) under Grant No.
61806093 and No. 61682281, and the Key Program of
NSFC under Grant No. 61732006, and Jiangsu Postdoctoral Research Grant Program No. 2018K004A.

\bibliography{ZOFW}
\bibliographystyle{icml2020}


\clearpage

\begin{onecolumn}
\begin{appendices}

\section{Supplementary materials}
In this section, we first give the convergence analysis of our algorithms.
Next, we further provide detailed experimental setup and additional experimental results.

We first give some useful lemmas for both Acc-SZOFW and Acc-SZOFW* algorithms.

\begin{lemma} \label{lem:A1}
 The sequence $\{x_t,y_t,z_t\}_{t=0}^{T-1}$ is generated from the Acc-SZOFW or Acc-SZOFW* algorithm. Let $\alpha_t = \frac{1}{t+1}$, $\eta=\eta_t$ and $\gamma_t = (1+\theta_t)\eta_t$
 with $\theta_t = \frac{1}{(t+1)(t+2)}$,
 we have
 \begin{align}
  \mathbb{E}[\|z_{t+1}-z_t\|] \leq 2\eta D, \quad \mathbb{E}[\|x_{t+1} - z_{t+1}\| ]\leq \eta D.
 \end{align}
\end{lemma}

\begin{proof}
By using the steps 12 to 14 in Algorithm \ref{alg:1} or the steps 10 to 12 in Algorithm \ref{alg:2}, we have
\begin{align} \label{eq:A1}
 z_{t+1}-x_{t+1} & = (1 - \alpha_{t+1})(y_{t+1}-x_{t+1}) \nonumber \\
 & = (1 - \alpha_{t+1})(z_t + \eta_t(w_t - x_t) - x_{t} - \gamma_t(w_t-x_t)) \nonumber \\
 & = (1 - \alpha_{t+1})(1-\gamma_t)(z_t-x_t) + (1 - \alpha_{t+1})(\eta_t-\gamma_t)(w_t-z_t).
\end{align}
By recursion to the above equality \eqref{eq:A1},
we have
\begin{align} \label{eq:A2}
 x_{t+1} - z_{t+1} = (1-\alpha_{t+1}) \sum_{i=0}^t\big( \prod_{j=i+1}^t (1-\alpha_j)(1-\gamma_j)\big) (\gamma_i-\eta_i)(w_i-z_i).
\end{align}
According to \eqref{eq:A2}, we can obtain
\begin{align} \label{eq:A3}
 z_{t+1} - z_t & = \eta_t(w_t - z_t) - \alpha_{t+1}(y_{t+1} - x_{t+1}) \nonumber \\
 & = \eta_t(w_t - z_t) + \frac{\alpha_{t+1}}{1-\alpha_{t+1}}(x_{t+1} - z_{t+1}) \nonumber \\
 & = \eta_t(w_t - z_t) + \alpha_{t+1}\sum_{i=0}^t\big( \prod_{j=i+1}^t (1-\alpha_j)(1-\gamma_j)\big) (\gamma_i-\eta_i)(w_i-z_i).
\end{align}
By \eqref{eq:A3}, we have
\begin{align} \label{eq:A4}
 \|z_{t+1}-z_t\| & \leq \eta_t\|w_t - z_t\| + \alpha_{t+1}\|\sum_{i=0}^t\big( \prod_{j=i+1}^t (1-\alpha_j)(1-\gamma_j)\big) (\gamma_i-\eta_i)(w_i-z_i)\| \nonumber \\
 & \leq  \eta_t\|w_t - z_t\| + \alpha_{t+1}\sum_{i=0}^t\|\big( \prod_{j=i+1}^t (1-\alpha_j)(1-\gamma_j)\big) (\gamma_i-\eta_i)(w_i-z_i)\| \nonumber \\
 & \leq  \eta_t\|w_t - z_t\| + \alpha_{t+1}\sum_{i=0}^t\| (\gamma_i-\eta_i)(w_i-z_i)\| \nonumber \\
 & =  \eta_t\|w_t - z_t\| + \alpha_{t+1}\eta\sum_{i=0}^t \theta_i \| w_i-z_i\| \nonumber \\
 & \leq 2\eta D,
\end{align}
where the third inequality follows by $\prod_{j=i+1}^t (1-\alpha_j)(1-\gamma_j) \in [0,1]$
due to that $0 \leq \alpha_j \leq 1$ and $0<\gamma_j <1$, and the above equality follows by $\gamma_i = (1+\theta_i)\eta_i$,
and the last inequality holds by Assumption 3 and $\alpha_{t+1}=\frac{1}{t+2}$, $\theta_i=\frac{1}{(i+1)(i+2)}\in (0,1)$.
Meanwhile, by \eqref{eq:A2}, we have
\begin{align} \label{eq:A5}
 \|x_{t+1} - z_{t+1}\| & = \|(1-\alpha_{t+1}) \sum_{i=0}^t\big( \prod_{j=i+1}^t (1-\alpha_j)(1-\gamma_j)\big) (\gamma_i-\eta_i)(w_i-z_i)\| \nonumber \\
 & \leq \sum_{i=0}^t\|\big( \prod_{j=i+1}^t (1-\alpha_j)(1-\gamma_j)\big) (\gamma_i-\eta_i)(w_i-z_i)\| \nonumber \\
 & \leq \sum_{i=0}^t\| (\gamma_i-\eta_i)(w_i-z_i)\| \leq \eta \sum_{i=0}^t \theta_i \|w_i-z_i\|\nonumber \\
 & \leq \eta \sum_{i=0}^t \theta_i D \leq \eta D,
\end{align}
where the first inequality holds by Triangle inequality and $\alpha_{t+1}\in [0,1]$, and the second inequality follows by $\prod_{j=i+1}^t (1-\alpha_j)(1-\gamma_j) \in [0,1]$
due to that $0 \leq \alpha_j \leq 1$ and $0<\gamma_j <1$, and
the last inequality follows by $\sum_{i=0}^t \theta_i = \sum_{i=0}^t\frac{1}{(i+1)(i+2)}=1-\frac{1}{t+2}\leq 1$.
\end{proof}

\begin{lemma}  \label{lem:A2}
 Suppose that random variables $\zeta_1, \cdots, \zeta_m$ are independent and its individual mean is $0$, \emph{i.e.,} $\mathbb{E}[\zeta_i]=0$ for $i=1,\cdots,m$,
 we have
\begin{align}
 \mathbb{E}[ \|\zeta_1+\cdots+\zeta_m\|^2 ] = \mathbb{E}[ \|\zeta_1\|^2 + \cdots + \|\zeta_m\|^2].
\end{align}
\end{lemma}
\begin{proof}
 It is easy verified that
 \begin{align}
  \mathbb{E}[ \|\zeta_1+\cdots+\zeta_m\|^2 ] & = \sum_{i, j=1}^m \mathbb{E}[\zeta_i\zeta_j]
  = \mathbb{E}[ \|\zeta_1\|^2 + \cdots + \|\zeta_m\|^2] + \sum_{i\neq j} \mathbb{E}[\zeta_i\zeta_j] \nonumber \\
  & = \mathbb{E}[ \|\zeta_1\|^2 + \cdots + \|\zeta_m\|^2],
 \end{align}
 where the last inequality holds by the fact that random variables $\zeta_1, \cdots, \zeta_m$ are independent and its individual mean is $0$.
\end{proof}

Next, we review some useful lemmas.
\begin{lemma} \label{lem:A3}
(Lemma 3 in \cite{ji2019improved}) For any $L$-smooth function $f(x)$, given the zeroth-order gradient
 $\hat{\nabla}_{coo}f(x) = \sum_{j=1}^d \frac{f(x + \mu e_j)-f(x - \mu e_j)}{2\mu}e_j$, for any parameter $\mu>0$ and
any $x\in \mathbb{R}^d$, we have
\begin{align}
\|\hat{\nabla}_{coo}f(x) - \nabla f(x)\|^2 \leq L^2d\mu^2.
\end{align}
\end{lemma}
\begin{lemma} \label{lem:A4}
(Lemma 5 in \cite{ji2019improved}) Let $f_{\beta}(x)=\mathbb{E}_{u\sim U_B}[f(x+\beta u)]$ be a smooth approximation of $f(x)$,
where $U_B$ is the uniform distribution over the $d$-dimensional unit Euclidean ball $B$. Then we have
\begin{itemize}
 \vspace*{-4pt}
\item[(1)] $|f_{\beta}(x) - f(x)| \leq \frac{\beta^2 L}{2}$ and $\|\nabla f_{\beta}(x) - \nabla f(x)\| \leq \frac{\beta Ld}{2}$ for any $x\in \mathbb{R}^d$;
 \vspace*{-4pt}
\item[(2)] $\mathbb{E}[\frac{1}{|\mathcal{S}|}\sum_{j\in \mathcal{S}}\hat{\nabla}_{uni}f_j(x)] = \nabla f_{\beta}(x)$ for any $x\in \mathbb{R}^d$;
 \vspace*{-4pt}
\item[(3)] $\mathbb{E}\|\hat{\nabla}_{uni}f_j(x)-\hat{\nabla}_{uni}f_j(y)\|^2 \leq 3dL^2\|x-y\|^2 + \frac{3L^2d^2\beta}{2}$ for any $x, y \in \mathbb{R}^d$ and any $j$ draw from certain distribution.
\end{itemize}
\end{lemma}
Note that the result (3) of Lemma \ref{lem:A4} is an extended result from Lemma 5 in \cite{ji2019improved}.
\subsection{ Convergence Analysis of the Acc-SZOFW Algorithm }
\label{Appendix:A1}
In this subsection, we study the convergence properties of the Acc-SZOFW Algorithm based on the CooGE and UniGE,
respectively. We first study the convergence properties of the deterministic \textbf{Acc-ZO-FW} Algorithm as a baseline,
which is Algorithm \ref{alg:1} using the \textbf{deterministic} zeroth-order gradient $v_t=\frac{1}{n}\sum_{i=1}^n \hat{\nabla}_{coo} f_i(z_t)$
for solving the \textbf{finite-sum} problem \eqref{eq:1}.

\subsubsection{ Convergence Analysis of the Acc-ZO-FW Algorithm }
\begin{theorem}
Suppose $\{x_t,y_t,z_t\}^{T-1}_{t=0}$ be generated from Algorithm \ref{alg:1} by using the \textbf{deterministic}
zeroth-order gradient $v_t=\frac{1}{n}\sum_{i=1}^n \hat{\nabla}_{coo} f_i(z_t)$,
and let $\alpha_t = \frac{1}{t+1}$, $\theta_t = \frac{1}{(t+1)(t+2)}$,
$\gamma_t = (1+\theta_t)\eta_t$, $\eta = \eta_t= T^{-\frac{1}{2}}$, $\mu=d^{-\frac{1}{2}}T^{-\frac{1}{2}}$,
then we have
\begin{align}
\frac{1}{T}\sum_{t=1}^{T-1}\mathcal{G}(z_{t}) \leq O(\frac{1}{T^{\frac{1}{2}}}) + O(\frac{\ln(T)}{T^{\frac{3}{2}}}).
\end{align}
\end{theorem}

\begin{proof}
By using the Assumption 1, \emph{i.e.,} $f(x)$ is L-smooth, we have
\begin{align} \label{eq:H1}
 f(z_{t+1}) &\le f(z_t) + \langle \nabla f(z_t), z_{t+1} - z_k \rangle + \frac{L}{2} \|z_{t+1} - z_t \|^2 \nonumber \\
&= f(z_t) + (1 -\alpha_{t+1}) \langle \nabla f(z_t), y_{t+1} - z_t\rangle  + \alpha_{t+1} \langle \nabla f(z_t), x_{t+1} - z_t \rangle + \frac{L}{2} \|z_{t+1} - z_t \|^2 \nonumber \\
&= f(z_t) + ((1 -\alpha_{t+1}) \eta_t + \alpha_{t+1} \gamma_t) \langle \nabla f(z_t),  w_t - z_t \rangle + \alpha_{t+1} (1 - \gamma_t) \langle \nabla f(z_t), x_t - z_t \rangle \nonumber \\
&  \quad + \frac{L}{2} \|z_{t+1} - z_t\|^2,
\end{align}
where the first equality holds by $z_{t+1} = (1 - \alpha_{t+1}) y_{t+1} + \alpha_{t+1} x_{t+1}$, and the last equality holds by
$x_{t+1}=x_t+\gamma_t(w_t-x_t)$ and $y_{t+1}=z_t+\eta_t(w_t-z_t)$.

Let $\hat{w}_t = \arg\max_{w\in \mathcal{X}} \langle w, -\nabla f(z_t)\rangle = \arg\min_{w\in \mathcal{X}} \langle w, \nabla f(z_t)\rangle$, we have
\begin{align} \label{eq:H2}
 \langle \nabla f(z_t), w_t - z_t\rangle & = \langle \nabla f(z_t) - v_t, w_t - z_t\rangle + \langle v_t, w_t - z_t\rangle \nonumber \\
 & \leq \langle  \nabla f(z_t) - v_t, w_t - z_t\rangle + \langle v_t, \hat{w}_t - z_t\rangle \nonumber \\
 & = \langle  \nabla f(z_t) - v_t, w_t - \hat{w}_t\rangle + \langle \nabla f(z_t), \hat{w}_t - z_t\rangle \nonumber \\
 & = \langle  \nabla f(z_t) - v_t, w_t - \hat{w}_t\rangle - \mathcal{G}(z_t) \nonumber \\
 & \leq D\|\nabla f(z_t) - v_t\| - \mathcal{G}(z_t),
\end{align}
where the first inequality holds by the step 11 of Algorithm \ref{alg:1}, and the third equality holds by the definition of Frank-Wolfe gap
$\mathcal{G}(z_t)=\max_{w\in \mathcal{X}}\langle w-z_t,-\nabla f(z_t)\rangle=\langle \hat{w}_t-z_t,-\nabla f(z_t)\rangle$,
and the last inequality follows by Cauchy-Schwarz inequality and Assumption 3.

Next, we consider the upper bound of $\langle \nabla f(z_t), x_t - z_t \rangle$. We have
\begin{align} \label{eq:H3}
& \langle \nabla f(z_t), x_t - z_t \rangle =  \langle \nabla f(z_t)-\nabla f(z_{t-1}), x_t - z_t \rangle + \langle \nabla f(z_{t-1}), x_t - z_t \rangle \nonumber \\
 & \leq \|\nabla f(z_t)-\nabla f(z_{t-1})\| \|x_t - z_t\| + \langle \nabla f(z_{t-1}), x_t - z_t \rangle \nonumber \\
 & \leq L\|z_t-z_{t-1}\| \|x_t - z_t\| + \langle \nabla f(z_{t-1}), x_t - z_t \rangle \nonumber \\
 & \leq 2 L\eta^2 D^2   + \langle \nabla f(z_{t-1}), x_t - z_t \rangle  \nonumber \\
 & = 2 L\eta^2 D^2 + (1 - \alpha_{t})(1-\gamma_{t-1})\langle \nabla f(z_{t-1}),z_{t-1}-x_{t-1}\rangle + (1 - \alpha_{t})(\eta_{t-1}-\gamma_{t-1})\langle \nabla f(z_{t-1}), w_{t-1}-z_{t-1} \rangle \nonumber \\
 & \leq 2 L\eta^2 D^2 + (1 - \alpha_{t})(1-\gamma_{t-1})\langle \nabla f(z_{t-1}),z_{t-1}-x_{t-1}\rangle + (1 - \alpha_{t})\theta_{t-1}\eta_{t-1}\big( D\|\nabla f(z_{t-1}) - v_{t-1}\|
 - \mathcal{G}(z_{t-1}) \big)\nonumber \\
  & \leq 2 L\eta^2 D^2  + \langle \nabla f(z_{t-1}),z_{t-1}-x_{t-1}\rangle + \theta_{t-1}\eta_{t-1}D\|\nabla f(z_{t-1}) - v_{t-1}\|
\end{align}
where the first inequality holds by Cauchy-Schwarz inequality, and the third inequality holds by Lemma \ref{lem:A1}, and the second equality follows by the above equality \eqref{eq:A1},
and the forth inequality holds by the inequality \eqref{eq:H2}, and the last inequality follows by $\mathcal{G}(z_{t-1}) \geq 0$ and $0\leq 1 - \alpha_{t}\leq 1$.
By recursion to \eqref{eq:H3}, we can obtain
\begin{align} \label{eq:H4}
 \langle \nabla f(z_t), x_t - z_t \rangle \leq 2 t L\eta^2 D^2 + \eta D\sum_{i=0}^{t-1}\theta_{i}\|\nabla f(z_{i}) - v_{i}\|
\end{align}

By using the above inequalities \eqref{eq:H1}, \eqref{eq:H2} and \eqref{eq:H4}, we have
\begin{align} \label{eq:H5}
 f(z_{t+1}) &\le f(z_t) + ((1 -\alpha_{t+1}) \eta_t + \alpha_{t+1} \gamma_t) \langle \nabla f(z_t),  w_t - z_t \rangle + \alpha_{t+1} (1 - \gamma_t) \langle \nabla f(z_t), x_t - z_t \rangle \nonumber \\
&  \quad + \frac{L}{2} \|z_{t+1} - z_t\|^2 \nonumber \\
& \le f(z_t)+ \frac{L}{2} \|z_{t+1} - z_t\|^2 - ((1 -\alpha_{t+1}) \eta_t + \alpha_{t+1} \gamma_t)\mathcal{G}(z_{t}) + ((1 -\alpha_{t+1}) \eta_t + \alpha_{t+1} \gamma_t)D\|\nabla f(z_t) - v_t\| \nonumber \\
& \quad +  2t\alpha_{t+1} (1 - \gamma_t) L\eta^2 D^2 + \alpha_{t+1} (1 - \gamma_t)\eta D\sum_{i=0}^{t-1}\theta_{i}\|\nabla f(z_{i}) - v_{i}\| \nonumber \\
& \le f(z_t) +  2L\eta^2D^2 - \eta \mathcal{G}(z_{t}) + 2\eta D\|\nabla f(z_t) - v_t\| +  2L\eta^2 D^2 + \alpha_{t+1}\eta D\sum_{i=0}^{t-1}\theta_{i}\|\nabla f(z_{i}) - v_{i}\| \nonumber \\
& =  f(z_t) +  4L\eta^2D^2 - \eta \mathcal{G}(z_{t}) + 2\eta D\|\nabla f(z_t) - v_t\| + \alpha_{t+1}\eta D\sum_{i=0}^{t-1}\theta_{i}\|\nabla f(z_{i}) - v_{i}\|,
\end{align}
where the third inequality holds by the results in Lemma \ref{lem:A1}, and $\gamma_t\in (0,1)$, $\alpha_{t+1}=\frac{1}{t+2}$ and $\alpha_{t+1}(\gamma_t-\eta_t)=\frac{1}{(t+2)^2(t+1)}\eta \leq \eta$.
Summing the inequality \eqref{eq:H5} from $t=0$ to $T-1$, we can obtain
\begin{align}
 \eta \sum_{t=1}^{T-1}\mathcal{G}(z_{t}) & \leq f(z_0) - f(z_{T-1}) +  4TL\eta^2D^2 + 2\eta D \sum_{t=0}^{T-1}\|\nabla f(z_t) - v_t\|
 + \eta D \sum_{t=0}^{T-1}\alpha_{t+1}\sum_{i=0}^{t-1}\theta_{i}\|\nabla f(z_{i}) - v_{i}\| \nonumber \\
 & = f(z_0) - f(z_{T-1}) +  4TL\eta^2D^2 + 2\eta D \sum_{t=0}^{T-1}\|\nabla f(z_t) - v_t\| + \eta D \sum_{t=0}^{T-2}\theta_{t}\big(\sum_{i=t+1}^{T-1}\alpha_{i+1}\big)\|\nabla f(z_{t}) - v_{t}\| \nonumber \\
 & \leq f(z_0) - \inf_{z\in \mathcal{X}} f(z) +  4TL\eta^2D^2 + 2\eta D \sum_{t=0}^{T-1}\|\nabla f(z_t) - v_t\|
 + \eta D \sum_{t=0}^{T-2}\theta_{t}\big(\sum_{i=t+1}^{T-1}\alpha_{i+1}\big)\|\nabla f(z_{t}) - v_{t}\|\nonumber \\
 & \leq \triangle +  4TL\eta^2D^2 + 2\eta D \sum_{t=0}^{T-1}\|\nabla f(z_t) - v_t\| + \eta D \sum_{t=0}^{T-2}\theta_{t}\big(\sum_{i=t+1}^{T-1}\alpha_{i+1}\big)\|\nabla f(z_{t}) - v_{t}\|,
\end{align}
where the last inequality holds by the Assumption 4.
Since $v_t =\frac{1}{n}\sum_{i=1}^n \hat{\nabla}_{coo} f_i(z_t) = \hat{\nabla}_{coo} f(z_t)$, we have
\begin{align}
 \frac{1}{T}\sum_{t=1}^{T-1}\mathcal{G}(z_{t}) &\leq \frac{\triangle}{\eta T} +  4L\eta D^2 + \frac{2D}{T} \sum_{t=0}^{T-1}\|\nabla f(z_t) - v_t\| +  \frac{D}{T} \sum_{t=0}^{T-2}\theta_{t}\big(\sum_{i=t+1}^{T-1}\alpha_{i+1}\big)\|\nabla f(z_{t}) - v_{t}\| \nonumber \\
 & \leq \frac{\triangle}{\eta T} +  4L\eta D^2 + 2DL\sqrt{d}\mu +  \frac{D}{T} \sum_{t=0}^{T-2}\frac{1}{(t+1)(t+2)}\big(\sum_{i=t+1}^{T-1}\frac{1}{i+1}\big) L\sqrt{d}\mu \nonumber \\
 & \leq \frac{\triangle}{\eta T} +  4L\eta D^2 + 2DL\sqrt{d}\mu +  \frac{D}{T} \sum_{t=0}^{T-1}\frac{1}{(t+1)(t+2)}\ln(\frac{T}{t+1}) L\sqrt{d}\mu \nonumber \\
 & \leq \frac{\triangle}{\eta T} +  4L\eta D^2 + 2DL\sqrt{d}\mu  +  \frac{D\ln(T)}{T} L\sqrt{d}\mu,
\end{align}
where the second inequality holds by Lemma \ref{lem:A3}, and the third inequality follows by the inequality $\sum_{i=t+1}^{T-1}\frac{1}{i+1} \leq \int^{T}_{t+1}\frac{1}{x}dx \leq \ln(\frac{T}{t+1})$,
and the last inequality holds by $\sum_{t=0}^{T-1} \frac{1}{(t+1)(t+2)}\ln(\frac{T}{t+1}) \leq \ln(T)\sum_{t=0}^{T-1} \frac{1}{(t+1)(t+2)}=\ln(T)(1-\frac{1}{T+1})\leq \ln(T)$.
Let $\eta = T^{-\frac{1}{2}}$, $\mu=d^{-\frac{1}{2}}T^{-\frac{1}{2}}$,
then we have
\begin{align}
 \frac{1}{T}\sum_{t=1}^{T-1}\mathcal{G}(z_{t}) \leq O(\frac{1}{T^{\frac{1}{2}}}) + O(\frac{\ln(T)}{T^{\frac{3}{2}}}).
\end{align}

\end{proof}

\subsubsection{ Convergence Analysis of the Acc-SZOFW (CooGE) Algorithm }
In this subsection, we study the convergence properties of the Acc-SZOFW (CooGE) Algorithm, which uses the CooGE to estimate gradients.
We begin with giving an upper bound of variance of stochastic zeroth-order gradient $v_t$.
\begin{lemma} \label{lem:B1}
 Suppose the zeroth-order gradient $v_t $ be generated from Algorithm \ref{alg:1} by using the CooGE zeroth-order gradient.
 Let $\alpha_t = \frac{1}{t+1}$, $\theta_t = \frac{1}{(t+1)(t+2)}$, $\eta=\eta_t$ and
$\gamma_t = (1+\theta_t)\eta_t$ in Algorithm \ref{alg:1}.
 For the \textbf{finite-sum} setting, we have
 \begin{align}
  \mathbb{E} \|\nabla f(z_t)-v_t\|  \leq L\sqrt{d}\mu + \frac{L(\sqrt{6d}\mu + 2\sqrt{3D}\eta)}{\sqrt{b/q}}.
 \end{align}
 For the \textbf{stochastic} setting, we have
 \begin{align}
  \mathbb{E} \|\nabla f(z_t)-v_t\|  \leq L\sqrt{d}\mu + \frac{L(\sqrt{6d}\mu + 2\sqrt{3D}\eta)}{\sqrt{b_2/q}} + \frac{\sqrt{3}\sigma_1}{\sqrt{b_1}} + \sqrt{6d}L\mu.
 \end{align}
\end{lemma}
\begin{proof}
 Without loss of generality, we first give an upper bound of $\mathbb{E}\|v_t - \hat{\nabla} f_{coo}(z_{t})\|^2$ in the stochastic setting. By the definition of $v_t$,
 we have
 \begin{align} \label{eq:B1}
  &\mathbb{E}\| \hat{\nabla} f_{coo}(z_{t})-v_t\|^2= \mathbb{E}\|\underbrace{\hat{\nabla} f_{coo}(z_{t}) -\hat{\nabla} f_{coo}(z_{t-1}) -\frac{1}{b_2} \sum_{j \in \mathcal{B}_2} [\hat{\nabla}_{coo} f_{j}(z_t) - \hat{\nabla}_{coo} f_{j}(z_{t-1})] }_{=T_1} + \underbrace{\hat{\nabla} f_{coo}(z_{t-1})- v_{t-1}}_{=T_2}\|^2 \nonumber \\
  & = \mathbb{E}\|\hat{\nabla} f_{coo}(z_{t}) -\hat{\nabla} f_{coo}(z_{t-1}) -\frac{1}{b_2} \sum_{j \in \mathcal{B}_2} [\hat{\nabla}_{coo} f_{j}(z_t) - \hat{\nabla}_{coo} f_{j}(z_{t-1})] \|^2 + \mathbb{E}\| \hat{\nabla} f_{coo}(z_{t-1})- v_{t-1}\|^2 \nonumber \\
  & = \frac{1}{b_2^2} \sum_{j \in \mathcal{B}_2}\mathbb{E}\| \hat{\nabla} f_{coo}(z_{t}) -\hat{\nabla} f_{coo}(z_{t-1}) - \hat{\nabla}_{coo} f_{j}(z_t) + \hat{\nabla}_{coo} f_{j}(z_{t-1}) \|^2 + \mathbb{E}\| \hat{\nabla} f_{coo}(z_{t-1})- v_{t-1}\|^2 \nonumber \\
  & \leq \frac{1}{b^2_2} \sum_{j \in \mathcal{B}_2} \mathbb{E}\|\hat{\nabla}_{coo} f_{j}(z_t) - \hat{\nabla}_{coo} f_{j}(z_{t-1}) \|^2 + \mathbb{E}\| \hat{\nabla} f_{coo}(z_{t-1})- v_{t-1}\|^2 \nonumber \\
  & \leq \frac{1}{b_2^2} \sum_{j \in \mathcal{B}_2}\mathbb{E}\| \hat{\nabla}_{coo} f_{j}(z_t) \!-\! \nabla f_{j}(z_t) \!+\! \nabla f_{j}(z_t) \!-\! \nabla f_{j}(z_{t-1}) \!+\! \nabla f_{j}(z_{t-1})\!-\! \hat{\nabla}_{coo} f_{j}(z_{t-1}) \|^2 \!+\! \mathbb{E}\| \hat{\nabla} f_{coo}(z_{t-1})- v_{t-1}\|^2 \nonumber \\
  & \leq \frac{1}{b_2} \big( 3L^2\|z_{t}-z_{t-1}\|^2 + 6L^2d\mu^2 \big) +   \mathbb{E}\| \hat{\nabla} f_{coo}(z_{t-1})- v_{t-1}\|^2 \nonumber \\
  & \leq \frac{6L^2d\mu^2 }{b_2} + \frac{12L^2\eta^2D}{b_2} +   \mathbb{E}\| \hat{\nabla} f_{coo}(z_{t-1})- v_{t-1}\|^2,
 \end{align}
 where the second equality holds by $\mathbb{E}[T_1]=0$ and $T_2$ is independent to $\mathcal{B}_2$, and the third equality holds by Lemma \ref{lem:A2},
 and the third inequality holds by Cauchy-Schwarz inequality and Lemma \ref{lem:A3}, and the last inequality follows by Lemma \ref{lem:A1}.

 Let $n_t = \lfloor t/q\rfloor $ such that $n_tq \leq t \leq (n_t+1)q-1$.
 When $t=n_tq$, $v_t = \frac{1}{b_1}\sum_{j\in\mathcal{B}_1}\hat{\nabla}_{coo} f_{j}(z_{t})$,
 so we have
 \begin{align}
 &\mathbb{E}\| \hat{\nabla} f_{coo}(z_{n_tq})-  \frac{1}{b_1}\sum_{j\in\mathcal{B}_1}\hat{\nabla}_{coo} f_{j}(z_{n_tq})\|^2 \nonumber \\
 & = \mathbb{E}\| \hat{\nabla} f_{coo}(z_{n_tq}) - \nabla f(z_{n_tq}) + \nabla f(z_{n_tq}) -\frac{1}{b_1}\sum_{j\in\mathcal{B}_1}\nabla f_{j}(z_{n_tq})
 + \frac{1}{b_1}\sum_{j\in\mathcal{B}_1}\nabla f_{j}(z_{n_tq})-  \frac{1}{b_1}\sum_{j\in\mathcal{B}_1}\hat{\nabla}_{coo} f_{j}(z_{n_tq})\|^2 \nonumber \\
 & \leq 3\mathbb{E}\| \nabla f(z_{n_tq}) -\frac{1}{b_1}\sum_{j\in\mathcal{B}_1}\nabla f_{j}(z_{n_tq})\|^2 + 3\mathbb{E}\| \hat{\nabla} f_{coo}(z_{n_tq}) - \nabla f(z_{n_tq})\|^2 + 3\mathbb{E}\|\frac{1}{b_1}\sum_{j\in\mathcal{B}_1}\nabla f_{j}(z_{n_tq})-  \frac{1}{b_1}\sum_{j\in\mathcal{B}_1}\hat{\nabla}_{coo} f_{j}(z_{n_tq})\|^2 \nonumber \\
 & \leq \frac{3\sigma^2_1}{b_1} + 3\mathbb{E}\| \hat{\nabla} f_{coo}(z_{n_tq}) - \nabla f(z_{n_tq})\|^2 + 3\mathbb{E}\|\frac{1}{b_1}\sum_{j\in\mathcal{B}_1}\nabla f_{j}(z_{n_tq})-  \frac{1}{b_1}\sum_{j\in\mathcal{B}_1}\hat{\nabla}_{coo} f_{j}(z_{n_tq})\|^2 \nonumber \\
 & \leq \frac{3\sigma^2_1}{b_1} + 3L^2d\mu^2 + 3L^2d\mu^2 = \frac{3\sigma^2_1}{b_1} + 6L^2d\mu^2,
 \end{align}
 where the first inequality holds by the Young's inequality; the second inequality holds by Lemma \ref{lem:A2};
 the third inequality follows by Lemma \ref{lem:A3}.
 By recursion to \eqref{eq:B1}, we have
 \begin{align}
 \mathbb{E}\| \hat{\nabla} f_{coo}(z_{t})-v_t\|^2 & \leq (t-n_t q )\frac{6L^2(d\mu^2 + 2\eta^2D)}{b_2} + \mathbb{E}\| \hat{\nabla} f_{coo}(z_{n_tq})- v_{n_tq}\|^2 \nonumber \\
 & \leq \frac{6qL^2(d\mu^2 + 2\eta^2D)}{b_2} + \frac{3\sigma^2_1}{b_1} + 6L^2d\mu^2.
 \end{align}
  By Jensen's inequality and the inequality $(a+b+c+d)^{1/2} \leq a^{1/2} + b^{1/2} + c^{1/2} + d^{1/2}$ with $a,b,c,d\geq0$, we have
 \begin{align}
  \mathbb{E} \|\hat{\nabla} f_{coo}(z_{t}) - v_t\| \leq \sqrt{\mathbb{E} \|\hat{\nabla} f_{coo}(z_{t}) - v_t\|^2} \leq \frac{L(\sqrt{6d}\mu + 2\sqrt{3D}\eta)}{\sqrt{b_2/q}} +  \frac{\sqrt{3}\sigma_1}{\sqrt{b_1}} + \sqrt{6d}L\mu.
 \end{align}
 Thus we have
 \begin{align}
  \mathbb{E} \|\nabla f(z_t)-v_t\| &=\mathbb{E} \|\nabla f(z_t)- \hat{\nabla}_{coo} f(z_t) + \hat{\nabla}_{coo} f(z_t)- v_t\| \nonumber\\
  & \leq \mathbb{E} \|\nabla f(z_t)- \hat{\nabla}_{coo} f(z_t)\| + \mathbb{E} \|\hat{\nabla}_{coo} f(z_t)- v_t\| \nonumber\\
  & \leq L\sqrt{d}\mu + \frac{L(\sqrt{6d}\mu + 2\sqrt{3D}\eta)}{\sqrt{b_2/q}} +  \frac{\sqrt{3}\sigma_1}{\sqrt{b_1}} + \sqrt{6d}L\mu,
 \end{align}
 where the last inequality holds by the above Lemma \ref{lem:A3}.

 For the finite-sum setting, when $t=n_tq$, $v_t = \hat{\nabla} f_{coo}(z_{t})$, so we have $\mathbb{E}\| \hat{\nabla} f_{coo}(z_{n_tq})- v_{n_tq}\|^2=0$.
 Following the above result, we have
  \begin{align}
  \mathbb{E} \|\nabla f(z_t)-v_t\| \leq L\sqrt{d}\mu + \frac{L(\sqrt{6d}\mu + 2\sqrt{3D}\eta)}{\sqrt{b/q}}.
 \end{align}
\end{proof}

Next, based on the above lemma,
we will give the convergence properties of the Acc-SZOFW (CooGE) algorithm.
\begin{theorem} \label{th:C1}
Suppose $\{x_t,y_t,z_t\}^{T-1}_{t=0}$ be generated from Algorithm \ref{alg:1} by using the CooGE zeroth-order gradient, and let $\alpha_t = \frac{1}{t+1}$, $\theta_t = \frac{1}{(t+1)(t+2)}$,
$\gamma_t = (1+\theta_t)\eta_t$, $\eta = \eta_t= T^{-\frac{1}{2}}$, $\mu=d^{-\frac{1}{2}}T^{-\frac{1}{2}}$, $b_2 = q$ or $b=q$, and $b_1 = T$,
then we have
\begin{align}
\frac{1}{T}\sum_{t=1}^{T-1}\mathcal{G}(z_{t}) \leq O(\frac{1}{T^{\frac{1}{2}}}) + O(\frac{\ln(T)}{T^{\frac{3}{2}}}).
\end{align}
\end{theorem}
\begin{proof}
By using the Assumption 1, \emph{i.e.,} $f(x)$ is $L$-smooth, we have
\begin{align} \label{eq:C1}
 f(z_{t+1}) &\le f(z_t) + \langle \nabla f(z_t), z_{t+1} - z_k \rangle + \frac{L}{2} \|z_{t+1} - z_t \|^2 \nonumber \\
&= f(z_t) + (1 -\alpha_{t+1}) \langle \nabla f(z_t), y_{t+1} - z_t\rangle  + \alpha_{t+1} \langle \nabla f(z_t), x_{t+1} - z_t \rangle + \frac{L}{2} \|z_{t+1} - z_t \|^2 \nonumber \\
&= f(z_t) + ((1 -\alpha_{t+1}) \eta_t + \alpha_{t+1} \gamma_t) \langle \nabla f(z_t),  w_t - z_t \rangle + \alpha_{t+1} (1 - \gamma_t) \langle \nabla f(z_t), x_t - z_t \rangle \nonumber \\
&  \quad + \frac{L}{2} \|z_{t+1} - z_t\|^2,
\end{align}
where the first equality holds by $z_{t+1} = (1 - \alpha_{t+1}) y_{t+1} + \alpha_{t+1} x_{t+1}$, and the last equality holds by
$x_{t+1}=x_t+\gamma_t(w_t-x_t)$ and $y_{t+1}=z_t+\eta_t(w_t-z_t)$.

Let $\hat{w}_t = \arg\max_{w\in \mathcal{X}} \langle w, -\nabla f(z_t)\rangle = \arg\min_{w\in \mathcal{X}} \langle w, \nabla f(z_t)\rangle$, we have
\begin{align} \label{eq:C2}
 \langle \nabla f(z_t), w_t - z_t\rangle & = \langle \nabla f(z_t) - v_t, w_t - z_t\rangle + \langle v_t, w_t - z_t\rangle \nonumber \\
 & \leq \langle  \nabla f(z_t) - v_t, w_t - z_t\rangle + \langle v_t, \hat{w}_t - z_t\rangle \nonumber \\
 & = \langle  \nabla f(z_t) - v_t, w_t - \hat{w}_t\rangle + \langle \nabla f(z_t), \hat{w}_t - z_t\rangle \nonumber \\
 & = \langle  \nabla f(z_t) - v_t, w_t - \hat{w}_t\rangle - \mathcal{G}(z_t) \nonumber \\
 & \leq D\|\nabla f(z_t) - v_t\| - \mathcal{G}(z_t),
\end{align}
where the first inequality holds by the step 11 of Algorithm \ref{alg:1}, and the third equality holds by the definition of Frank-Wolfe gap $\mathcal{G}(z_t)$,
and the second inequality follows by Cauchy-Schwarz inequality and Assumption 3.

Next, we consider the upper bound of $\langle \nabla f(z_t), x_t - z_t \rangle$. We have
\begin{align} \label{eq:C3}
& \langle \nabla f(z_t), x_t - z_t \rangle =  \langle \nabla f(z_t)-\nabla f(z_{t-1}), x_t - z_t \rangle + \langle \nabla f(z_{t-1}), x_t - z_t \rangle \nonumber \\
 & \leq \|\nabla f(z_t) - \nabla f(z_{t-1})\| \|x_t - z_t\| + \langle \nabla f(z_{t-1}), x_t - z_t \rangle \nonumber \\
 & \leq L\|z_t-z_{t-1}\| \|x_t - z_t\| + \langle \nabla f(z_{t-1}), x_t - z_t \rangle \nonumber \\
 & \leq 2 L\eta^2 D^2   + \langle \nabla f(z_{t-1}), x_t - z_t \rangle  \nonumber \\
 & = 2 L\eta^2 D^2 + (1 - \alpha_{t})(1-\gamma_{t-1})\langle \nabla f(z_{t-1}),z_{t-1}-x_{t-1}\rangle + (1 - \alpha_{t})(\eta_{t-1}-\gamma_{t-1})\langle \nabla f(z_{t-1}), w_{t-1}-z_{t-1} \rangle \nonumber \\
 & \leq 2 L\eta^2 D^2 + (1 - \alpha_{t})(1-\gamma_{t-1})\langle \nabla f(z_{t-1}),z_{t-1}-x_{t-1}\rangle + (1 - \alpha_{t})\theta_{t-1}\eta_{t-1}\big( D\|\nabla f(z_{t-1}) - v_{t-1}\|
 - \mathcal{G}(z_{t-1}) \big)\nonumber \\
  & \leq 2 L\eta^2 D^2  + \langle \nabla f(z_{t-1}),z_{t-1}-x_{t-1}\rangle + \theta_{t-1}\eta_{t-1}D\|\nabla f(z_{t-1}) - v_{t-1}\|,
\end{align}
where the first inequality holds by Cauchy-Schwarz inequality, and the third inequality holds by Lemma \ref{lem:A1},
and the forth inequality holds by the inequality \eqref{eq:C2}, and the last inequality follows by $\alpha_{t} \in [0,1]$, $\gamma_t\in (0,1)$ and $\mathcal{G}(z_{t-1}) \geq 0$.
By recursion to \eqref{eq:C3}, we can obtain
\begin{align} \label{eq:C4}
 \langle \nabla f(z_t), x_t - z_t \rangle \leq 2 t L\eta^2 D^2 + \eta D\sum_{i=0}^{t-1}\theta_{i}\|\nabla f(z_{i}) - v_{i}\|
\end{align}

By using the above inequalities \eqref{eq:C1}, \eqref{eq:C2} and \eqref{eq:C4} and the results in Lemma \ref{lem:A1}, we have
\begin{align} \label{eq:C5}
 f(z_{t+1}) &\le f(z_t) + ((1 -\alpha_{t+1}) \eta_t + \alpha_{t+1} \gamma_t) \langle \nabla f(z_t),  w_t - z_t \rangle + \alpha_{t+1} (1 - \gamma_t) \langle \nabla f(z_t), x_t - z_t \rangle \nonumber \\
&  \quad + \frac{L}{2} \|z_{t+1} - z_t\|^2 \nonumber \\
& \le f(z_t) +  2L\eta^2D^2 - ((1 -\alpha_{t+1}) \eta_t + \alpha_{t+1} \gamma_t)\mathcal{G}(z_{t}) + ((1 -\alpha_{t+1}) \eta_t + \alpha_{t+1} \gamma_t)D\|\nabla f(z_t) - v_t\| \nonumber \\
& \quad +  2t\alpha_{t+1} (1 - \gamma_t) L\eta^2 D^2 + \alpha_{t+1} (1 - \gamma_t)\eta D\sum_{i=0}^{t-1}\theta_{i}\|\nabla f(z_{i}) - v_{i}\| \nonumber \\
& \le f(z_t) +  2L\eta^2D^2 - \eta \mathcal{G}(z_{t}) + 2\eta D\|\nabla f(z_t) - v_t\| +  2L\eta^2 D^2 + \alpha_{t+1}\eta D\sum_{i=0}^{t-1}\theta_{i}\|\nabla f(z_{i}) - v_{i}\| \nonumber \\
& =  f(z_t) +  4L\eta^2D^2 - \eta \mathcal{G}(z_{t}) + 2\eta D\|\nabla f(z_t) - v_t\| + \alpha_{t+1}\eta D\sum_{i=0}^{t-1}\theta_{i}\|\nabla f(z_{i}) - v_{i}\|.
\end{align}
Summing the inequality \eqref{eq:C5} from $t=0$ to $T-1$, we can obtain
\begin{align}
 \eta \sum_{t=1}^{T-1}\mathcal{G}(z_{t}) & \leq f(z_0) - f(z_{T-1}) +  4TL\eta^2D^2 + 2\eta D \sum_{t=0}^{T-1}\|\nabla f(z_t) - v_t\|
 + \eta D \sum_{t=0}^{T-1}\alpha_{t+1}\sum_{i=0}^{t-1}\theta_{i}\|\nabla f(z_{i}) - v_{i}\| \nonumber \\
 & = f(z_0) - f(z_{T-1}) +  4TL\eta^2D^2 + 2\eta D \sum_{t=0}^{T-1}\|\nabla f(z_t) - v_t\| + \eta D \sum_{t=0}^{T-2}\theta_{t}\big(\sum_{i=t+1}^{T-1}\alpha_{i+1}\big)\|\nabla f(z_{t}) - v_{t}\| \nonumber \\
 & \leq f(z_0) - \inf_{z\in \mathcal{X}} f(z) +  4TL\eta^2D^2 + 2\eta D \sum_{t=0}^{T-1}\|\nabla f(z_t) - v_t\|
 + \eta D \sum_{t=0}^{T-2}\theta_{t}\big(\sum_{i=t+1}^{T-1}\alpha_{i+1}\big)\|\nabla f(z_{t}) - v_{t}\|\nonumber \\
 & \leq \triangle +  4TL\eta^2D^2 + 2\eta D \sum_{t=0}^{T-1}\|\nabla f(z_t) - v_t\| + \eta D \sum_{t=0}^{T-2}\theta_{t}\big(\sum_{i=t+1}^{T-1}\alpha_{i+1}\big)\|\nabla f(z_{t}) - v_{t}\|,
\end{align}
where the last inequality holds by the Assumption 4.
Without loss of generality, we first consider the stochastic setting. Then we have
\begin{align}
 \frac{1}{T}\sum_{t=1}^{T-1}\mathcal{G}(z_{t}) &\leq \frac{\triangle}{\eta T} +  4L\eta D^2 + \frac{2D}{T} \sum_{t=0}^{T-1}\|\nabla f(z_t) - v_t\| +  \frac{D}{T} \sum_{t=0}^{T-2}\theta_{t}\big(\sum_{i=t+1}^{T-1}\alpha_{i+1}\big)\|\nabla f(z_{t}) - v_{t}\| \nonumber \\
 & \leq \frac{\triangle}{\eta T} +  4L\eta D^2 + 2DL\sqrt{d}\mu + \frac{2DL(\sqrt{6d}\mu + 2\sqrt{3D}\eta)}{\sqrt{b_2/q}}  + \frac{2\sqrt{3}D\sigma_1}{\sqrt{b_1}} + 2\sqrt{6d}D L\mu \nonumber \\
 & \quad +  \frac{D}{T} \sum_{t=0}^{T-2}\frac{1}{(t+1)(t+2)}\big(\sum_{i=t+1}^{T-1}\frac{1}{i+1}\big)\big( L\sqrt{d}\mu + \frac{L(\sqrt{6d}\mu + 2\sqrt{3D}\eta)}{\sqrt{b_2/q}} + \frac{\sqrt{3}\sigma_1}{\sqrt{b_1}} + \sqrt{6d}L\mu\big) \nonumber \\
 & \leq \frac{\triangle}{\eta T} +  4L\eta D^2 + 2DL\sqrt{d}\mu + \frac{2DL(\sqrt{6d}\mu + 2\sqrt{3D}\eta)}{\sqrt{b_2/q}} +\frac{2\sqrt{3}D\sigma_1}{\sqrt{b_1}} + 2\sqrt{6d}DL\mu \nonumber \\
 & \quad +  \frac{D}{T} \sum_{t=0}^{T-1}\frac{1}{(t+1)(t+2)}\ln(\frac{T}{t+1})\big( L\sqrt{d}\mu + \frac{L(\sqrt{6d}\mu + 2\sqrt{3D}\eta)}{\sqrt{b_2/q}} + \frac{\sqrt{3}\sigma_1}{\sqrt{b_1}} + \sqrt{6d}L\mu \big)\nonumber \\
 & \leq \frac{\triangle}{\eta T} +  4L\eta D^2 + 2DL\sqrt{d}\mu + \frac{2DL(\sqrt{6d}\mu + 2\sqrt{3D}\eta)}{\sqrt{b_2/q}} + \frac{2\sqrt{3}D\sigma_1}{\sqrt{b_1}} + 2\sqrt{6d}DL\mu\nonumber \\
 & \quad   +  \frac{D\ln(T)}{T} \big( L\sqrt{d}\mu + \frac{L(\sqrt{6d}\mu + 2\sqrt{3D}\eta)}{\sqrt{b_2/q}} + \frac{\sqrt{3}\sigma_1}{\sqrt{b_1}} + \sqrt{6d}L\mu \big),
\end{align}
where the second inequality holds by Lemma \ref{lem:B1}, and the third inequality follows by the inequality $\sum_{i=t+1}^{T-1}\frac{1}{i+1} \leq \int^{T}_{t+1}\frac{1}{x}dx \leq \ln(\frac{T}{t+1})$,
and the last inequality holds by $\sum_{t=0}^{T-1} \frac{1}{(t+1)(t+2)}\ln(\frac{T}{t+1}) \leq \ln(T)\sum_{t=0}^{T-1} \frac{1}{(t+1)(t+2)}=\ln(T)(1-\frac{1}{T+1})\leq \ln(T)$.
Let $\eta = T^{-\frac{1}{2}}$, $\mu=d^{-\frac{1}{2}}T^{-\frac{1}{2}}$, $b_2 = q$ and $b_1 = T^{-1}$,
then we have
\begin{align}
 \frac{1}{T}\sum_{t=1}^{T-1}\mathcal{G}(z_{t}) \leq O(\frac{1}{T^{\frac{1}{2}}}) + O(\frac{\ln(T)}{T^{\frac{3}{2}}}).
\end{align}

In the finite-sum setting, let $\eta = T^{-\frac{1}{2}}$, $\mu=d^{-\frac{1}{2}}T^{-\frac{1}{2}}$ and $b_2 = q$.
Following the above the above result, we also have
\begin{align}
 \frac{1}{T}\sum_{t=1}^{T-1}\mathcal{G}(z_{t}) \leq O(\frac{1}{T^{\frac{1}{2}}}) + O(\frac{\ln(T)}{T^{\frac{3}{2}}}).
\end{align}

\end{proof}

\subsubsection{ Convergence Analysis of the Acc-SZOFW (UniGE) Algorithm }
In this subsection, we study the convergence properties of the Acc-SZOFW (UniGE) Algorithm.
We begin with giving an upper bound of variance of stochastic zeroth-order gradient $v_t$.

\begin{lemma} \label{lem:D1}
 Suppose the zeroth-order stochastic gradient $v_t$ be generated from Algorithm \ref{alg:1} by using the UniGE zeroth-order gradient.
 Let $\alpha_t = \frac{1}{t+1}$, $\theta_t = \frac{1}{(t+1)(t+2)}$, $\gamma_t = (1+\theta_t)\eta_t$ in Algorithm \ref{alg:1}.
 For the \textbf{stochastic} setting, we have
 \begin{align}
  \mathbb{E} \|\nabla f(z_t)-v_t\|  \leq \frac{\beta L d}{2} + \frac{L(\sqrt{3}d\beta + 2\sqrt{6Dd}\eta)}{\sqrt{2b_2/q}} +  \frac{\sigma_2}{\sqrt{b_1}}.
 \end{align}
\end{lemma}
\begin{proof}
First, we define $f_{\beta}(x)=\mathbb{E}_{u\sim U_B}[f(x+\beta u)]$ be a smooth approximation of $f(x)$,
where $U_B$ is the uniform distribution over the $d$-dimensional unit Euclidean ball $B$. By Lemma 5 in \cite{ji2019improved},
we have $\mathbb{E}_{(u,\xi)}[\hat{\nabla}_{uni} f_{\xi}(x)]=f_{\beta}(x)$.
We give an upper bound of $\mathbb{E}\|v_t - \nabla f_{\beta}(z_{t})\|^2$ in the stochastic setting. By the definition of $v_t$,
 we have
 \begin{align} \label{eq:D1}
  &\mathbb{E}\| \nabla f_{\beta}(z_{t})-v_t\|^2= \mathbb{E}\|\underbrace{\nabla f_{\beta}(z_{t}) -\nabla f_{\beta}(z_{t-1}) -\frac{1}{b_2} \sum_{j \in \mathcal{B}_2} [\hat{\nabla}_{uni} f_{j}(z_t) - \hat{\nabla}_{uni} f_{j}(z_{t-1})] }_{=T_1} + \underbrace{\nabla f_{\beta}(z_{t-1})- v_{t-1}}_{=T_2}\|^2 \nonumber \\
  & = \mathbb{E}\|\nabla f_{\beta}(z_{t}) -\nabla f_{\beta}(z_{t-1}) -\frac{1}{b_2} \sum_{j \in \mathcal{B}_2} [\hat{\nabla}_{uni} f_{j}(z_t) - \hat{\nabla}_{uni} f_{j}(z_{t-1})] \|^2
  + \mathbb{E}\| \nabla f_{\beta}(z_{t-1})- v_{t-1}\|^2 \nonumber \\
  & = \frac{1}{b_2^2} \sum_{j \in \mathcal{B}_2}\mathbb{E}\| \nabla f_{\beta}(z_{t}) -\nabla f_{\beta}(z_{t-1}) - \hat{\nabla}_{uni} f_{j}(z_t) + \hat{\nabla}_{uni} f_{j}(z_{t-1}) \|^2 + \mathbb{E}\| \nabla f_{\beta}(z_{t-1}) - v_{t-1}\|^2 \nonumber \\
  & \leq\frac{1}{b^2_2} \sum_{j \in \mathcal{B}_2}\mathbb{E}\| \hat{\nabla}_{uni} f_{j}(z_t) - \hat{\nabla}_{uni} f_{j}(z_{t-1}) \|^2 + \mathbb{E}\| \nabla f_{\beta}(z_{t-1}) - v_{t-1}\|^2 \nonumber \\
  & \leq \frac{1}{b_2} \big( 3dL^2\|z_{t}-z_{t-1}\|^2 + \frac{3L^2d^2\beta^2}{2} \big) +   \mathbb{E}\| \nabla f_{\beta}(z_{t-1})- v_{t-1}\|^2 \nonumber \\
  & \leq \frac{3L^2d^2\beta^2 }{2b_2} + \frac{12dL^2\eta^2D}{b_2} +   \mathbb{E}\| \nabla f_{\beta}(z_{t-1})- v_{t-1}\|^2,
 \end{align}
 where the second equality holds by $\mathbb{E}[T_1]=0$ and $T_2$ is independent to $\mathcal{B}_2$, and the second inequality holds by Lemma \ref{lem:A4}.

 Let $n_t = \lfloor t/q\rfloor $ such that $n_tq \leq t \leq (n_t+1)q-1$.
 When $t=n_tq$, $v_t = \frac{1}{b_1}\sum_{j\in\mathcal{B}_1}\hat{\nabla}_{uni} f_{j}(z_{t})$,
 so we have $\mathbb{E}\| \nabla f_{\beta}(z_{n_tq})- v_{n_tq}\|^2 \leq \frac{\sigma_2^2}{b_1}$.
 By recursion to \eqref{eq:D1}, we have
 \begin{align}
 \mathbb{E}\| \nabla f_{\beta}(z_{t})-v_t\|^2 & \leq (t-n_t q )\frac{3L^2(d^2\beta^2 + 8d\eta^2D)}{2b_2} + \mathbb{E}\| \nabla f_{\beta}(z_{n_tq})- v_{n_tq}\|^2 \nonumber \\
 & \leq \frac{3qL^2(d^2\beta^2 + 8d\eta^2D)}{2b_2} +  \frac{\sigma_2^2}{b_1}.
 \end{align}
  By Jensen's inequality and the inequality $(a+b+c)^{1/2} \leq a^{1/2} + b^{1/2} + c^{1/2}$ with $a,b,c\geq0$, we have
 \begin{align}
  \mathbb{E} \|\nabla f_{\beta}(z_{t}) - v_t\| \leq \sqrt{\mathbb{E} \|\nabla f_{\beta}(z_{t}) - v_t\|^2} \leq \frac{L(\sqrt{3}d\beta + 2\sqrt{6Dd}\eta)}{\sqrt{2b_2/q}} +  \frac{\sigma_2}{\sqrt{b_1}}.
 \end{align}
 Thus we have
 \begin{align}
  \mathbb{E} \|\nabla f(z_t)-v_t\| &=\mathbb{E} \|\nabla f(z_t)- \nabla f_{\beta}(z_{t}) + \nabla f_{\beta}(z_{t})- v_t\| \nonumber\\
  & \leq \mathbb{E} \|\nabla f(z_t)- \nabla f_{\beta}(z_{t})\| + \mathbb{E} \|\nabla f_{\beta}(z_{t}) - v_t\| \nonumber\\
  & \leq \frac{\beta L d}{2} + \frac{L(\sqrt{3}d\beta + 2\sqrt{6Dd}\eta)}{\sqrt{2b_2/q}} +  \frac{\sigma_2}{\sqrt{b_1}},
 \end{align}
 where the last inequality holds by the above Lemma \ref{lem:A4}.

\end{proof}

Next, based on the above lemma,
we will give the convergence properties of the Acc-SZOFW (UniGE) algorithm.
\begin{theorem}
Suppose $\{x_t,y_t,z_t\}^{T-1}_{t=0}$ be generated from Algorithm \ref{alg:1}, and let $\alpha_t = \frac{1}{t+1}$, $\theta_t = \frac{1}{(t+1)(t+2)}$,
$\gamma_t = (1+\theta_t)\eta_t$, $\eta = \eta_t= T^{-\frac{1}{2}}$, $\beta=d^{-1}T^{-\frac{1}{2}}$, $b_2 = q$, and $b_1 = T$,
then we have
\begin{align}
\frac{1}{T}\sum_{t=1}^{T-1}\mathcal{G}(z_{t}) \leq O(\frac{\sqrt{d}}{T^{\frac{1}{2}}}) + O(\frac{\sqrt{d}\ln(T)}{T^{\frac{3}{2}}}).
\end{align}
\end{theorem}
\begin{proof}
 This proof can follow the proof of Theorem \ref{th:C1}. Here we show some different results.
 We have
 \begin{align}
 \frac{1}{T}\sum_{t=1}^{T-1}\mathcal{G}(z_{t}) &\leq \frac{\triangle}{\eta T} +  4L\eta D^2 + \frac{2D}{T} \sum_{t=0}^{T-1}\|\nabla f(z_t) - v_t\| +  \frac{D}{T} \sum_{t=0}^{T-2}\theta_{t}\big(\sum_{i=t+1}^{T-1}\alpha_{i+1}\big)\|\nabla f(z_{t}) - v_{t}\| \nonumber \\
 & \leq \frac{\triangle}{\eta T} +  4L\eta D^2 + DLd\beta + \frac{2DL(\sqrt{3}d\beta + 2\sqrt{6Dd}\eta)}{\sqrt{2b_2/q}} +  \frac{2D\sigma_2}{\sqrt{b_1}} \nonumber \\
 & \quad +  \frac{D}{T} \sum_{t=0}^{T-2}\frac{1}{(t+1)(t+2)}\big(\sum_{i=t+1}^{T-1}\frac{1}{i+1}\big)\big( \frac{\beta L d}{2} + \frac{L(\sqrt{3}d\beta + 2\sqrt{6Dd}\eta)}{\sqrt{2b_2/q}}
 +  \frac{\sigma_2}{\sqrt{b_1}}\big) \nonumber \\
 & \leq \frac{\triangle}{\eta T} +  4L\eta D^2 + DLd\beta + \frac{2DL(\sqrt{3}d\beta + 2\sqrt{6Dd}\eta)}{\sqrt{2b_2/q}} +  \frac{2D\sigma_2}{\sqrt{b_1}} \nonumber \\
 & \quad +  \frac{D}{T} \sum_{t=0}^{T-1}\frac{1}{(t+1)(t+2)}\ln(\frac{T}{t+1})\big( \frac{\beta L d}{2} + \frac{L(\sqrt{3}d\beta + 2\sqrt{6Dd}\eta)}{\sqrt{2b_2/q}} +  \frac{\sigma_2}{\sqrt{b_1}}\big)\nonumber \\
 & \leq \frac{\triangle}{\eta T} +  4L\eta D^2 + DLd\beta + \frac{2DL(\sqrt{3}d\beta + 2\sqrt{6Dd}\eta)}{\sqrt{2b_2/q}} +  \frac{2D\sigma_2}{\sqrt{b_1}} \nonumber \\
 & \quad   +  \frac{D\ln(T)}{T} \big( \frac{\beta L d}{2} + \frac{L(\sqrt{3}d\beta + 2\sqrt{6Dd}\eta)}{\sqrt{2b_2/q}} +  \frac{\sigma_2}{\sqrt{b_1}}\big),
\end{align}
where the second inequality holds by Lemma \ref{lem:D1}, and the third inequality follows by the inequality $\sum_{i=t+1}^{T-1}\frac{1}{i+1} \leq \int^{T}_{t+1}\frac{1}{x}dx \leq \ln(\frac{T}{t+1})$,
and the fourth inequality holds by $\sum_{t=0}^{T-1} \frac{1}{(t+1)(t+2)}\ln(\frac{T}{t+1}) \leq \ln(T)\sum_{t=0}^{T-1} \frac{1}{(t+1)(t+2)}=\ln(T)(1-\frac{1}{T+1})\leq \ln(T)$.
Let $\eta = T^{-\frac{1}{2}}$, $\beta=d^{-1}T^{-\frac{1}{2}}$, $b_2 = q$ and $b_1 = T^{-1}$,
then we have
\begin{align}
 \frac{1}{T}\sum_{t=1}^{T-1}\mathcal{G}(z_{t}) \leq O(\frac{\sqrt{d}}{T^{\frac{1}{2}}}) + O(\frac{\sqrt{d}\ln(T)}{T^{\frac{3}{2}}}).
\end{align}

\end{proof}

\subsection{ Convergence Analysis of the Acc-SZOFW* Algorithm }
\label{Appendix:A2}
In this subsection, we study the convergence properties of the Acc-SZOFW* Algorithm based on the CooGE and UniGE,
respectively.

\subsubsection{ Convergence Analysis of the Acc-SZOFW* (CooGE) Algorithm}
\begin{lemma} \label{lem:E1}
 Suppose the zeroth-order stochastic gradient $v_{t} = \hat{\nabla}_{coo} f_{\xi_t}(z_t) + (1-\rho_t)\big(v_{t-1} -\hat{\nabla}_{coo} f_{\xi_t}(z_{t-1})\big)$
 be generated from Algorithm \ref{alg:2}.
 Let $\alpha_t = \frac{1}{t+1}$, $\theta_t = \frac{1}{(t+1)(t+2)}$,
 $\gamma_t = (1+\theta_t)\eta_t$, $\eta=\eta_t \leq (t+1)^{-a}$ and $\rho_t = t^{-a}$ for some $a\in (0,1]$ and the smoothing parameter $\mu = \mu_t \leq d^{-\frac{1}{2}}(t+1)^{-a}$,
 then we have
 \begin{align}
  \mathbb{E} \|v_t - \nabla f(z_t)\| \leq L\sqrt{d}\mu + \sqrt{C}(t+1)^{-\frac{a}{2}},
 \end{align}
 where $C=\frac{2 (12L^2 D^2 + 12L^2 + 3\sigma^2_1)}{2-2^{-a}-a}$.
\end{lemma}

\begin{proof}
 We begin with giving an upper bound of $ \mathbb{E} \|v_t - \hat{\nabla}_{coo} f(z_t)\|$ with $\hat{\nabla}_{coo} f(z_t) = \mathbb{E}_{\xi_t} [\hat{\nabla}_{coo} f_{\xi_t}(z_t)]$.
 It is easy verified that
 \begin{align}
  v_t-v_{t-1} = -\rho_tv_{t-1} + (1-\rho_t)(\hat{\nabla}_{coo} f_{\xi_t}(z_t) - \hat{\nabla}_{coo} f_{\xi_t}(z_{t-1})) + \rho_t\hat{\nabla}_{coo} f_{\xi_t}(z_t),
 \end{align}
 and $\hat{\nabla}_{coo} f(z_t) = \mathbb{E}_{\xi_t} [\hat{\nabla}_{coo} f_{\xi_t}(z_t)]$, and $\mathbb{E}_{\xi_t} [\hat{\nabla}_{coo} f_{\xi_t}(z_t)-\hat{\nabla}_{coo} f_{\xi_t}(z_{t-1})]=
 \hat{\nabla}_{coo} f(z_t)  - \hat{\nabla}_{coo} f(z_{t-1})$.
 Then we have
 \begin{align} \label{eq:E1}
  & A_t = \mathbb{E} \|\hat{\nabla}_{coo} f(z_t) - v_t\|^2 = \mathbb{E} \|\hat{\nabla}_{coo} f(z_t) - \hat{\nabla}_{coo} f(z_{t-1}) + \hat{\nabla}_{coo} f(z_{t-1}) - v_{t-1} - (v_t-v_{t-1})\|^2 \nonumber \\
  & =  \mathbb{E} \|\hat{\nabla}_{coo} f(z_t) \!-\! \hat{\nabla}_{coo} f(z_{t-1}) \!+\! \hat{\nabla}_{coo} f(z_{t-1}) \!-\! v_{t-1} \!+\! \rho_tv_{t-1} \!-\! (1\!-\!\rho_t)(\hat{\nabla}_{coo} f_{\xi_t}(z_t) \!-\! \hat{\nabla}_{coo} f_{\xi_t}(z_{t-1})) \!-\! \rho_t\hat{\nabla}_{coo} f_{\xi_t}(z_t)\|^2 \nonumber \\
  & = \mathbb{E} \|(1-\rho_t)(\hat{\nabla}_{coo} f(z_{t-1}) -v_{t-1}) + (1-\rho_t)\big(\hat{\nabla}_{coo} f(z_t)  - \hat{\nabla}_{coo} f(z_{t-1}) - \hat{\nabla}_{coo} f_{\xi_t}(z_t) + \hat{\nabla}_{coo} f_{\xi_t}(z_{t-1})\big) \nonumber \\
  & \quad + \rho_t(\hat{\nabla}_{coo} f(z_t)- \hat{\nabla}_{coo} f_{\xi_t}(z_t))\|^2 \nonumber \\
  & = (1-\rho_t)^2\mathbb{E} \|\hat{\nabla}_{coo} f(z_{t-1}) -v_{t-1}\|^2 + (1-\rho_t)^2\mathbb{E} \|\hat{\nabla}_{coo} f(z_t)  - \hat{\nabla}_{coo} f(z_{t-1}) - \hat{\nabla}_{coo} f_{\xi_t}(z_t) + \hat{\nabla}_{coo} f_{\xi_t}(z_{t-1})\|^2 \nonumber \\
  & \quad  + 2\rho_t(1-\rho_t)\langle \hat{\nabla}_{coo} f(z_t)  - \hat{\nabla}_{coo} f(z_{t-1}) - \hat{\nabla}_{coo} f_{\xi_t}(z_t) + \hat{\nabla}_{coo} f_{\xi_t}(z_{t-1}), \hat{\nabla}_{coo} f(z_t)- \hat{\nabla}_{coo} f_{\xi_t}(z_t)\rangle \nonumber \\
  & \quad + \rho_t^2 \mathbb{E} \|\hat{\nabla}_{coo} f(z_t)- \hat{\nabla}_{coo} f_{\xi_t}(z_t)\|^2 \nonumber \\
  & \leq (1-\rho_t)^2\mathbb{E} \|\hat{\nabla}_{coo} f(z_{t-1}) -v_{t-1}\|^2 + 2(1-\rho_t)^2\mathbb{E} \|\hat{\nabla}_{coo} f(z_t)  - \hat{\nabla}_{coo} f(z_{t-1}) - \hat{\nabla}_{coo} f_{\xi_t}(z_t) + \hat{\nabla}_{coo} f_{\xi_t}(z_{t-1})\|^2 \nonumber \\
  & \quad + 2\rho_t^2 \mathbb{E} \|\hat{\nabla}_{coo} f(z_t)- \hat{\nabla}_{coo} f_{\xi_t}(z_t)\|^2 \nonumber \\
  & \leq (1-\rho_t)^2 A_{t-1} + 2(1-\rho_t)^2 \mathbb{E} \| \hat{\nabla}_{coo} f_{\xi_t}(z_t) - \hat{\nabla}_{coo} f_{\xi_t}(z_{t-1})\|^2 + 2\rho_t^2 \mathbb{E} \|\hat{\nabla}_{coo} f(z_t)- \hat{\nabla}_{coo} f_{\xi_t}(z_t)\|^2  \nonumber \\
  & = (1-\rho_t)^2 A_{t-1} + 2(1-\rho_t)^2 \mathbb{E} \| \hat{\nabla}_{coo} f_{\xi_t}(z_t) \!-\! \nabla f_{\xi_t}(z_t) \!+\! \nabla f_{\xi_t}(z_t) \!-\! \nabla f_{\xi_t}(z_{t-1}) \!+\! \nabla f_{\xi_t}(z_{t-1}) \!-\! \hat{\nabla}_{coo} f_{\xi_t}(z_{t-1})\|^2 \nonumber \\
  & \quad + 2\rho_t^2 \mathbb{E} \|\hat{\nabla}_{coo} f(z_t)- \hat{\nabla}_{coo} f_{\xi_t}(z_t)\|^2  \nonumber \\
  & \leq (1-\rho_t)^2 A_{t-1} + 6(1-\rho_t)^2 L^2 \mathbb{E} \| z_t - z_{t-1}\|^2 + 12(1-\rho_t)^2L^2d\mu^2 + 2\rho_t^2 \mathbb{E} \|\hat{\nabla}_{coo} f(z_t)- \hat{\nabla}_{coo} f_{\xi_t}(z_t)\|^2  \nonumber \\
  & \leq (1-\rho_t)^2 A_{t-1} + 24(1-\rho_t)^2 \eta^2 L^2 D^2 + 12(1-\rho_t)^2L^2d\mu^2 + 2\rho_t^2 \mathbb{E} \|\hat{\nabla}_{coo} f(z_t)- \hat{\nabla}_{coo} f_{\xi_t}(z_t)\|^2,
 \end{align}
 where the first inequality holds by Cauchy-Schwarz inequality, and the second inequality follows by the equality $\mathbb{E}\|\zeta-\mathbb{E}[\zeta]\|^2 =
 \mathbb{E}\|\zeta\|^2 - \|\mathbb{E}[\zeta]\|^2$, and the third inequality holds by Cauchy-Schwarz inequality and the Lemma \ref{lem:A3},
 and the last inequality holds by the above Lemma \ref{lem:A1}.
 Next, we consider the upper bound of the term $\mathbb{E} \|\hat{\nabla}_{coo} f(z_t)- \hat{\nabla}_{coo} f_{\xi_t}(z_t)\|^2$. We have
 \begin{align}
  & \mathbb{E} \|\hat{\nabla}_{coo} f(z_t)- \hat{\nabla}_{coo} f_{\xi_t}(z_t)\|^2 \nonumber \\
  & =  \mathbb{E} \|\hat{\nabla}_{coo} f(z_t) - \nabla f(z_t) + \nabla f(z_t) - \nabla f_{\xi_t}(z_t) + \nabla f_{\xi_t}(z_t)- \hat{\nabla}_{coo} f_{\xi_t}(z_t)\|^2 \nonumber \\
  & \leq 3\mathbb{E} \|\hat{\nabla}_{coo} f(z_t) - \nabla f(z_t)\|^2\|^2 + 3\mathbb{E} \|\nabla f(z_t) - \nabla f_{\xi_t}(z_t)\|^2 + 3\mathbb{E} \|\nabla f_{\xi_t}(z_t)- \hat{\nabla}_{coo} f_{\xi_t}(z_t)\|^2 \nonumber \\
  & \leq 3\sigma^2_1 + 3L^2d\mu^2 + 3L^2d\mu^2 = 3\sigma^2_1 + 6L^2d\mu^2,
 \end{align}
 where the first inequality holds by the Young's inequality, and the second inequality holds by Lemma \ref{lem:A3} and Assumption 2.
 Thus, we have
 \begin{align} \label{eq:E1}
  A_t \leq (1-\rho_t)^2 A_{t-1} + 24(1-\rho_t)^2 \eta^2 L^2 D^2 + 12(1-\rho_t)^2L^2d\mu^2 + 6\sigma^2_1\rho_t^2  + 12L^2d\mu^2\rho_t^2.
 \end{align}
 Let $\rho_t = t^{-a}$, $\eta=\eta_t \leq (t+1)^{-a}$ for some $a\in (0,1]$ and $\mu=\mu_t \leq d^{-\frac{1}{2}}(t+1)^{-a}$,
 by \eqref{eq:E1}, we have
 \begin{align}
 A_{t+1} & \leq (1-\frac{1}{(t+1)^a})^2A_t +(1-\frac{1}{(t+1)^a})^2 \frac{24 L^2 D^2+12L^2}{(t+1)^{2a}} + \frac{6\sigma^2_1 + 12L^2d\mu^2}{(t+1)^{2a}}
 \nonumber \\
 & \leq (1-\frac{1}{(t+1)^a})^2A_t + \frac{2 (12L^2 D^2 + 12L^2 + 3\sigma^2_1)}{(t+1)^{2a}}.
 \end{align}
 Here we claim that $A_t \leq C(t+1)^{-a}$, where $C=\frac{2 (12L^2 D^2 + 12L^2 + 3\sigma^2_1)}{2-2^{-a}-a}$, and prove it. Define $h(a)=2-2^{-a}-a$,
 since $h'(a)=2^{-a}\ln(2)-1 \geq 0$, it is easy verified that $2\leq \frac{4}{2-2^{-a}-a}\leq 4$.
 When $t=0$, we have
 \begin{align}
  A_0 =  \mathbb{E}[\|\hat{\nabla}_{coo} f(z_0) - \hat{\nabla}_{coo} f_{\xi_0}(z_0)\|^2] \leq 3\sigma^2_1 + 6L^2 \leq \frac{2 (12L^2 D^2 + 12L^2 + 3\sigma^2_1 )}{2-2^{-a}-a}\leq C\cdot 1^{-a}.
 \end{align}
 When $t=1$, we have
 \begin{align}
  A_1 =  \mathbb{E}[\|\hat{\nabla}_{coo} f(z_1) - \hat{\nabla}_{coo} f_{\xi_1}(z_1)\|^2] \leq 3\sigma^2_1 + 6L^2 \leq \frac{2 (12L^2D^2 + 12L^2 + 3\sigma^2_1 )}{2-2^{-a}-a}/2\leq C\cdot 2^{-a}.
 \end{align}
 Assume that $A_t \leq C(t+1)^{-a}$ for $t\geq 1$, we have
 \begin{align}
 A_{t+1} & \leq (1-\frac{1}{(t+1)^a})^2A_t + \frac{2 (12L^2 D^2 + 12L^2 + 3\sigma^2_1)}{(t+1)^{2a}} \nonumber \\
 &  \leq (1-\frac{1}{(t+1)^a})^2\cdot C(t+1)^{-a} + \frac{2 (12L^2 D^2 + 12L^2 + 3\sigma^2_1)}{(t+1)^{2a}} \nonumber \\
 & \leq C(t+1)^{-a} - 2C(t+1)^{-2a} + C(t+1)^{-3a} + \frac{(2-2^{-a}-a)C}{(t+1)^{2a}} \nonumber \\
 & \leq C(t+1)^{-a} + \frac{-2C + C(t+1)^{-a}+(2-2^{-a}-a)C}{(t+1)^{2a}}\nonumber \\
 & \leq C(t+1)^{-a} + \frac{-2C + C2^{-a}+(2-2^{-a}-a)C}{(t+1)^{2a}}\nonumber \\
 & = C(t+1)^{-a} - \frac{aC }{(t+1)^{2a}}.
 \end{align}
 Define $g(t)=(t+1)^{-a}$ for $a\in (0,1]$. Since $g(t)$ is a convex function, we have $g(t+1)\geq g(t) + g'(t)$, i.e.,
 $(t+2)^{-a}-(t+1)^{-a} \geq -a t^{a+1}$. Then we have
 \begin{align}
 A_{t+1}  \leq C(t+1)^{-a} - \frac{aC }{(t+1)^{2a}} \leq C(t+1)^{-a} - \frac{aC }{(t+1)^{a+1}} \leq C(t+2)^{-a}.
 \end{align}
 Thus we have $A_t=\mathbb{E} \|\hat{\nabla}_{coo} f(z_t) - v_t\|^2 \leq C(t+1)^{-a}, \ a\in (0,1]$ for all $t\geq 0$.

 By Jensen's inequality, we have
 \begin{align}
  \mathbb{E} \|\hat{\nabla}_{coo} f(z_t) - v_t\| \leq \sqrt{\mathbb{E} \|\hat{\nabla}_{coo} f(z_t) - v_t\|^2} \leq \sqrt{C}(t+1)^{-\frac{a}{2}}.
 \end{align}
 Thus we have
 \begin{align}
  \mathbb{E} \|\nabla f(z_t)-v_t\| &=\mathbb{E} \|\nabla f(z_t)- \hat{\nabla}_{coo} f(z_t) + \hat{\nabla}_{coo} f(z_t)- v_t\| \nonumber\\
  & \leq \mathbb{E} \|\nabla f(z_t)- \hat{\nabla}_{coo} f(z_t)\| + \mathbb{E} \|\hat{\nabla}_{coo} f(z_t)- v_t\| \nonumber\\
  & \leq L\sqrt{d}\mu + \sqrt{C}(t+1)^{-\frac{a}{2}},
 \end{align}
 where the last inequality holds by the above Lemma \ref{lem:A3}.
\end{proof}

\begin{theorem} \label{th:F1}
Suppose $\{x_t,y_t,z_t\}^{T-1}_{t=0}$ be generated from Algorithm \ref{alg:2} by using the CooGE zeroth-order gradient estimator. Let $\alpha_t = \frac{1}{t+1}$, $\theta_t = \frac{1}{(t+1)(t+2)}$,
$\eta=\eta_t = T^{-\frac{2}{3}}$, $\gamma_t = (1+\theta_t)\eta_t$, $\rho_t = t^{-\frac{2}{3}}$ for $t\geq 1$
and $\mu=d^{-\frac{1}{2}}T^{-\frac{2}{3}}$, then we have
\begin{align}
\mathbb{E} [\mathcal{G} (z_{\zeta})] =  \frac{1}{T}\sum_{t=1}^{T-1}\mathcal{G}(z_{t}) \leq O(\frac{1}{T^{\frac{1}{3}}}) + O(\frac{\ln(T)}{T^{\frac{4}{3}}}),
\end{align}
where $z_{\zeta}$ is chosen uniformly randomly from $\{z_t\}_{t=0}^{T-1}$.
\end{theorem}

\begin{proof}
Using the Assumption 1, i.e., $f(x)$ is L-smooth, we have
\begin{align} \label{eq:F1}
 f(z_{t+1}) &\le f(z_t) + \langle \nabla f(z_t), z_{t+1} - z_k \rangle + \frac{L}{2} \|z_{t+1} - z_t \|^2 \nonumber \\
&= f(z_t) + (1 -\alpha_{t+1}) \langle \nabla f(z_t), y_{t+1} - z_t\rangle  + \alpha_{t+1} \langle \nabla f(z_t), x_{t+1} - z_t \rangle + \frac{L}{2} \|z_{t+1} - z_t \|^2 \nonumber \\
&= f(z_t) + ((1 -\alpha_{t+1}) \eta_t + \alpha_{t+1} \gamma_t) \langle \nabla f(z_t),  w_t - z_t \rangle + \alpha_{t+1} (1 - \gamma_t) \langle \nabla f(z_t), x_t - z_t \rangle \nonumber \\
&  \quad + \frac{L}{2} \|z_{t+1} - z_t\|^2,
\end{align}
where the first equality holds by $z_{t+1} = (1 - \alpha_{t+1}) y_{t+1} + \alpha_{t+1} x_{t+1}$, and the last equality holds by
$x_{t+1}=x_t+\gamma_t(w_t-x_t)$ and $y_{t+1}=z_t+\eta_t(w_t-z_t)$.

Let $\hat{w}_t = \arg\max_{w\in \mathcal{X}} \langle w, -\nabla f(z_t)\rangle = \arg\min_{w\in \mathcal{X}} \langle w, \nabla f(z_t)\rangle$, we have
\begin{align} \label{eq:F2}
 \langle \nabla f(z_t), w_t - z_t\rangle & = \langle \nabla f(z_t) - v_t, w_t - z_t\rangle + \langle v_t, w_t - z_t\rangle \nonumber \\
 & \leq \langle  \nabla f(z_t) - v_t, w_t - z_t\rangle + \langle v_t, \hat{w}_t - z_t\rangle \nonumber \\
 & = \langle  \nabla f(z_t) - v_t, w_t - \hat{w}_t\rangle + \langle \nabla f(z_t), \hat{w}_t - z_t\rangle \nonumber \\
 & = \langle  \nabla f(z_t) - v_t, w_t - \hat{w}_t\rangle - \mathcal{G}(z_t) \nonumber \\
 & \leq D\|\nabla f(z_t) - v_t\| - \mathcal{G}(z_t),
\end{align}
where the first inequality holds by the step 9 of Algorithm \ref{alg:2}, and the third equality holds by the definition of Frank-Wolfe gap $\mathcal{G}(z_t)$,
and the second inequality follows by Cauchy-Schwarz inequality and Assumption 3.

Next, we consider the upper bound of $\langle \nabla f(z_t), x_t - z_t \rangle$, and we have
\begin{align} \label{eq:F3}
& \langle \nabla f(z_t), x_t - z_t \rangle =  \langle \nabla f(z_t) - \nabla f(z_{t-1}), x_t - z_t \rangle + \langle \nabla f(z_{t-1}), x_t - z_t \rangle \nonumber \\
 & \leq \|\nabla f(z_t) - \nabla f(z_{t-1})\| \|x_t - z_t\| + \langle \nabla f(z_{t-1}), x_t - z_t \rangle \nonumber \\
 & \leq L\|z_t-z_{t-1}\| \|x_t - z_t\| + \langle \nabla f(z_{t-1}), x_t - z_t \rangle \nonumber \\
 & \leq 2 L\eta^2 D^2   + \langle \nabla f(z_{t-1}), x_t - z_t \rangle  \nonumber \\
 & = 2 L\eta^2 D^2 + (1 - \alpha_{t})(1-\gamma_{t-1})\langle \nabla f(z_{t-1}),z_{t-1}-x_{t-1}\rangle + (1 - \alpha_{t})(\eta_{t-1}-\gamma_{t-1})\langle \nabla f(z_{t-1}), w_{t-1}-z_{t-1} \rangle \nonumber \\
 & \leq 2 L\eta^2 D^2 + (1 - \alpha_{t})(1-\gamma_{t-1})\langle \nabla f(z_{t-1}),z_{t-1}-x_{t-1}\rangle
 + (1 - \alpha_{t})\theta_{t-1}\eta_{t-1}\big( D\|\nabla f(z_{t-1}) - v_{t-1}\|  - \mathcal{G}(z_{t-1}) \big)\nonumber \\
  & \leq 2 L\eta^2 D^2  + \langle \nabla f(z_{t-1}),z_{t-1}-x_{t-1}\rangle + \theta_{t-1}\eta_{t-1}D\|\nabla f(z_{t-1}) - v_{t-1}\|
\end{align}
where the first inequality holds by Cauchy-Schwarz inequality, and the third inequality holds by the above Lemma \ref{lem:A1},
and the forth inequality holds by the inequality \eqref{eq:F2},
and the last inequality follows by $\alpha_{t} \in [0,1]$, $\gamma_t\in (0,1)$ and $\mathcal{G}(z_{t-1}) \geq 0$.
By recursion to \eqref{eq:F3}, we have
\begin{align} \label{eq:F4}
 \langle \nabla f(z_t), x_t - z_t \rangle \leq 2 t L\eta^2 D^2 + \eta D\sum_{i=0}^{t-1}\theta_{i}\|\nabla f(z_{i}) - v_{i}\|
\end{align}

Based on the above inequalities \eqref{eq:F1}, \eqref{eq:F2}, \eqref{eq:F4} and the above Lemma \ref{lem:A1}, we have
\begin{align} \label{eq:F5}
 f(z_{t+1}) &\le f(z_t) + ((1 -\alpha_{t+1}) \eta_t + \alpha_{t+1} \gamma_t) \langle \nabla f(z_t),  w_t - z_t \rangle + \alpha_{t+1} (1 - \gamma_t) \langle \nabla f(z_t), x_t - z_t \rangle \nonumber \\
&  \quad + \frac{L}{2} \|z_{t+1} - z_t\|^2 \nonumber \\
& \le f(z_t) +  2L\eta^2D^2 - ((1 -\alpha_{t+1}) \eta_t + \alpha_{t+1} \gamma_t)\mathcal{G}(z_{t}) + ((1 -\alpha_{t+1}) \eta_t + \alpha_{t+1} \gamma_t)D\|\nabla f(z_t) - v_t\| \nonumber \\
& \quad +  2t\alpha_{t+1} (1 - \gamma_t) L\eta^2 D^2 + \alpha_{t+1} (1 - \gamma_t)\eta D\sum_{i=0}^{t-1}\theta_{i}\|\nabla f(z_{i}) - v_{i}\| \nonumber \\
& \le f(z_t) +  2L\eta^2D^2 - \eta \mathcal{G}(z_{t}) + 2\eta D\|\nabla f(z_t) - v_t\| +  2L\eta^2 D^2 + \alpha_{t+1}\eta D\sum_{i=0}^{t-1}\theta_{i}\|\nabla f(z_{i}) - v_{i}\| \nonumber \\
& =  f(z_t) +  4L\eta^2D^2 - \eta \mathcal{G}(z_{t}) + 2\eta D\|\nabla f(z_t) - v_t\| + \alpha_{t+1}\eta D\sum_{i=0}^{t-1}\theta_{i}\|\nabla f(z_{i}) - v_{i}\|.
\end{align}
Summing the inequality \eqref{eq:F5} from $t=0$ to $T-1$, we can obtain
\begin{align}
 \eta \sum_{t=1}^{T-1}\mathcal{G}(z_{t}) & \leq f(z_0) - f(z_{T-1}) +  4TL\eta^2D^2 + 2\eta D \sum_{t=0}^{T-1}\|\nabla f(z_t) - v_t\|
 + \eta D \sum_{t=0}^{T-1}\alpha_{t+1}\sum_{i=0}^{t-1}\theta_{i}\|\nabla f(z_{i}) - v_{i}\| \nonumber \\
 & = f(z_0) - f(z_{T-1}) +  4TL\eta^2D^2 + 2\eta D \sum_{t=0}^{T-1}\|\nabla f(z_t) - v_t\| + \eta D \sum_{t=0}^{T-2}\theta_{t}\big(\sum_{i=t+1}^{T-1}\alpha_{i+1}\big)\|\nabla f(z_{t}) - v_{t}\| \nonumber \\
 & \leq f(z_0) - \inf_{z\in \mathcal{X}} f(z) +  4TL\eta^2D^2 + 2\eta D \sum_{t=0}^{T-1}\|\nabla f(z_t) - v_t\|
 + \eta D \sum_{t=0}^{T-2}\theta_{t}\big(\sum_{i=t+1}^{T-1}\alpha_{i+1}\big)\|\nabla f(z_{t}) - v_{t}\|\nonumber \\
 & \leq \triangle +  4TL\eta^2D^2 + 2\eta D \sum_{t=0}^{T-1}\|\nabla f(z_t) - v_t\| + \eta D \sum_{t=0}^{T-2}\theta_{t}\big(\sum_{i=t+1}^{T-1}\alpha_{i+1}\big)\|\nabla f(z_{t}) - v_{t}\|,
\end{align}
where the last inequality holds by the Assumption 4.
Then we have
\begin{align}
 \frac{1}{T}\sum_{t=1}^{T-1}\mathcal{G}(z_{t}) &\leq \frac{\triangle}{\eta T} +  4L\eta D^2 + \frac{2D}{T} \sum_{t=0}^{T-1}\|\nabla f(z_t) - v_t\| +  \frac{D}{T} \sum_{t=0}^{T-2}\theta_{t}\big(\sum_{i=t+1}^{T-1}\alpha_{i+1}\big)\|\nabla f(z_{t}) - v_{t}\| \nonumber \\
 & \leq \frac{\triangle}{\eta T} +  4L\eta D^2 + 2DL\sqrt{d}\mu + \frac{2D\sqrt{C}}{T} \sum_{t=0}^{T-1}(t+1)^{-\frac{1}{3}} \nonumber \\
 & \quad +  \frac{D}{T} \sum_{t=0}^{T-2}\frac{1}{(t+1)(t+2)}\big(\sum_{i=t+1}^{T-1}\frac{1}{i+1}\big)(L\sqrt{d}\mu + \sqrt{C}(t+1)^{-\frac{1}{3}}) \nonumber \\
 & \leq \frac{\triangle}{\eta T} +  4L\eta D^2 + 2DL\sqrt{d}\mu + \frac{2D\sqrt{C}}{T} \sum_{t=0}^{T-1}(t+1)^{-\frac{1}{3}} \nonumber \\
 & \quad +  \frac{D}{T} \sum_{t=0}^{T-1}\frac{1}{(t+1)(t+2)}\ln(\frac{T}{t+1})(L\sqrt{d}\mu + \sqrt{C}(t+1)^{-\frac{1}{3}}) \nonumber \\
 & \leq \frac{\triangle}{\eta T} +  4L\eta D^2 + 2DL\sqrt{d}\mu + \frac{2D\sqrt{C}}{T} \sum_{t=0}^{T-1}(t+1)^{-\frac{1}{3}} \nonumber \\
 & \quad +  \frac{\ln(T)DL\sqrt{d}\mu}{T} + \frac{\sqrt{C}D\ln(T)}{T} \nonumber \\
 & \leq \frac{\triangle}{\eta T} +  4L\eta D^2 + 2DL\sqrt{d}\mu + \frac{3D\sqrt{C}}{T^{\frac{1}{3}}} +  \frac{\ln(T)DL\sqrt{d}\mu}{T} + \frac{\sqrt{C}D\ln(T)}{T},
\end{align}
where the second inequality holds by Lemma \ref{lem:E1} with $a=\frac{2}{3}$, and the third inequality follows by the inequality $\sum_{i=t+1}^{T-1}\frac{1}{i+1} \leq \int^{T}_{t+1}\frac{1}{x}dx \leq \ln(\frac{T}{t+1})$,
and the fourth inequality holds by $\sum_{t=0}^{T-1} \frac{1}{(t+1)(t+2)}\ln(\frac{T}{t+1}) \leq \ln(T)\sum_{t=0}^{T-1} \frac{1}{(t+1)(t+2)}=\ln(T)(1-\frac{1}{T+1})\leq \ln(T)$ and $(t+1)^{-\frac{1}{3}}\leq 1$,
and the the fifth inequality holds by the inequality $\sum_{t=1}^Tt^{-\frac{1}{3}} \leq \int^T_0x^{-\frac{1}{3}}dx = \frac{3}{2}T^{\frac{2}{3}}$.
Let $\eta = T^{-\frac{2}{3}}$ and $\mu=d^{-\frac{1}{2}}T^{-\frac{2}{3}}$,
then we have
\begin{align}
 \frac{1}{T}\sum_{t=1}^{T-1}\mathcal{G}(z_{t}) \leq \frac{\triangle + 3D\sqrt{C}}{T^{\frac{1}{3}}} +  \frac{4LD^2+2DL}{T^{\frac{2}{3}}} +  \frac{\ln(T)DL}{T^{\frac{5}{3}}} + \frac{\ln(T)\sqrt{C}D}{T^{\frac{4}{3}}}.
\end{align}
Thus, we have
\begin{align}
\mathbb{E} [\mathcal{G} (z_{\zeta})] =  \frac{1}{T}\sum_{t=1}^{T-1}\mathcal{G}(z_{t}) \leq O(\frac{1}{T^{\frac{1}{3}}}) + O(\frac{\ln(T)}{T^{\frac{4}{3}}}),
\end{align}
where $z_{\zeta}$ is chosen uniformly randomly from $\{z_t\}_{t=0}^{T-1}$.
\end{proof}

\subsubsection{ Convergence Analysis of the Acc-SZOFW* (UniGE) Algorithm }
\begin{lemma} \label{lem:G1}
 Suppose the zeroth-order gradient $v_{t} = \hat{\nabla}_{uni} f_{\xi_t}(z_t) + (1-\rho_t)\big(v_{t-1} -\hat{\nabla}_{uni} f_{\xi_t}(z_{t-1})\big)$ be generated from Algorithm \ref{alg:2}.
 Let $\alpha_t = \frac{1}{t+1}$, $\theta_t = \frac{1}{(t+1)(t+2)}$,
 $\gamma_t = (1+\theta_t)\eta_t$, $\eta=\eta_t \leq (t+1)^{-a}$ and $\rho_t = t^{-a}$ for some $a\in (0,1]$ and the smoothing parameter $\beta=\beta_t \leq d^{-1}(t+1)^{-a}$, then
 we have
 \begin{align}
  \mathbb{E} \|v_t - \nabla f(z_t)\| \leq \frac{\beta L d}{2}+ \sqrt{C}(t+1)^{-\frac{a}{2}},
 \end{align}
 where $C=\frac{24L^2dD^2 + 3L^2+ 2\sigma^2_2}{2-2^{-a}-a}$.
\end{lemma}

\begin{proof}
First, we define $f_{\beta}(x)=\mathbb{E}_{u\sim U_B}[f(x+\beta u)]$ be a smooth approximation of $f(x)$,
where $U_B$ is the uniform distribution over the $d$-dimensional unit Euclidean ball $B$. By Lemma 5 in \cite{ji2019improved},
we have $\mathbb{E}_{(u,\xi)}[\hat{\nabla}_{uni} f_{\xi}(x)]=f_{\beta}(x)$.
Next, we will give an upper bound of $ \mathbb{E} \|v_t - \nabla f_{\beta}(z_t)\|$.
 It is easy verified that
 \begin{align}
  v_t-v_{t-1} = -\rho_tv_{t-1} + (1-\rho_t)(\hat{\nabla}_{uni} f_{\xi_t}(z_t) - \hat{\nabla}_{uni} f_{\xi_t}(z_{t-1})) + \rho_t\hat{\nabla}_{uni} f_{\xi_t}(z_t),
 \end{align}
 and $\nabla f_{\beta}(z_t) = \mathbb{E}_{\xi_t} [\hat{\nabla}_{uni} f_{\xi_t}(z_t)]$ and $\mathbb{E}_{\xi_t} [\hat{\nabla}_{uni} f_{\xi_t}(z_t)-\hat{\nabla}_{uni} f_{\xi_t}(z_{t-1})]=
 \nabla f_{\beta}(z_t) - \nabla f_{\beta}(z_{t-1})$.
 Then we have
 \begin{align} \label{eq:G2}
  & A_t = \mathbb{E} \|\nabla f_{\beta}(z_t) - v_t\|^2 = \mathbb{E} \|\nabla f_{\beta}(z_t) - \nabla f_{\beta}(z_{t-1}) + \nabla f_{\beta}(z_{t-1}) - v_{t-1} - (v_t-v_{t-1})\|^2 \nonumber \\
  & =  \mathbb{E} \|\nabla f_{\beta}(z_t) - \nabla f_{\beta}(z_{t-1}) + \nabla f_{\beta}(z_{t-1}) \!-\! v_{t-1} \!+\! \rho_tv_{t-1} \!-\! (1\!-\!\rho_t)(\hat{\nabla}_{uni} f_{\xi_t}(z_t) \!-\! \hat{\nabla}_{uni} f_{\xi_t}(z_{t-1})) \!-\! \rho_t\hat{\nabla}_{uni} f_{\xi_t}(z_t)\|^2 \nonumber \\
  & = \mathbb{E} \|(1-\rho_t)(\nabla f_{\beta}(z_{t-1}) -v_{t-1}) + (1-\rho_t)\big(\nabla f_{\beta}(z_t) - \nabla f_{\beta}(z_{t-1}) - \hat{\nabla}_{uni} f_{\xi_t}(z_t) + \hat{\nabla}_{uni} f_{\xi_t}(z_{t-1})\big) \nonumber \\
  & \quad + \rho_t(\nabla f_{\beta}(z_{t})- \hat{\nabla}_{uni} f_{\xi_t}(z_t))\|^2 \nonumber \\
  & = (1-\rho_t)^2\mathbb{E} \|\nabla f_{\beta}(z_{t-1}) -v_{t-1}\|^2 + (1-\rho_t)^2\mathbb{E} \|\nabla f_{\beta}(z_t) - \nabla f_{\beta}(z_{t-1}) - \hat{\nabla}_{uni} f_{\xi_t}(z_t) + \hat{\nabla}_{uni} f_{\xi_t}(z_{t-1})\|^2 \nonumber \\
  & \quad  + 2\rho_t(1-\rho_t)\langle \nabla f_{\beta}(z_t) - \nabla f_{\beta}(z_{t-1}) - \hat{\nabla}_{uni} f_{\xi_t}(z_t) + \hat{\nabla}_{uni} f_{\xi_t}(z_{t-1}), \nabla f_{\beta}(z_t) - \hat{\nabla}_{uni} f_{\xi_t}(z_t)\rangle \nonumber \\
  & \quad + \rho_t^2 \mathbb{E} \|\nabla f_{\beta}(z_t) - \hat{\nabla}_{uni} f_{\xi_t}(z_t)\|^2 \nonumber \\
  & \leq (1-\rho_t)^2\mathbb{E} \|\nabla f_{\beta}(z_{t-1}) -v_{t-1}\|^2 + 2(1-\rho_t)^2\mathbb{E} \|\nabla f_{\beta}(z_t) - \nabla f_{\beta}(z_{t-1}) - \hat{\nabla}_{uni} f_{\xi_t}(z_t) + \hat{\nabla}_{uni} f_{\xi_t}(z_{t-1})\|^2 \nonumber \\
  & \quad + 2\rho_t^2 \mathbb{E} \|\nabla f_{\beta}(z_t) - \hat{\nabla}_{uni} f_{\xi_t}(z_t)\|^2 \nonumber \\
  & \leq (1-\rho_t)^2 A_{t-1} + 2(1-\rho_t)^2 \mathbb{E} \| \hat{\nabla}_{uni} f_{\xi_t}(z_t) - \hat{\nabla}_{uni} f_{\xi_t}(z_{t-1})\|^2 + 2\rho_t^2\sigma^2_2 \nonumber \\
  & \leq (1-\rho_t)^2 A_{t-1} + 2(1-\rho_t)^2 \big( 3dL^2\mathbb{E} \| z_t - z_{t-1}\|^2 + \frac{3L^2d^2\beta^2}{2}\big)+ 2\rho_t^2\sigma^2_2 \nonumber \\
  & \leq (1-\rho_t)^2 A_{t-1} + 24(1-\rho_t)^2 \eta^2 d L^2 D^2 + 3(1-\rho_t)^2 L^2d^2\beta^2 + 2\rho_t^2\sigma^2_2,
 \end{align}
 where the first inequality holds by Cauchy-Schwarz inequality, and the second inequality follows by the equality $\mathbb{E}\|\zeta-\mathbb{E}[\zeta]\|^2 =
 \mathbb{E}\|\zeta\|^2 - \|\mathbb{E}[\zeta]\|^2$, and the Lemma \ref{lem:A4},
 and the last inequality holds by the above Lemma \ref{lem:A1}.

 Let $\rho_t = t^{-a}$, $\eta=\eta_t \leq (t+1)^{-a}$ for some $a\in (0,1]$ and $\beta \leq d^{-1}(t+1)^{-a}$,
  by \eqref{eq:G2}, we have
 \begin{align}
 A_{t+1} & \leq (1-\frac{1}{(t+1)^a})^2A_t +(1-\frac{1}{(t+1)^a})^2 \frac{24L^2dD^2 + 3L^2}{(t+1)^{2a}} + \frac{2\sigma^2_2}{(t+1)^{2a}} \nonumber \\
 & \leq (1-\frac{1}{(t+1)^a})^2A_t + \frac{24L^2dD^2 + 3L^2 + 2\sigma^2_2}{(t+1)^{2a}}.
 \end{align}
 Here we claim that $A_t \leq C(t+1)^{-a}$, where $C=\frac{24L^2dD^2 + 3L^2+ 2\sigma^2_2}{2-2^{-a}-a}$, and prove it. Define $h(a)=2-2^{-a}-a$,
 since $h'(a)=2^{-a}\ln(2)-1 \geq 0$, it is easy verified that $2\leq \frac{4}{2-2^{-a}-a}\leq 4$.
 When $t=0$, we have
 \begin{align}
  A_0 =  \mathbb{E}[\|\nabla f_{\beta}(z_0) - \hat{\nabla}_{uni} f_{\xi_0}(z_0)\|^2] \leq \sigma^2_2 \leq \frac{24L^2dD^2 + 3L^2+ 2\sigma^2_2}{2-2^{-a}-a}\leq C\cdot 1^{-a}.
 \end{align}
 When $t=1$, we have
 \begin{align}
  A_1 =  \mathbb{E}[\|\nabla f_{\beta}(z_1) - \hat{\nabla}_{uni} f_{\xi_1}(z_1)\|^2] \leq \sigma^2_2 \leq \frac{24L^2dD^2 + 3L^2+ 2\sigma^2_2}{2-2^{-a}-a}/2\leq C\cdot 2^{-a}.
 \end{align}
 Assume that $A_t \leq C(t+1)^{-a}$ for $t\geq 1$, we have
 \begin{align}
 A_{t+1} & \leq (1-\frac{1}{(t+1)^a})^2A_t + \frac{24L^2dD^2 + 3L^2+ 2\sigma^2_2}{(t+1)^{2a}} \nonumber \\
 &  \leq (1-\frac{1}{(t+1)^a})^2\cdot C(t+1)^{-a} + \frac{24L^2dD^2 + 3L^2+ 2\sigma^2_2}{(t+1)^{2a}} \nonumber \\
 & \leq C(t+1)^{-a} - 2C(t+1)^{-2a} + C(t+1)^{-3a} + \frac{(2-2^{-a}-a)C}{(t+1)^{2a}} \nonumber \\
 & \leq C(t+1)^{-a} + \frac{-2C + C(t+1)^{-a}+(2-2^{-a}-a)C}{(t+1)^{2a}}\nonumber \\
 & \leq C(t+1)^{-a} + \frac{-2C + C2^{-a}+(2-2^{-a}-a)C}{(t+1)^{2a}}\nonumber \\
 & = C(t+1)^{-a} - \frac{aC }{(t+1)^{2a}}.
 \end{align}
 Define $g(t)=(t+1)^{-a}$ for $a\in (0,1]$. Since $g(t)$ is a convex function, we have $g(t+1)\geq g(t) + g'(t)$, i.e.,
 $(t+2)^{-a}-(t+1)^{-a} \geq -a t^{a+1}$. Then we have
 \begin{align}
 A_{t+1}  \leq C(t+1)^{-a} - \frac{aC }{(t+1)^{2a}} \leq C(t+1)^{-a} - \frac{aC }{(t+1)^{a+1}} \leq C(t+2)^{-a}.
 \end{align}
 Thus we have $A_t=\mathbb{E} \|\nabla f_{\beta}(z_t) - v_t\|^2 \leq C(t+1)^{-a}, \ a\in (0,1]$ for all $t\geq 0$.

 By Jensen's inequality, we have
 \begin{align}
  \mathbb{E} \|\nabla f_{\beta}(z_t) - v_t\| \leq \sqrt{\mathbb{E} \|\nabla f_{\beta}(z_t) - v_t\|^2} \leq \sqrt{C}(t+1)^{-\frac{a}{2}}.
 \end{align}
 Thus we have
 \begin{align}
  \mathbb{E} \|\nabla f(z_t)-v_t\| &=\mathbb{E} \|\nabla f(z_t)- \nabla f_{\beta}(z_t) + \nabla f_{\beta}(z_t)- v_t\| \nonumber\\
  & \leq \mathbb{E} \|\nabla f(z_t)- \nabla f_{\beta}(z_t)\| + \mathbb{E} \|\nabla f_{\beta}(z_t)- v_t\| \nonumber\\
  & \leq \frac{\beta L d}{2}+ \sqrt{C}(t+1)^{-\frac{a}{2}},
 \end{align}
 where the last inequality holds by the Lemma \ref{lem:A4}.
\end{proof}

\begin{theorem}
Suppose $\{x_t,y_t,z_t\}^{T-1}_{t=0}$ be generated from Algorithm \ref{alg:2} by using the UniGE zeroth-order gradient estimator.
Let $\alpha_t = \frac{1}{t+1}$, $\theta_t = \frac{1}{(t+1)(t+2)}$,
$\eta=\eta_t = T^{-\frac{2}{3}}$, $\gamma_t = (1+\theta_t)\eta_t$, $\rho_t = t^{-\frac{2}{3}}$ for $t\geq 1$ and $\beta=d^{-1}T^{-\frac{2}{3}}$, then we have
 \begin{align}
 \mathbb{E} [\mathcal{G} (z_{\zeta})] \leq \frac{1}{T}\sum_{t=1}^{T-1}\mathcal{G}(z_{t}) \leq O(\frac{\sqrt{d}}{T^{\frac{1}{3}}})+ O(\frac{\sqrt{d}\ln(T)}{T^{\frac{4}{3}}}),
\end{align}
where $z_{\zeta}$ is chosen uniformly randomly from $\{z_t\}_{t=0}^{T-1}$.
\end{theorem}

\begin{proof}
 This proof can follow the proof of Theorem \ref{th:F1}. Here we show some different results.
 Similarly, we have
 \begin{align}
 \frac{1}{T}\sum_{t=1}^{T-1}\mathcal{G}(z_{t}) &\leq \frac{\triangle}{\eta T} +  4L\eta D^2 + \frac{2D}{T} \sum_{t=0}^{T-1}\|\nabla f(z_t) - v_t\| +  \frac{D}{T} \sum_{t=0}^{T-2}\theta_{t}\big(\sum_{i=t+1}^{T-1}\alpha_{i+1}\big)\|\nabla f(z_{t}) - v_{t}\| \nonumber \\
 & \leq \frac{\triangle}{\eta T} +  4L\eta D^2 + DLd\beta + \frac{2D\sqrt{C}}{T} \sum_{t=0}^{T-1}(t+1)^{-\frac{1}{3}} \nonumber \\
 & \quad +  \frac{D}{T} \sum_{t=0}^{T-2}\frac{1}{(t+1)(t+2)}\big(\sum_{i=t+1}^{T-1}\frac{1}{i+1}\big)(\frac{\beta L d}{2} + \sqrt{C}(t+1)^{-\frac{1}{3}}) \nonumber \\
 & \leq \frac{\triangle}{\eta T} +  4L\eta D^2 + DLd\beta + \frac{2D\sqrt{C}}{T} \sum_{t=0}^{T-1}(t+1)^{-\frac{1}{3}} \nonumber \\
 & \quad +  \frac{D}{T} \sum_{t=0}^{T-1}\frac{1}{(t+1)(t+2)}\ln(\frac{T}{t+1})(\frac{\beta L d}{2} + \sqrt{C}(t+1)^{-\frac{1}{3}}) \nonumber \\
 & \leq \frac{\triangle}{\eta T} +  4L\eta D^2 + DLd\beta + \frac{2D\sqrt{C}}{T} \sum_{t=0}^{T-1}(t+1)^{-\frac{1}{3}} \nonumber \\
 & \quad +  \frac{\ln(T)DLd\beta}{2T} + \frac{\sqrt{C}D}{T^2}\big(\sum_{t=0}^{T-1}\frac{1}{(t+1)(t+2)}\ln(\frac{T}{t+1})\big) \big(\sum_{t=0}^{T-1}(t+1)^{-\frac{1}{3}} \big)\nonumber \\
 & \leq \frac{\triangle}{\eta T} +  4L\eta D^2 + DLd\beta + \frac{3D\sqrt{C}}{T^{\frac{1}{3}}} +  \frac{\ln(T)DLd\beta}{2T} + \frac{\sqrt{C}D\ln(T)}{T^{\frac{4}{3}}},
\end{align}
where the second inequality holds by Lemma \ref{lem:G1} with $a=\frac{2}{3}$, and the third inequality follows by the inequality $\sum_{i=t+1}^{T-1}\frac{1}{i+1} \leq \int^{T}_{t+1}\frac{1}{x}dx \leq \ln(\frac{T}{t+1})$,
and the fourth inequality holds by $\sum_{t=0}^{T-1} \frac{1}{(t+1)(t+2)}\ln(\frac{T}{t+1}) \leq \ln(T)\sum_{t=0}^{T-1} \frac{1}{(t+1)(t+2)}=\ln(T)(1-\frac{1}{T+1})\leq \ln(T)$,
and the the fifth inequality holds by the inequality $\sum_{t=1}^Tt^{-\frac{1}{3}} \leq \int^T_0x^{-\frac{1}{3}}dx = \frac{3}{2}T^{\frac{2}{3}}$.
Let $\eta = T^{-\frac{2}{3}}$ and $\mu=d^{-1}T^{-\frac{2}{3}}$,
then we have
 \begin{align}
 \frac{1}{T}\sum_{t=1}^{T-1}\mathcal{G}(z_{t}) \leq \frac{\triangle + 3D\sqrt{C}}{T^{\frac{1}{3}}} +  \frac{4LD^2+DL}{T^{\frac{2}{3}}} +  \frac{\ln(T)DL}{2T^{\frac{5}{3}}} + \frac{\ln(T)\sqrt{C}D}{T^{\frac{4}{3}}}.
\end{align}
Since $a=\frac{2}{3}$ and $C=\frac{24L^2dD^2 + 3L^2+ 2\sigma^2_2}{2-2^{-a}-a}=O(d)$, we have
 \begin{align}
 \mathbb{E} [\mathcal{G} (z_{\zeta})] =\frac{1}{T}\sum_{t=1}^{T-1}\mathcal{G}(z_{t}) \leq O(\frac{\sqrt{d}}{T^{\frac{1}{3}}})+ O(\frac{\sqrt{d}\ln(T)}{T^{\frac{4}{3}}}),
\end{align}
where $z_{\zeta}$ is chosen uniformly randomly from $\{z_t\}_{t=0}^{T-1}$.
\end{proof}

\subsection{Application: black-box adversarial attack}
\label{Appendix:A3}
\textbf{Details of pre-trained models. }
In the experiment, we use the pre-trained DNN models on MNIST and CIFAR10 datasets as the target black-box models.
For MNIST dataset, we attack a pre-trained 4-layer CNN: 2 convolutional layers followed by 2 fully-connected layers with ReLU activation and max-pooling applied after the convolutional layers,
which can achieves 99.16\% test accuracy on natural examples. For CIFAR10 dataset, we attack a pre-trained ResNet18 model \cite{he2016deep}, which can achieves 93.07\% test accuracy on natural examples.

\textbf{Parameter setting of the SAP problem.}
In the SAP experiment, we choose $\varepsilon=0.3$ for MNIST dataset and $\varepsilon=0.1$ for CIFAR10 dataset. For fair comparison, we choose step size $\eta=1/\sqrt{T}$ and the smoothing parameter $\delta=\mu=0.01$ for both FW-Black algorithm and Acc-ZO-FW algorithm. We set the number of samples of random gradient estimators to be $d$ in FW-Black, where $d=28\times28$ and $d=3\times32\times32$ for MNIST and CIFAR10. We set $\alpha_{t}=\frac{1}{t+1}$, $\theta_t=\frac{1}{(t+1)(t+2)}$ and $\gamma_{t}=2(1+\theta_{t}) \eta_{t}$ in Acc-ZO-FW as our theory suggested. The total iteration T is set to be 1000 for both datasets. For MNIST dataset, we randomly choose 1000 images correctly classified by the pre-trained model from the same class, so $n=1$ in the problem \eqref{eq:ap} and we run the experiment 1000 times. For CIFAR10 dataset, we randomly choose 100 images correctly classified from the same class, so $n=1$ in the problem \eqref{eq:ap} and we run the experiment 100 times with different images.

\textbf{Parameter setting of the UAP problem.}
In the UAP experiment, we choose $\varepsilon=0.3$ for both MNIST dataset and CIFAR10 dataset.
For fair comparison, we choose the mini-batch size $b=20$ for all stochastic zeroth-order methods, the number of samples of random gradient estimators is 1 in both our accelerated algorithms and the existing algorithms.
We choose step size $\eta=1/\sqrt{T}$ in ZSCG and Acc-SZOFW, $\eta=T^{-\frac{3}{4}}$ in ZO-SFW and $\eta=T^{-\frac{2}{3}}$ in Acc-SZOFW* as their theories suggested, respectively. We set $b_2=q=20$, $b_1=300$, $\alpha_{t}=\frac{1}{t+1}$, $\theta_t=\frac{1}{(t+1)(t+2)}$, $\gamma_{t}=(1+\theta_{t}) \eta_{t}$, $\mu=d^{-\frac{1}{2}} T^{-\frac{1}{2}}$ and $\beta=d^{-1} T^{-\frac{1}{2}}$ in Acc-SZOFW as our theory suggested. We set $\alpha_{t}=\frac{1}{t+1}$, $\theta_t=\frac{1}{(t+1)(t+2)}$, $\gamma_{t}=6(1+\theta_{t}) \eta_{t}$, $\mu=d^{-\frac{1}{2}} T^{-\frac{2}{3}}$ and $\beta=d^{-1} T^{-\frac{2}{3}}$ in Acc-SZOFW* as our theory suggested. The total iteration T is set to be 1000 for MNIST and 5000 for CIFAR10. For both datasets, we randomly choose 300 images correctly classified by their corresponding pre-trained models from the same class, so $n=300$ in the problem \ref{eq:ap} and we run all the stochastic zeroth-order algorithms once.

\textbf{Generated adversarial examples of the SAP problem.}
Figure \ref{fig:sap-adversarial-images} displays the original images, single adversarial perturbations and adversarial images generated by the Acc-ZO-FW algorithm. The pixel values of the showed adversarial perturbations are normalized to $[0,1]^{d}$ by min-max normalization, respectively.

\begin{figure*}[htb]
    \vskip 0.2in
    \begin{center}
        \subfigure[MNIST]{
            \begin{minipage}[b]{0.35\columnwidth}
                \includegraphics[width=1\textwidth]{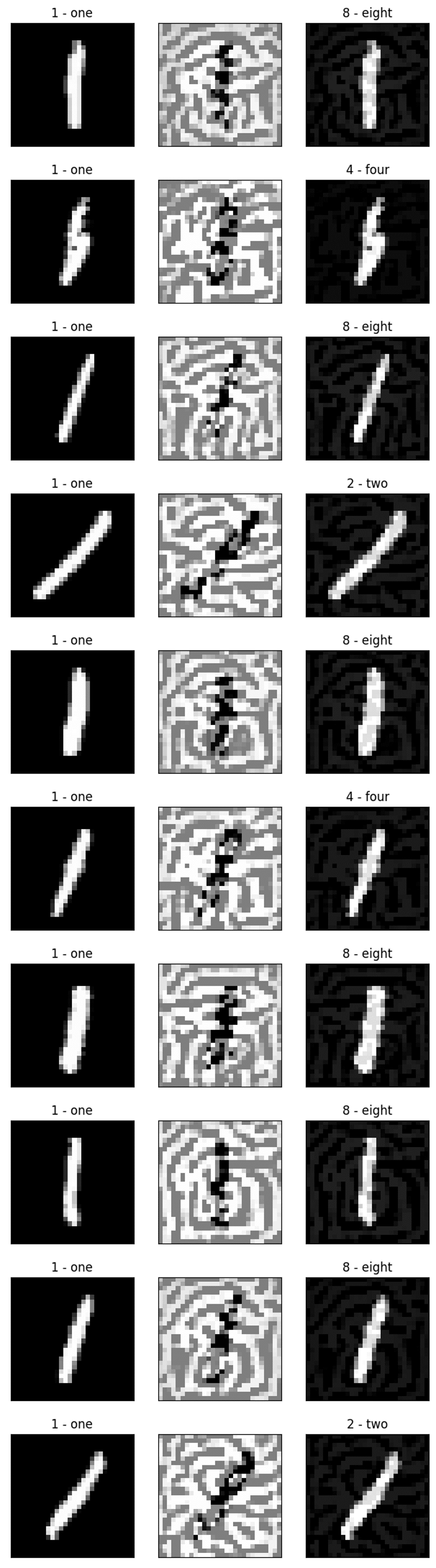}
            \end{minipage}
            \label{fig:mnist_sap}
        }
        \subfigure[CIFAR10]{
            \begin{minipage}[b]{0.35\columnwidth}
                \includegraphics[width=1\textwidth]{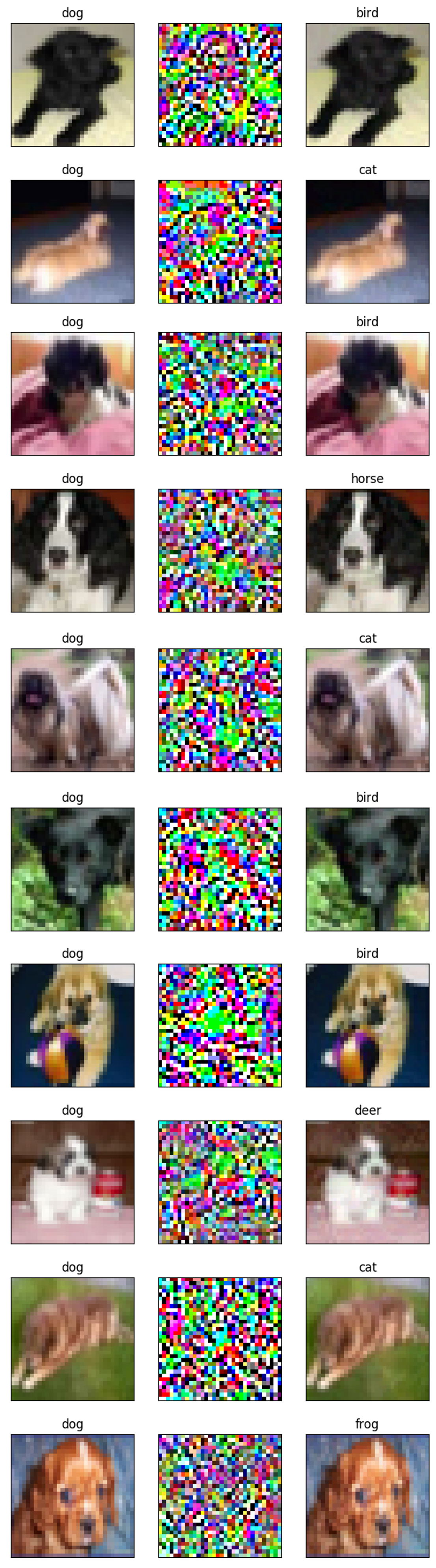}
            \end{minipage}
            \label{fig:cifar10_sap}
        }
	\caption{Generated adversarial examples on MNIST (class '1') and CIFAR10 (class 'dog') for the SAP problem. In each dataset, the first column represents the original images, the second column represents the single adversarial perturbations, the third column represents the adversarial images. The model predictions of these images are showed above each image. }
    \label{fig:sap-adversarial-images}
    \end{center}
    \vskip -0.2in
\end{figure*}

\textbf{Generated adversarial examples of the UAP problem.}
Figure \ref{fig:uap-adversarial-images-mnist} and Figure \ref{fig:uap-adversarial-images-cifar10} displays the original images, universal adversarial perturbations and adversarial images generated by our proposed accelerated zeroth-order algorithms. The pixel values of the showed adversarial perturbations are normalized to $[0,1]^{d}$ by min-max normalization, respectively.

\begin{figure*}[htb]
    \vskip 0.2in
    \begin{center}
        \subfigure[Acc-SZOFW (UniGE)]{
            \begin{minipage}[b]{0.8\columnwidth}
                \includegraphics[width=1\textwidth]{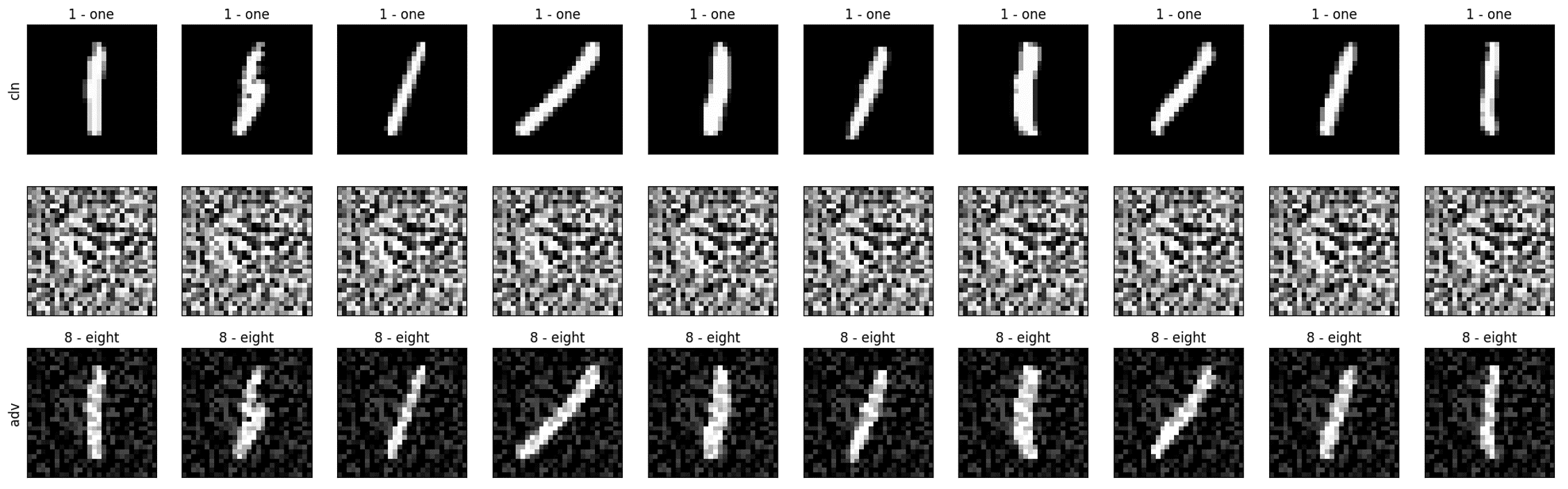}
            \end{minipage}
            \label{fig:mnist-uap-accszofw-unige}
        }
        \subfigure[Acc-SZOFW (CooGE)]{
            \begin{minipage}[b]{0.8\columnwidth}
                \includegraphics[width=1\textwidth]{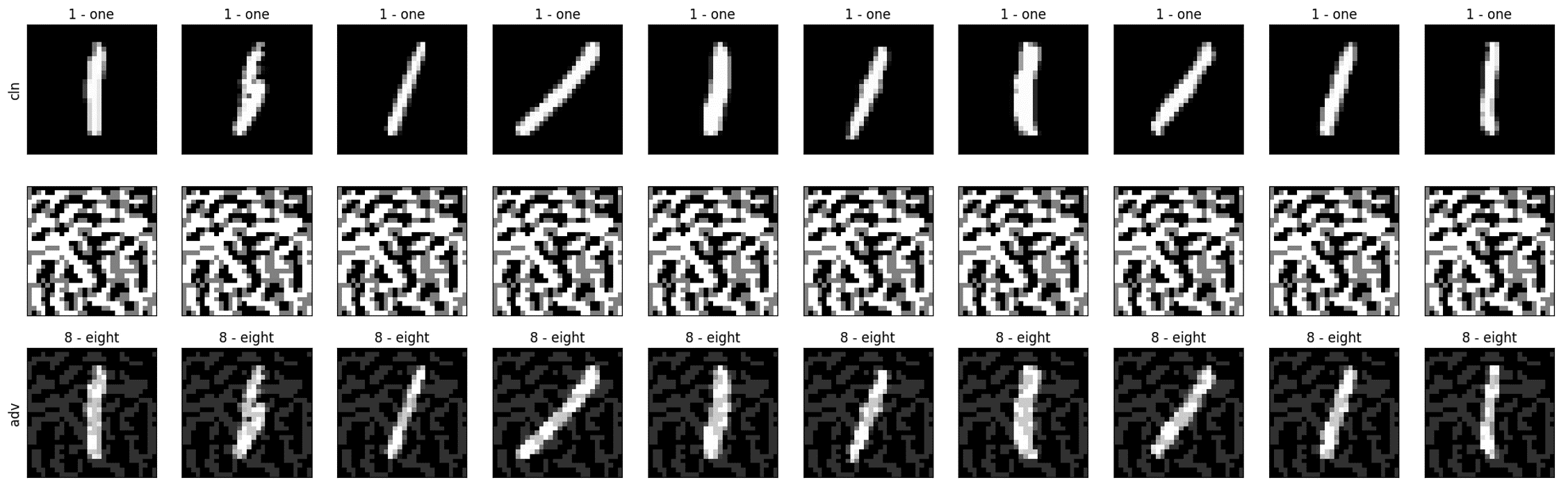}
            \end{minipage}
            \label{fig:mnist-uap-accszofw-cooge}
        }
        \subfigure[Acc-SZOFW* (UniGE)]{
            \begin{minipage}[b]{0.8\columnwidth}
                \includegraphics[width=1\textwidth]{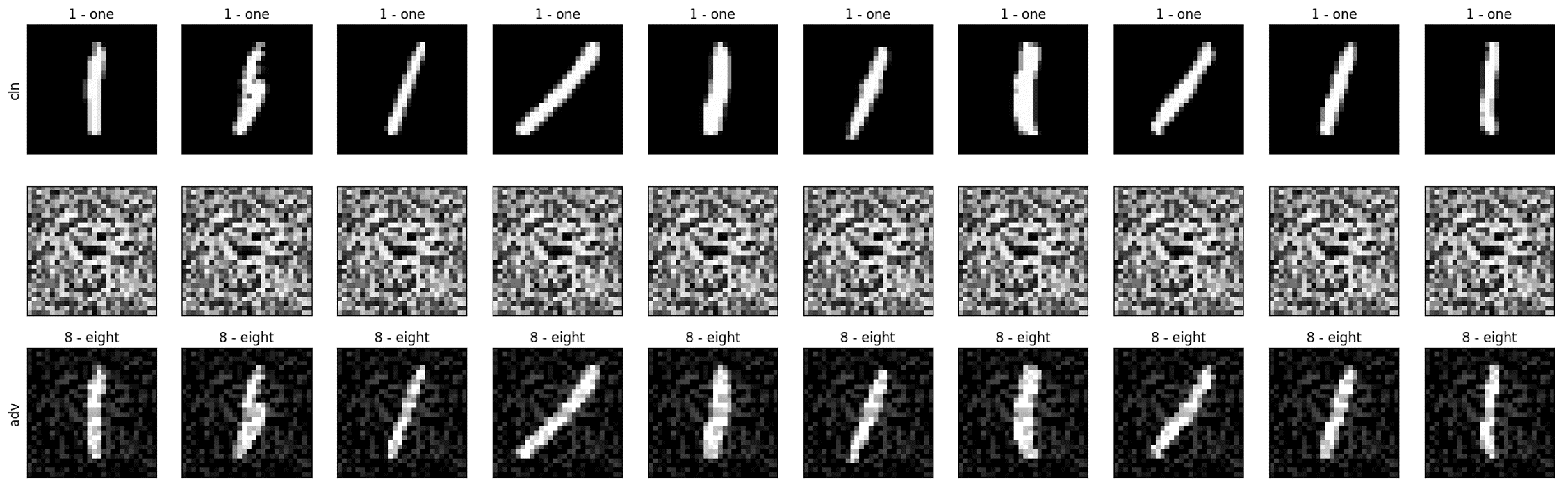}
            \end{minipage}
            \label{fig:mnist-uap-accszofwp-unige}
        }
        \subfigure[Acc-SZOFW* (CooGE)]{
            \begin{minipage}[b]{0.8\columnwidth}
                \includegraphics[width=1\textwidth]{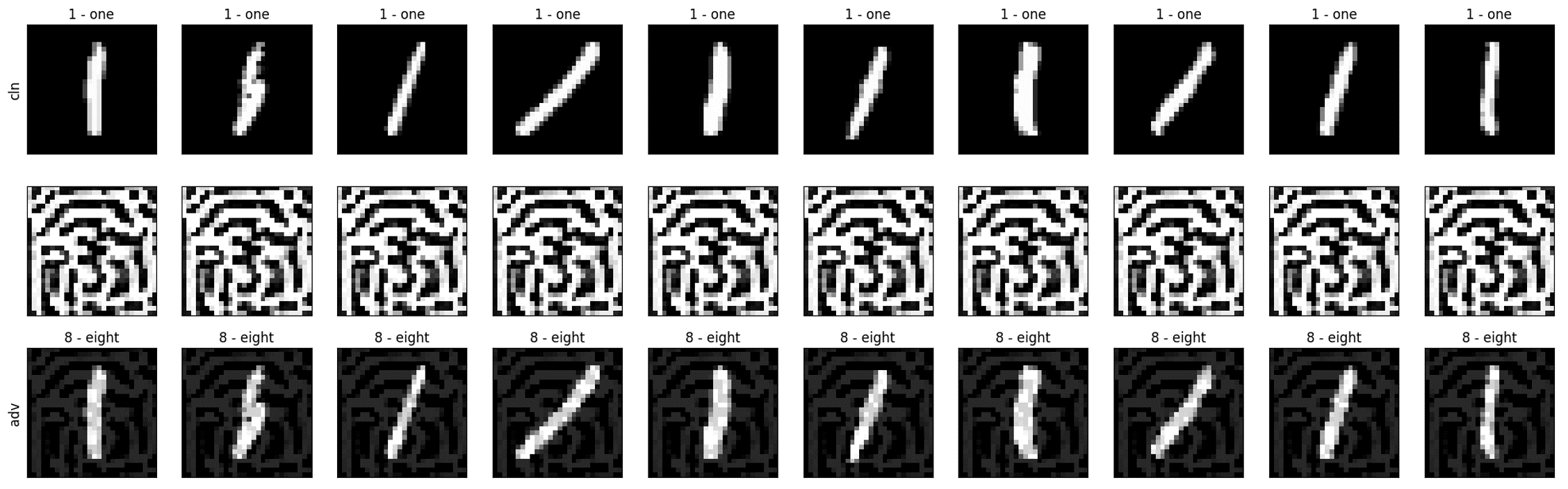}
            \end{minipage}
            \label{fig:mnist-uap-accszofwp-cooge}
        }
	\caption{Generated adversarial examples on MNIST (class '1') for the UAP problem. In each sub-figure, the first row represents the original images, the second row represents the universal adversarial perturbations, the third row represents the adversarial images. The model predictions of these images are showed above each image. }
    \label{fig:uap-adversarial-images-mnist}
    \end{center}
    \vskip -0.2in
\end{figure*}

\begin{figure*}[htb]
    \vskip 0.2in
    \begin{center}
        \subfigure[Acc-SZOFW (UniGE)]{
            \begin{minipage}[b]{0.8\columnwidth}
                \includegraphics[width=1\textwidth]{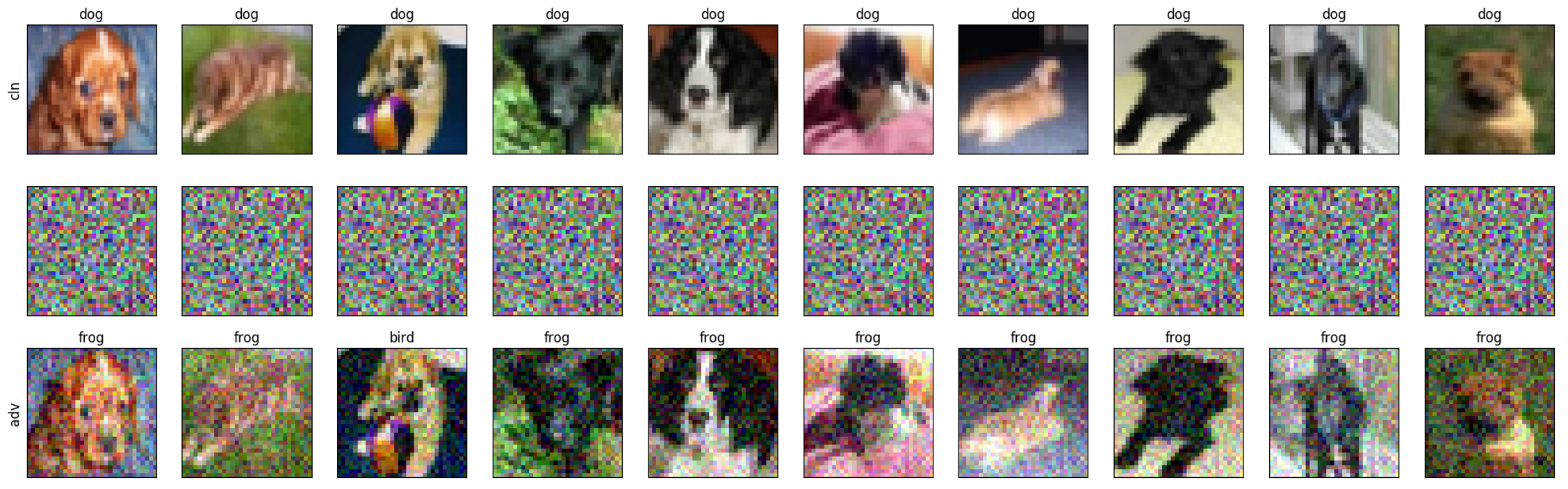}
            \end{minipage}
            \label{fig:cifar10-uap-accszofw-unige}
        }
        \subfigure[Acc-SZOFW (CooGE)]{
            \begin{minipage}[b]{0.8\columnwidth}
                \includegraphics[width=1\textwidth]{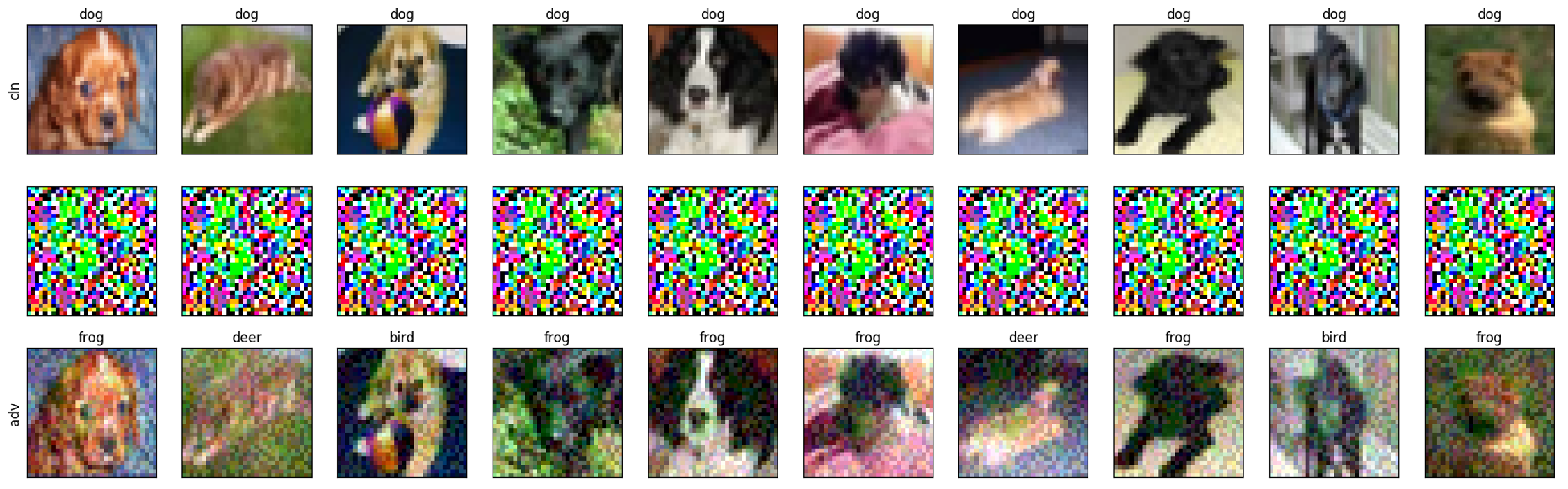}
            \end{minipage}
            \label{fig:cifar10-uap-accszofw-cooge}
        }
        \subfigure[Acc-SZOFW* (UniGE)]{
            \begin{minipage}[b]{0.8\columnwidth}
                \includegraphics[width=1\textwidth]{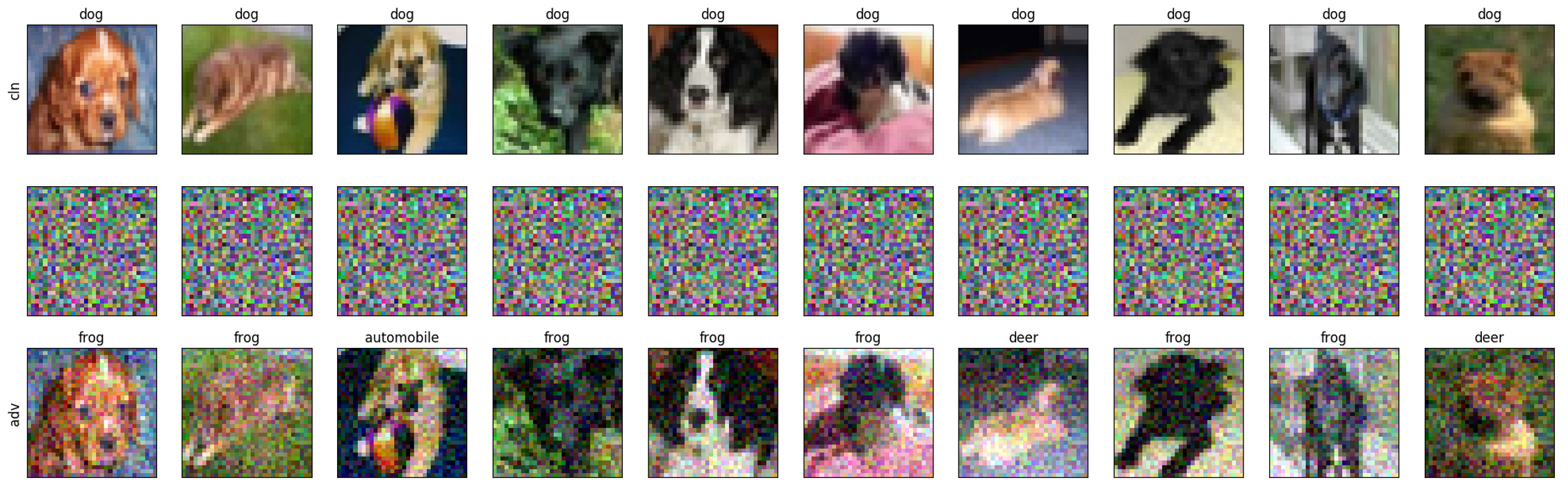}
            \end{minipage}
            \label{fig:cifar10-uap-accszofwp-unige}
        }
        \subfigure[Acc-SZOFW* (CooGE)]{
            \begin{minipage}[b]{0.8\columnwidth}
                \includegraphics[width=1\textwidth]{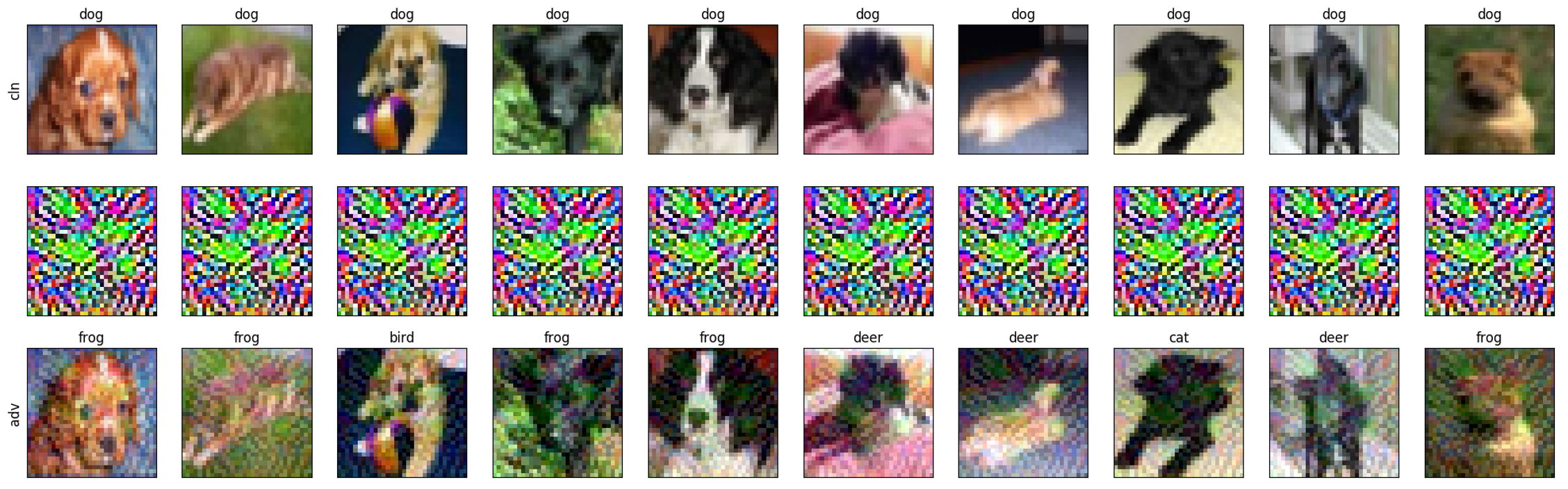}
            \end{minipage}
            \label{fig:cifar10-uap-accszofwp-cooge}
        }
	\caption{Generated adversarial examples on CIFAR10 (class 'dog') for the UAP problem. In each sub-figure, the first row represents the original images, the second row represents the universal adversarial perturbations, the third row represents the adversarial images. The model predictions of these images are showed above each image. }
    \label{fig:uap-adversarial-images-cifar10}
    \end{center}
    \vskip -0.2in
\end{figure*}

\subsection{Application: robust black-box binary classification}
\label{Appendix:A4}
\begin{figure*}[htb]
    \vskip 0.2in
    \begin{center}
        \subfigure[phishing]{
            \begin{minipage}[b]{0.22\textwidth}
                \includegraphics[width=1\textwidth]{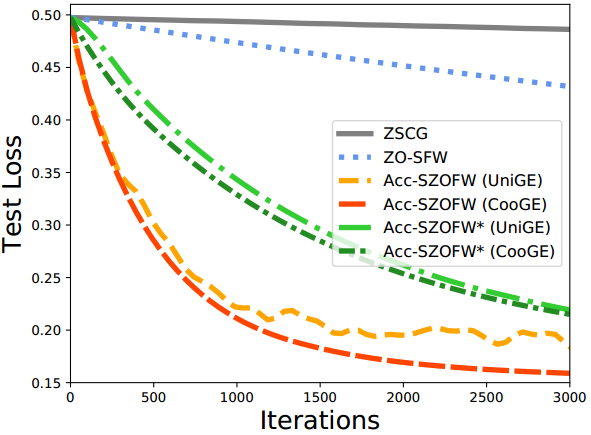}
            \end{minipage}
            \label{fig:rlc_phishing_iteration_test}
        }
        \subfigure[a9a]{
            \begin{minipage}[b]{0.22\textwidth}
                \includegraphics[width=1\textwidth]{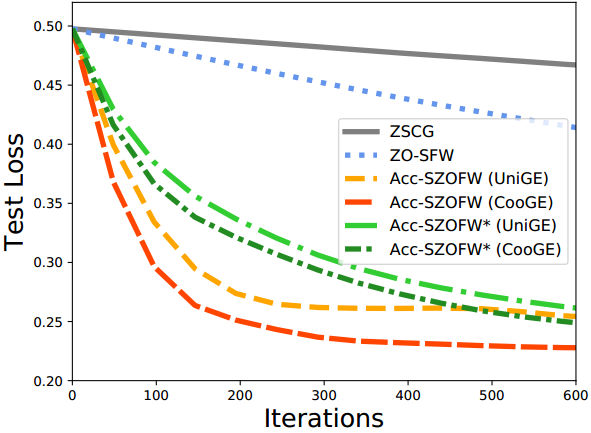}
            \end{minipage}
            \label{fig:rlc_a9a_iteration_test}
        }
        \subfigure[w8a]{
            \begin{minipage}[b]{0.22\textwidth}
                \includegraphics[width=1\textwidth]{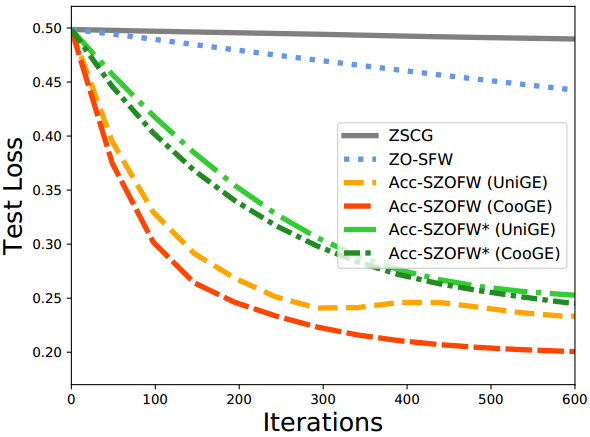}
            \end{minipage}
            \label{fig:rlc_w8a_iteration_test}
        }
        \subfigure[covtype.binary]{
            \begin{minipage}[b]{0.22\textwidth}
                \includegraphics[width=1\textwidth]{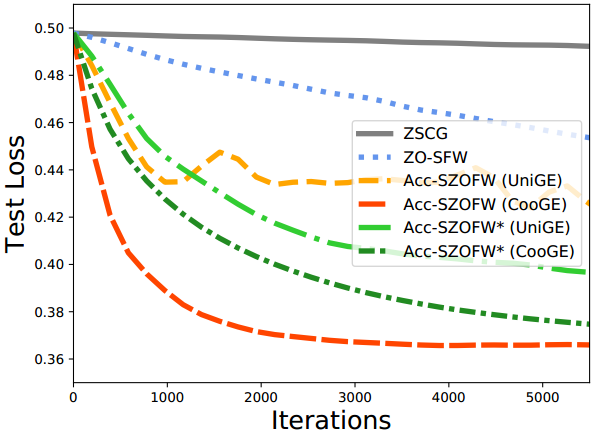}
            \end{minipage}
            \label{fig:rlc_covtype_iteration_test}
        }

        \subfigure[phishing]{
            \begin{minipage}[b]{0.22\textwidth}
                \includegraphics[width=1\textwidth]{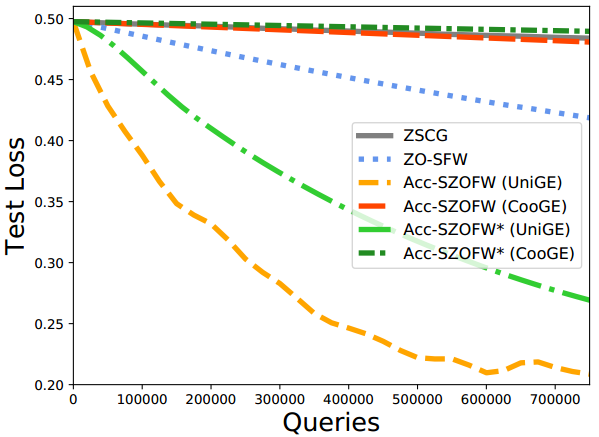}
            \end{minipage}
            \label{fig:rlc_phishing_queries_test}
        }
        \subfigure[a9a]{
            \begin{minipage}[b]{0.22\textwidth}
                \includegraphics[width=1\textwidth]{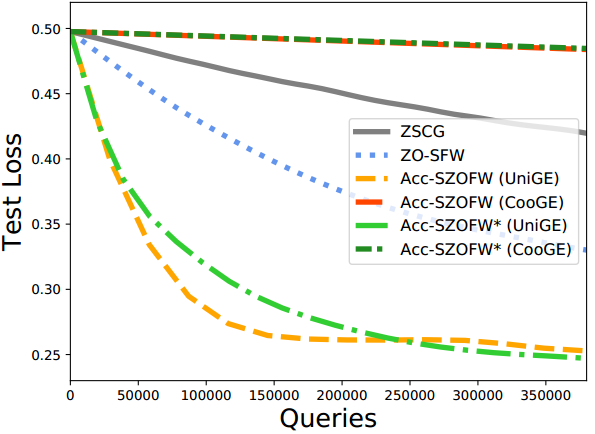}
            \end{minipage}
            \label{fig:rlc_a9a_queries_test}
        }
        \subfigure[w8a]{
            \begin{minipage}[b]{0.22\textwidth}
                \includegraphics[width=1\textwidth]{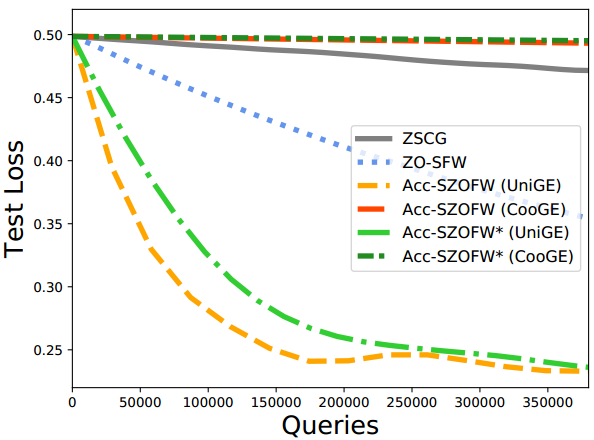}
            \end{minipage}
            \label{fig:rlc_w8a_queries_test}
        }
        \subfigure[covtype.binary]{
            \begin{minipage}[b]{0.22\textwidth}
                \includegraphics[width=1\textwidth]{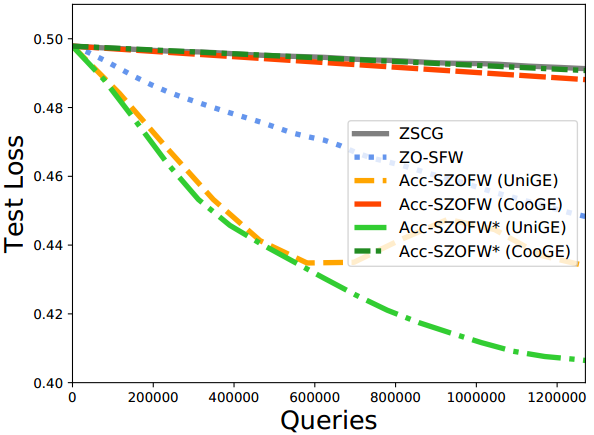}
            \end{minipage}
            \label{fig:rlc_covtype_queries_test}
        }
	\caption{Comparison of six algorithms for robust black-box binary classification. Above: the convergence of test loss against iterations. Below: the convergence of test loss against queries.}
    \label{fig:rlc-testloss}
    \end{center}
    \vskip -0.2in
\end{figure*}

\textbf{Parameter setting of the robust black-box classification problem.}
For all datasets, we set $\sigma=10$ and $\theta=10$. For fair comparison, we choose the mini-batch size $b=100$ for all stochastic zeroth-order methods, the number of samples of random gradient estimators is 1 in both our accelerated algorithms and the existing algorithms. We choose step size $\eta=1/\sqrt{T}$ in ZSCG and Acc-SZOFW, $\eta=T^{-\frac{3}{4}}$ in ZO-SFW and $\eta=T^{-\frac{2}{3}}$ in Acc-SZOFW* as their theories suggested, respectively. We set $b_2=q=100$, $b_1=10000$, $\alpha_{t}=\frac{1}{t+1}$, $\theta_t=\frac{1}{(t+1)(t+2)}$, $\gamma_{t}=(1+\theta_{t}) \eta_{t}$, $\mu=d^{-\frac{1}{2}} T^{-\frac{1}{2}}$ and $\beta=d^{-1} T^{-\frac{1}{2}}$ in Acc-SZOFW as our theory suggested. We set $\alpha_{t}=\frac{1}{t+1}$, $\theta_t=\frac{1}{(t+1)(t+2)}$, $\gamma_{t}=6(1+\theta_{t}) \eta_{t}$, $\mu=d^{-\frac{1}{2}} T^{-\frac{2}{3}}$ and $\beta=d^{-1} T^{-\frac{2}{3}}$ in Acc-SZOFW* as our theory suggested. The total iteration T is set to be 1000000 and we run all the stochastic zeroth-order algorithms once.

\textbf{The convergence of test loss.}
Figure~\ref{fig:rlc-testloss} shows that the convergence of test loss against iterations and queries, which is analogous to those of train loss.

\end{appendices}

\end{onecolumn}


\end{document}